\documentclass[twoside]{amsart}
\usepackage[english]{babel}
\usepackage{graphicx}
\usepackage{mathtools}		% improves amsmath
\usepackage{mathabx}		% makes many symbols (as $\leq$) more beautiful; must be loaded after mathtools 
\usepackage{indentfirst}	% Indent first paragraph after section header
\usepackage{amsmath}
\usepackage{amssymb} 
\usepackage{amsthm}
\usepackage{thmtools}
\usepackage{enumitem}		% special itemize with custom tags and labels that work 
\usepackage[modulo,right,pagewise]{lineno} %enumera lineas lines de paginas pages
\usepackage[backref=page,colorlinks=true]{hyperref}	% Links in the pdf. Backref option: each bibliographical entry denotes where it was cited
\usepackage[capitalise]{cleveref} % Must be loaded after hyperref. It's advisable to compile the file with LUALATEX instead of PDFLATEX to avoid some annoying warning messages.
\usepackage{tikz}
\usepackage{tikz-cd}
\usepackage[font=footnotesize]{caption}
\usepackage{mathrsfs}

\newcommand{\mcal}[1]{\mathcal{#1}}
\newcommand{\wtilde}[1]{\widetilde{#1}}
\newcommand{\mbb}[1]{\mathbb{#1}}
\newcommand{\msf}[1]{\mathsf{#1}}
\newcommand{\mfrk}[1]{\mathfrak{#1}}

\newcommand{\Stab}[1]{\mathrm{Stab}{#1}}
\newcommand{\Z}{\mathbb{Z}}
\newcommand{\N}{\mathbb{N}}

\newcommand{\R}{\mathbb{R}}

\renewcommand{\epsilon}{\varepsilon}

\newcommand{\corchete}[1]{\left\{{#1}\right\}}
\newcommand{\ov}[1]{\overline{#1}}

\newcommand{\CAT}[1]{\textup{CAT}(#1)}

\newcommand{\genby}[1]{\langle #1\rangle}

\usepackage{color}
\usepackage[colorlinks=true]{hyperref}
\usepackage{fancyhdr}
\usepackage[left=3.25cm,top=3.5cm,right=3.25cm,bottom=2.5cm]{geometry}

\renewcommand*{\backref}[1]{}
\renewcommand*{\backrefalt}[4]{\quad \tiny 
  \ifcase #1 (\textbf{NOT CITED.})%
  \or    (Cited on page~#2.)%
  \else   (Cited on pages~#2.)%
  \fi}

\makeatletter

\def\MRbibitem{\@ifnextchar[\my@lbibitem\my@bibitem}

\def\mybiblabel#1#2{\@biblabel{{\hyperref{http://www.ams.org/mathscinet-getitem?mr=#1}{}{}{#2}}}}

\def\myhyperanchor#1{\Hy@raisedlink{\hyper@anchorstart{cite.#1}\hyper@anchorend}}

\def\my@lbibitem[#1]#2#3#4\par{%
  \item[\mybiblabel{#2}{#1}\myhyperanchor{#3}\hfill]#4%
  \@ifundefined{ifbackrefparscan}{}{\BR@backref{#3}}%
  \if@filesw{\let\protect\noexpand\immediate% write to aux-file
    \write\@auxout{\string\bibcite{#3}{#1}}}\fi\ignorespaces%
}

\def\my@bibitem#1#2#3\par{%
  \refstepcounter\@listctr% standard tex item counter for the generic item number
  \item[\mybiblabel{#1}{\the\value\@listctr}\myhyperanchor{#2}\hfill]#3%
  \@ifundefined{ifbackrefparscan}{}{\BR@backref{#2}}%
  \if@filesw\immediate\write\@auxout% write to aux-file
    {\string\bibcite{#2}{\the\value\@listctr}}\fi\ignorespaces%
}

\makeatother

\newtheorem{thm}{Theorem}[section]
\newtheorem{prop}[thm]{Proposition}
\newtheorem{lemma}[thm]{Lemma}
\newtheorem{coro}[thm]{Corollary}
\theoremstyle{remark}\newtheorem{rmk}[thm]{Remark}
\theoremstyle{definition}

\newtheorem{defi}[thm]{Definition}

\begin{document}

\title{\textbf{ON CUBULATED RELATIVELY HYPERBOLIC GROUPS}}

\author{\small{Eduardo Oreg\'on-Reyes}}

%\author{
%\small{Eduardo Oreg\'on-Reyes (\texttt{eoregon@berkeley.edu})}\\
%\small{Department of Mathematics}\\
%\small{University of California at Berkeley}\\
%\small{850 Evans Hall, Berkeley, CA 94720-3860, U.S.A.}
%\small{\texttt{eoregon@berkeley.edu})}
%\small{+15105707751}
%\small{Facultad de Matem\'aticas, Pontificia Universidad Cat\'olica de Chile, Santiago, Chile.}\\
%\small{Av. Vicu$\tilde{\textnormal{n}}$a Mackenna 4860 Santiago Chile}\\
%\small{E-mail: \texttt{ecoregon@mat.uc.cl}}\\
%}
\markboth{E Oreg\'on Reyes}{Nombreartículo}

\date{}
\maketitle
% \author{\centerline{\small{University of California at Berkeley, 850 Evans Hall \#3840, Berkeley, CA 94720-3840}}\\
%\small{\textit{E-mail address}: \texttt{eoregon@berkeley.edu}\hspace{20mm}\textit{Phone}: (510) 570-7751}
%}
%\linenumbers  %(for adding numeration to lines)

\begin{abstract} We show that properly and cocompactly cubulated relatively hyperbolic groups are virtually special, provided the peripheral subgroups are virtually special in a way that is compatible with the cubulation. This extends Agol's result for cubulated hyperbolic groups, and applies to a wide range of peripheral subgroups. In particular, we deduce virtual specialness for properly and cocompactly cubulated groups that are hyperbolic relative to virtually abelian groups. As another consequence, by using a theorem of Martin and Steenbock we obtain virtual specialness for groups obtained as a quotient of a free product of finitely many virtually compact special groups by a finite set of relators satisfying the classical $C'(1/6)$-small cancellation condition.
\end{abstract}

%\paragraph{Mathematics Subject Classification (2010).} 

%\hspace{-4.3mm}\small{\textbf{Keywords.}} \small{Avalanche Principle, hyperbolic space, $\CAT{-1}$ space, Schur comparison lemma.}
%\normalsize
\section{Introduction}\label{introduction}

$\CAT{0}$ cube complexes were introduced by Gromov in his seminal paper \cite{Gromov1987HypGroups} as examples of singular metric spaces with non-positive curvature. This notion has played a prominent role in the last decades, and has shown to have important connections with other aspects of topology and group theory. In particular, the class of \emph{special cube complexes} introduced by Haglund and Wise \cite{HaglundWise2008Special,HaglundWise2012Comb} was key to Agol's proof of the Virtually Haken Conjecture, a result which turned out to be a consequence of the following property of cubulated hyperbolic groups \cite[Thm.~1.1]{Agol2012VirtualHaken}:
\begin{thm}[Agol]\label{agol}
Let $G$ be a hyperbolic group acting properly and cocompactly on a $\CAT{0}$ cube complex $X$. Then $G$ has a finite index subgroup $G'$ acting freely on $X$ such that $X/G'$ is special.
\end{thm}
For a good exposition about this result see \cite{Agol2014ICM,Shepherd2019AgolsCubulations}. Virtually special groups have finite index subgroups embedding nicely into right-angled Artin groups, and so they inherit some of their properties, mainly in terms of subgroup separability. In particular, hyperbolic groups satisfying Theorem \ref{agol} are residually finite, large, linear over $\Z$, and their quasiconvex subgroups are
separable (in fact they are virtual retracts) \cite{HaglundWise2008Special}.

The assumption of hyperbolicity in Theorem \ref{agol} is in some sense necessary, since there are examples of infinite simple groups acting properly and cocompactly on products of trees \cite{BurgerMozes2000Latticesinxtrees}. The goal of this manuscript is to extend Agol's result to relatively hyperbolic groups, but for that we need to restrict our class of peripheral subgroups since any countable group is hyperbolic relative to the whole group. The following is the main result of the paper\footnote{After this paper was written, Groves and Manning \cite[Thm.~A]{GrovesManning2020SpecialRH} gave a proof of Theorem \ref{mainthm} via different methods.}.

\begin{thm}[Main Theorem]\label{mainthm} Let $G$ be a group acting properly and cocompactly on the $\CAT{0}$ cube complex $X$, and suppose $G$ is hyperbolic relative to compatible virtually special subgroups. Then there exists a finite index subgroup $G'<G$ acting freely on $X$ such that $X/G'$ is a special cube complex.
\end{thm}
The formal definition of \emph{compatible virtually special peripheral subgroups} is given in Subsection \ref{subsecvs}, and it essentially means that if $P<G$ is a peripheral subgroup, then there exists \emph{some} $P$-invariant convex subcomplex $Z\subset X$, and there is a finite index subgroup $P'<P$ such that $Z/P'$ is a compact and special cube complex. This compatibility is flexible enough so that Theorem
\ref{mainthm} applies when $G$ is cubulated and hyperbolic relative to virtually abelian subgroups.
\begin{coro}\label{virtab} If $G$ acts properly and cocompactly on a $\CAT{0}$ cube complex and is hyperbolic relative to virtually abelian subgroups, then $G$ is virtually compact special.
\end{coro}
The preceding corollary recovers some remarkable results, such as virtual specialness of fundamental groups of non-compact finite-volume hyperbolic 3-manifolds and of limit groups \cite{Wise2011Hierar} due to Wise (see also \cite{CooperFuter2019qFsfcs} and \cite{GrovesManning2021Quasiconv}).

%The compatibility assumption is also implied, for instance, when $G$ is hyperbolic relative to residually finite groups and wall stabilizers are fully relatively quasiconvex subgroups of $G$. Therefore, Theorem \ref{mainthm} also implies:
%\begin{coro}\label{fullyrfimpliesspecial}Let $G$ be hyperbolic relative to residually finite groups, and suppose it acts properly and cocompactly on the $\CAT{0}$ cube complex $X$. If the wall stabilizer $G_W<G$ is fully relatively quasiconvex for any wall $W\subset X$, then there is a finite index subgroup $G'<G$ such that $X/G'$ is special.
%\end{coro}
Another consequence of our main result depends on the combination theorem for cubulations in small cancellation theory due to Martin and Steenbock \cite[Thm.~1.1]{MartinSteenbock2017Combsmallcan}.
\begin{thm}[Martin-Steenbock]\label{martin-steenbock} Let $F$ be the free product of finitely many groups $G_1,\dots , G_r$, and assume each $G_i$ acts properly and cocompactly on the $\CAT{0}$ cube complex $X_i$. If $G$ is a quotient of $F$ by a finite set of relators that satisfies the classical $C'(1/6)$-small cancellation condition over $F$, then $G$ acts properly and cocompactly on a $\CAT{0}$ cube complex $X$.\\ Moreover, this complex is constructed in such a way that for each $i$, there is a $G_i$-equivariant combinatorial isometric embedding $\dot{X}_i \hookrightarrow X$, where $\dot{X}_i$ is the cubical barycentric subdivision of $X_i$.
\end{thm}
The ``moreover" part of the previous theorem is implicit in the construction of $X$, see \cite[Rmk.~3.43]{MartinSteenbock2017Combsmallcan} (cf.~\cite[Cor.~4.5]{JankiewiczWise2017}). The group $G$ is hyperbolic relative to the free factors $G_1, \dots, G_r$ \cite[p2~Example~(II)]{Osin2006RHBook}, and since a compact non-positively curved complex is virtually special if and only if its cubical barycentric division is virtually special (see Corollary \ref{vspecialsubdivision}), the existence of equivariant isometric embeddings $\dot{X}_i \hookrightarrow X$ imply that the cubulation $(G, X)$ satisfies the compatibility condition provided each of the cubulations $(G_i,X_i)$ is virtually special. Hence Theorem \ref{mainthm} implies the following combination result:
\begin{prop}\label{combsmallcancellation}Let $F$ be the free product of finitely many virtually compact special groups. If $G$ is a quotient of $F$ by a finite set of relators satisfying the classical $C'(1/6)$-small cancellation condition over $F$, then $G$ is also virtually compact special.
\end{prop}

By the work of Schreve \cite[Thm.~1.2]{Schreve2014Atiyahvspecial}, virtually compact special groups satisfy the strong Atiyah conjecture. Martin and Steenbock used this result, together with Theorems \ref{agol} and \ref{martin-steenbock} to deduce the strong Atiyah conjecture for quotients of free products of cubulated hyperbolic groups satisfying the $C'(1/6)$-small cancellation condition. If we use Proposition \ref{combsmallcancellation} instead, we can remove the hyperbolicity assumption on the free factors: 

\begin{coro}\label{Atiyahconj}
Let $G$ be as in the assumptions of Proposition \ref{combsmallcancellation}. Then $G$ satisfies the strong Atiyah conjecture.
\end{coro}
\paragraph{\textbf{A cubulated malnormal hierarchy.}} One of the main tools in Agol's proof of Theorem \ref{agol} is Wise's quasiconvex hierarchy theorem (see Theorem \ref{wiseQVH}), which says that a hyperbolic group is virtually compact special if and only if it can be obtained from finite groups after a finite
sequence of (virtual) amalgamations or HNN extensions over quasiconvex subgroups. Our proof of Theorem \ref{mainthm} follows that approach, and for that we require a notion of hierarchy that is compatible with relatively hyperbolic groups.
\begin{defi}\label{defiCMVH}Let $\mcal{CMVH}$ denote the smallest class of cubulated and relatively hyperbolic groups $(G, X)$ (here $G$ acts properly and cocompactly on the cubulation $X$) relative to compatible virtually special subgroups, that is closed under the following operations:
\begin{enumerate}
    \item $(\corchete{1},X) \in \mcal{CMVH}$ for any finite $\CAT{0}$ cube complex $X$.
    \item If $G$ splits as a finite graph of groups $(\Gamma, \mcal{G})$ satisfying
\begin{itemize}
\item  each edge/vertex group is convex in $(G, X)$,
\item if $v$ is a vertex of $\Gamma$ then the collection $\mcal{A}_v := \corchete{G_e \colon e \text{ an edge attached to }v}$ is \emph{relatively malnormal} in $G_v$, and
\item  if $G_v$ is a vertex group, then it has a convex core $X_v \subset X$ with $(G_v, X_v) \in \mcal{CMVH}$, \end{itemize}
then $(G, X) \in \mcal{CMVH}$.
\item If $H<G$ with $|G : H|<\infty$ and $(H, X) \in  \mcal{CMVH}$, then $(G, X) \in \mcal{CMVH}$.
\end{enumerate}
\end{defi}
The notation $\mcal{CMVH}$ is meant to be an abbreviation for ``cubulated (relatively) malnormal virtual hierarchy", and the main concepts involved in this definition are explained in Section \ref{preliminaries}. In Section \ref{proofoftheorem} we give a relative version of Wise's quasiconvex hierarchy theorem and prove:
\begin{thm}\label{CMVHimpliesspecial}If $(G, X) \in \mcal{CMVH}$ then there is a finite index subgroup $G'< G$ acting freely on
$X$ such that $X/G'$ is special.
\end{thm}
Note that the compatibility between the peripheral structure on $G$ and the splitting $(\Gamma, \mcal{G})$ in condition (2) of the definition above is only reflected in the relative quasiconvexity of edge/vertex subgroups (this is implied by convexity, see Theorem \ref{hruskaundistorted}) and the relative malnormality of edge groups in the vertex groups. This seems to be a weak assumption when we compare, for instance, with the quasiconvex, malnormal and $\mcal{P}$-fully elliptic hierarchy that requires fully relative quasiconvexity of edge/vertex groups (see e.g. \cite[Sec.~3]{Einstein2019RMSQT}), and is extremely important in our argument, particularly in Proposition \ref{splittingimpliessepa}.

The second main ingredient in Agol's proof of Theorem \ref{agol} is the ``coloring" trick to produce a hierarchy for a cubulated hyperbolic group. By adapting this technique to a cubulated relatively hyperbolic group, we prove the theorem below, which together with Theorem \ref{CMVHimpliesspecial} and Corollary \ref{vspecialsubdivision} implies Theorem
\ref{mainthm}.
\begin{thm}\label{relhypimpliesCMVH}If $(G, X)$ is cubulated and hyperbolic relative to compatible virtually special subgroups, then $(G, \dot{X}) \in \mcal{CMVH}$, where $\dot{X}$ is the cubical barycentric subdivision of $X$.
\end{thm}
While the theorem above is enough to deduce our main result, virtual compact specialness for relatively hyperbolic groups implies the existence of strong virtual hierarchies, as was proven by Einstein \cite[Thm.~1]{Einstein2019RMSQT}. In consequence, Theorem \ref{mainthm} implies:
\begin{coro}\label{fullyelliptic}Let $(G,\mcal{P})$ be a cubulated relatively hyperbolic group with compatible virtually special subgroups. Then there exists a finite index subgroup $G_0< G$ with induced relatively hyperbolic structure $(G_0,\mcal{P}_0)$ so that $G_0$ has a quasiconvex, malnormal and fully $\mcal{P}_0$-elliptic hierarchy terminating in groups isomorphic to elements of $\mcal{P}_0$.
\end{coro}
\paragraph{\textbf{Organization of the paper.}} In Section \ref{preliminaries} we present the required background, recalling the main properties about relatively hyperbolic groups, cubulations, and special cube complexes. There we also define compatibility of virtually special peripheral subgroups, deduce Corollary
\ref{virtab} and Proposition \ref{combsmallcancellation}, and state Proposition \ref{independencecore}, which shows that the compatibility of the peripherals is independent on the convex core. \\ 
\indent We introduce Dehn filling in Section \ref{dehnfilling,relativeheight,andaweakseparationtheorem}, which is the main tool in the proof of Theorem \ref{CMVHimpliesspecial}, as well as extend some properties of $H$-wide fillings in Theorems \ref{heightfilling} and \ref{doublecosetfilling}. We also introduce
Einstein's relative malnormal special quotient theorem, a fundamental result that, by means of
Wise's quasiconvex hierarchy theorem, will allow us to obtain hyperbolic virtually special fillings for relatively hyperbolic groups in $\mcal{CMVH}$. All these results will be used in Section \ref{proofoftheorem} to prove Theorem \ref{CMVHimpliesspecial}.\\
\indent The rest of the paper consists in proving Theorem \ref{relhypimpliesCMVH}, by adapting Agol's proof that cubulated hyperbolic groups are in $\mcal{QVH}$. In section \ref{constructionofthecomplex} we construct a quotient cube complex $\mcal{X}$ for $X$ with finite embedded walls, which will be used to model the desired hierarchy. We color the walls of $\mcal{X}$ in Section \ref{coloringwalls}, and in Section \ref{cubicalpolyhedra} we use this coloring to start the construction of the \emph{cubical polyhedra}, an inductively defined collection of (disconnected) cube complexes with \emph{boundary walls} that codify the description of $(G, X)$ as a cubulated group in $\mcal{CMVH}$. We study these cubical polyhedra in more detail in Section \ref{boundarywallsandgraphofgroups}, and conclude the proof of
Theorem \ref{relhypimpliesCMVH} in Section \ref{constructingVj-1}, where we show how to perform the inductive construction.\\
\indent An appendix is included at the end of the article, in which we prove Theorem \ref{doublecosetconvexsep} that is used in the proof of Theorem \ref{CMVHimpliesspecial} and says that double cosets of convex subgroups of virtually compact special groups are separable. This is done by studying the functoriality of the canonical completion, and implies Proposition \ref{independencecore}.

\section{Preliminaries}\label{preliminaries}
We start by introducing the main concepts and notation about relative hyperbolicity, $\CAT{0}$ cube complexes and special cube complexes, which will be used in the paper. We expect the reader to be familiar with hyperbolic groups \cite{Gromov1987HypGroups} and $\CAT{0}$ metric spaces \cite{BridsonHaefliger1999}.
\subsection{Relatively hyperbolic groups}\label{subsecrh} There are several equivalent definitions for relatively hyperbolic groups \cite{Hruska2010RH+RQC}, some of them valid for arbitrary countable groups. For convenience we will restrict to the finitely generated case and introduce relative hyperbolicity in terms of \emph{cusped spaces}, which will be useful in Section \ref{dehnfilling,relativeheight,andaweakseparationtheorem}. See \cite{GrovesManning2008DehnFilling} or \cite[Appendix A]{Agol2012VirtualHaken} for more details about cusped spaces associated to relatively hyperbolic groups.

Let $G$ be a group generated by a finite symmetric set $S$, and let $\mcal{P}=\corchete{P_1, \dots, P_n}$ be a finite collection of subgroups of $G$ such that $S \cap P_i$ generates $P_i$ for every $i$. 

For each $P \in \mcal{P}$, let $S_0=S \cap P \backslash \corchete{1}$, and for $n > 0$ let $S_n := S_{n-1} \cup \corchete{s_1s_2 \neq 1\colon s_1, s_2 \in S_{n-1}}$.
Given a left coset $gP$ with $g \in G$, define the 1-complex $\mcal{H}(gP)$ as the vertex set $\mcal{H}(gP)^{(0)}=
gP \times \Z_{\geq 0}$ and edges given by:
\begin{enumerate}
    \item (vertical) If $(v, n) \in \mcal{H}(gP)^{(0)}$, then an edge joins $(v, n)$ and $(v, n+1)$.
    \item (horizontal) If $(v, n) \in \mcal{H}(gP)^{(0)}$ and $s \in S_n$, there is an edge from $(v, n)$ to $(vs, n)$ (so
that if for instance $s$ has order 2, then there are two different edges between $(v, n)$ and $(vs, n)$).
\end{enumerate}
Note that there is a natural way to glue $\mcal{H}(gP)$ and the Caley graph $\Gamma(G, S)$ along $gP=gP\times \corchete{0}$.
\begin{defi}The \emph{cusped space} $X(G,\mcal{P},S)$ is obtained from $\Gamma(G, S)$ by gluing all the complexes $\mcal{H}(gP)$ for $g \in G$ and $P \in \mcal{P}$ in the previously mentioned way. The group $G$ is then \emph{hyperbolic relative to} $\mcal{P}$ if the cusped space $X(G,\mcal{P}, S)$ is Gromov hyperbolic \cite[Rmk.~A.13]{Agol2012VirtualHaken}. In that case we say that $\mcal{P}$ is a \emph{peripheral structure} on $G$.
\end{defi}

\begin{rmk}With the definition of cusped space given above, the natural isometric action of $G$
on $X(G,\mcal{P}, S)$ is \emph{free}, and the distance function on vertices coincides with the one constructed in \cite{GrovesManning2008DehnFilling}. Therefore, the coarse geometry is unchanged, and the classical results about cusped spaces also hold for this slightly different construction (see \cite[Rmk.~A.13]{Agol2012VirtualHaken}).
\end{rmk}

If $(G,\mcal{P})$ is relatively hyperbolic, then a \emph{parabolic subgroup} will be any subgroup of $G$ that can be conjugated into a member of $\mcal{P}$, and a maximal parabolic subgroup will be called \emph{peripheral subgroup}. An element of $G$ is \emph{loxodromic} if it has infinite order and the group it generates is not parabolic. 

A \emph{horoball} of $X(G,\mcal{P}, S)$ is a full subgraph on the vertices $gP \times \Z_{\geq R}$ for some $R \geq 0$ and some left coset $gP$ with $P \in \mcal{P}$. Note that peripheral subgroups correspond to stabilizers in $G$ of horoballs. 

An infinite subgroup of $G$ is non-parabolic if and only if it contains a loxodromic element. Namely, if $H<G$ is infinite and with no loxodromics, then any $H$-orbit of a point in $X(G,\mcal{P},S)$ is unbounded and $H$ must fix a point in the Gromov boundary $\partial X(G,\mcal{P},S)$ (see e.g. \cite[Thm.~6.2.3]{DasSimmonsUrbanski2017Book}). The action of $G$ on $\partial X(G,\mcal{P},S)$ is geometrically finite \cite[Sec.~3]{Hruska2010RH+RQC}, so this fixed point corresponds to a horoball by \cite[Thm.~3A~(a)]{Tukia1998UniformConv}.

\subsection{Relative Quasiconvexity} Let $H<G$ be a finitely generated subgroup of a relatively hyperbolic group $(G,\mcal{P})$ with cusped space $X=X(G,\mcal{P}, S)$, and suppose there exists a finite set $\mcal{D}$ of representatives for $H$-conjugacy classes of the infinite groups of the form $H \cap P^g$ with $g \in G$
and $P \in \mcal{P}$. Given $D \in \mcal{D}$, there is a a unique $P_D \in \mcal{P}$ and some $c_D \in G$ of shortest word-length
so that $D = H \cap P_D^{c_D}$. Also, assume that each group in $\mcal{D}$ is finitely generated, and let $X_H$ be a combinatorial cusped space for the pair $(H, \mcal{D})$ with respect to some compatible finite generating set $S'$ of $H$, in the sense that $S'$ is symmetric and $S'\cap D$ generates $D$ for each $D\in \mcal{D}$. We extend the inclusion $\iota : H \hookrightarrow G$ to an $H$-equivariant Lipschitz map $\check{\iota} : X_H^{(0)}\rightarrow X$ as follows: a vertex in a horoball of $X_H$ is a tuple $(sD, h, n)$ with $s \in H$, $D \in \mcal{D}$, $h \in sD$ and $n \in \Z_{\geq 0}$, so define its image by $\check{\iota}(sD, h, n) := (sc_DP_D, hc_D, n)$.
\begin{defi}The pair $(H, \mcal{D})$ is \emph{relatively quasiconvex} in $(G,\mcal{P})$ if the image $\check{\iota}(X_H^{(0)})\subset X$ is
$\lambda$-quasiconvex for some $\lambda$, which will be called a \emph{quasiconvexity constant} for $(H, \mcal{D})$ in $(G,\mcal{P})$. Sometimes we will omit the peripheral structures and simply say that $H$ is relatively quasiconvex in $G$.
\end{defi}
As noted in \cite[Def.~2.9]{GrovesManning2021Quasiconv}, this definition is equivalent to other notions of relative quasiconvexity existing in literature, at least in the finitely generated case \cite{Hruska2010RH+RQC}. In particular, if $H <G$ is relatively quasiconvex, then the collection $\mcal{D}$ defined above makes $(H, \mcal{D})$ into a relatively hyperbolic group \cite[Thm.~9.1]{Hruska2010RH+RQC}, and we call $\mcal{D}$ an \emph{induced relatively hyperbolic structure} (or \emph{peripheral structure}) on $H$.

The characterization of relative quasiconvexity given above will be helpful in Section \ref{dehnfilling,relativeheight,andaweakseparationtheorem}, while for the rest of the paper we will use the criterion stated below. A subgroup $H<G$ of a finitely generated group $G$ is \emph{undistorted} if it is finitely generated and the inclusion map $H \hookrightarrow G$ is a quasi-isometric embedding with respect to some (hence any) choice of finite sets of generators for $H$ and $G$. Hruska proved the following criterion for relative quasiconvexity valid for finitely generated relatively hyperbolic groups \cite[Thm.~1.5]{Hruska2010RH+RQC}.
\begin{thm}[Hruska]\label{hruskaundistorted} If $H<G$ is an undistorted subgroup, then $H$ is relatively quasiconvex in $(G,\mcal{P})$ with respect to any possible peripheral structure $\mcal{P}$ on $G$.
\end{thm}
An important class of relatively quasiconvex subgroups is given by those which are full, and in a sense this is the correct generalization to quasiconvex subgroups of hyperbolic groups.
\begin{defi}
A relatively quasiconvex subgroup $H<G$ of a relatively hyperbolic group $(G,\mcal{P})$ is \emph{fully relatively quasiconvex} if for any peripheral subgroup $P$ of $G$, the group $H \cap P$ is either
finite or finite index in $P$.
\end{defi}

\subsection{Cube complexes and cubulated groups} A \emph{cube complex} is a metric polyhedral complex in which all polyhedra are unit length Euclidean cubes. Such a complex is \emph{non-positively curved} (NPC) if its universal cover is a $\CAT{0}$ metric space when endowed with the induced length distance. There is a combinatorial description of this property, due to Gromov \cite[7, Thm.~II.5.2]{BridsonHaefliger1999}.
\begin{prop}\label{linkcond}A cube complex $X$ is NPC if and only if the link of each vertex is a flag complex.
\end{prop}
Recall that the \emph{link} of a vertex $v$ is the complex $Lk_X(v)$ with vertices the edges of $X$ incident to $v$, and in which $n$ such edges are the 1-skeleton of an $(n-1)$-face in $Lk_X(v)$ if and only if they are incident to a common $n$-cell at $v$. Also, a \emph{flag simplicial complex} is a complex determined by its 1-skeleton: for every complete subgraph of the 1-skeleton there is a simplex with 1-skeleton equal to that subgraph. For more references about the geometry of NPC cube complexes see \cite{BridsonHaefliger1999, Sageev2014CAT0Groups}.

An $n$-cube $C=[0, 1]^n \subset \R^n$ has $n$-midcubes obtained by setting one coordinate to $1/2$. Since the face of a midcube of $C$ is the midcube of a face of $C$, the set of midcubes of a cube complex $X$ has a cube complex structure, called the \emph{wall complex}. A \emph{wall} of $X$ is a connected component of the wall complex. When $X$ is a $\CAT{0}$ cube complex any wall $W$ is 2-sided and embeds as a convex subspace, so we will not make distinction between a wall and its embedded image. Also, the complement of a wall in $X$ has exactly two connected components, whose closures are called \emph{half-spaces} of $X$. Since walls are convex subcomplexes of the cubical barycentric subdivision of
$X$, they are $\CAT{0}$ cube complexes as well. For an edge $e$ in the cube complex $X$, let $W(e)$ denote the unique wall it intersects, and say that $e$ is \emph{dual} to $W(e)$.

We are interested in group actions on $\CAT{0}$ cube complexes.
\begin{defi}Suppose $G$ is a group acting properly and cocompactly by cubical isometries on the $\CAT{0}$ cube complex $X$ with $\CAT{0}$ distance $d$. In that case we say that $(G, X)$ is a \emph{cubulated group}, and
that $X$ is a \emph{cubulation} of $G$.
\end{defi}

\begin{rmk}Although in this paper we will restrict to proper and cocompact actions, when $G$ is relatively hyperbolic we can consider cubulations that are non necessarily cocompact, but cosparse \cite{HruskaWise2014FinCub, SageevWise2015Cores}. Also, there has been some progress in understanding non-proper actions of hyperbolic groups \cite{Groves2018HyperbolicImproperly}. Recently, Einstein and Groves introduced relatively geometric cubulations of relatively hyperbolic groups \cite{EinsteinGroves2020RelCub}.
\end{rmk}
If $(G, X)$ is a cubulated group, then $X$ is finite-dimensional, locally finite, and $G$ is finitely generated and quasi-isometric to $X$ \cite[Prop. 4.2]{Shepherd2019AgolsCubulations}. Cocompactness and properness also imply the following lemma (cf.~\cite[p1052]{Agol2012VirtualHaken}):
\begin{lemma}\label{finiteclasses} If $(G, X)$ is a cubulated group, then:
\begin{enumerate}
    \item There are only finitely many conjugacy classes of torsion elements in $G$.
    \item If $Y \subset X$ is a non-empty subset and $H <  G$ preserves $Y$ and acts cocompactly on it, then for any $R > 0$ the set of double cosets
  \begin{equation*}
  A_{Y,H,R} := \corchete{HgH \colon g \in G, d(Y, gY) \leq  R}
  \end{equation*}
is finite.
\end{enumerate}
\end{lemma}
\begin{proof}Part (1) holds for arbitrary proper and cocompact actions by isometries on $\CAT{0}$ spaces
\cite[Cor.~II.2.8~(2)]{BridsonHaefliger1999}.\\
For part (2), let $g \in G$ be such that $d(gY, Y) \leq R$, and let $\gamma$ be a geodesic of length $\leq R+1$ joining $Y$ and $gY$ . Consider $D \subset Y$ a compact subset with $H \cdot D=Y$ and $w_1, w_2 \in H$ such that $\gamma$ intersects $w_1D$ in $Y$ and $gw_2D$ in $gY$. Then $d(w_1^{-1}gw_2D, D) \leq R+1$, and so by local finiteness and properness of the action, there is a finite set $F \subset G$ such that $w_1^{-1}gw_2 \in F$. Therefore $g \in w_1F w_2^{-1} \subset HFH$.
\end{proof}
Let $(G, X)$ be a cubulated group, and for $W$ a wall of $X$, let $G_W< G$ denote its set-wise stabilizer in $G$. The group $G_W$ acts properly and cocompactly on $W$ \cite[Rmk.~2.3]{Shepherd2019AgolsCubulations}, and since $W$ is convex in $X$, $(G_W , W)$ is a cubulated group and $G_W$ is undistorted in $G$.
\begin{defi} If $(G, X)$ is a cubulated group, we say that a subgroup $H<G$ is \emph{convex} in $(G, X)$ if it acts cocompactly on a convex subcomplex $Z \subset X$. Such a subcomplex $Z$ will be
called a \emph{convex core} for $H$.
\end{defi}
\begin{defi} If $Y$ is a NPC cube complex and $S \subset Y$ is any subset, then the \emph{cubical neighborhood} of $S$ is the subcomplex $\mcal{N}(S) \subset Y$ consisting of the cubes of $Y$ intersecting $S$.
\end{defi}
If $(G, X)$ is a cubulated group and $W$ is a wall of $X$, then $\mcal{N} (W)$ is a convex subcomplex of $X$ \cite[Lem.~13.4]{HaglundWise2008Special}, and hence it is a convex core for $G_W$. If in addition $G$ is relatively hyperbolic, and since $G$ quasi-isometric to $X$, by Theorem \ref{hruskaundistorted} every convex subgroup of $G$ is relatively quasiconvex. As a partial converse, Sageev and Wise proved the following \cite[Thm.~1.1]{SageevWise2015Cores}:

\begin{thm}[Sageev-Wise]\label{sageevwise} If $(G, X)$ is a cubulated group with $G$ relatively hyperbolic and $H<G$ is fully relatively quasiconvex, then for any compact subset $B \subset X$ there exists a convex core $Z \subset X$ for $H$ that contains $B$. In particular, any peripheral subgroup of $G$ is convex.
\end{thm}
\begin{rmk}\label{wallsubcomplex} When $(G, X)$ is cubulated with $G$ relatively hyperbolic and $P < G$ is a peripheral subgroup with convex core $Z \subset X$, then for any wall $W \subset X$ with $W \cap Z \neq \emptyset$ we have $P_{W\cap Z} = P \cap G_W$. The inclusion $P \cap G_W \subset P_{W\cap Z}$ is evident, and if $e$ is an edge of $Z$ dual to $W$, then $e$ is also dual to the wall $W \cap Z$ of $Z$. Therefore, for any $g \in P$ with $g(W \cap Z) = W \cap Z$, the edge $ge$ is dual to $W\cap Z \subset W$, and hence $gW = gW(e) = W(ge) = W$, implying $P_{W\cap Z} = P\cap G_W$.
\end{rmk}
From the previous remark we see that the intersection of a peripheral subgroup with a wall stabilizer is a convex subgroup. In general, we have the following lemma:
\begin{lemma}\label{intersectionconvex} If $(G, X)$ is a cubulated group and $H_1, H_2<G$ are convex subgroups with convex cores $Y_1, Y_2 \subset X$ respectively and such that $Y_1 \cap Y_2 \neq \emptyset$, then $H_1 \cap H_2$ is a convex subgroup with convex core $Y_1 \cap Y_2$.
\end{lemma}
\begin{proof}Let $x_0$ be a vertex in the convex subcomplex $Y_1 \cap Y_2$, and let $R \geq 0$ be such that if $D \subset X^{(0)}$
is the combinatorial $R$-ball around $x_0$, then $Y_1 \subset H_1\cdot D$ and $Y_2 \subset H_2 \cdot D$. Since $X^{(0)}$
is a proper metric space with the combinatorial distance and $G$ acts properly on $X^{(0)}$, it can be proven that for any $K \geq 0$ there exists some $L = L(K)$ such that 
\begin{equation*}N_K(H_1 \cdot x_0) \cap N_K(H_2 \cdot x_0) \subset N_{L}((H_1 \cap H_2) \cdot x_0),
\end{equation*}
with $N_r(S)$ denoting the combinatorial $r$-neighborhood of $S \subset X^{(0)}$ (cf.~\cite[Lem.~4.2]{MartinezPedroza2009Combqc}). In particular we have
\begin{equation*}
 Y_1 \cap Y_2 \subset H_1 \cdot D \cap H_2 \cdot D \subset N_R(H_1 \cdot x_0) \cap N_R(H_2 \cdot x_0) \subset N_{L(R)}((H_1 \cap H_2)\cdot  x_0),   
\end{equation*}
implying that $H_1 \cap H_2$ acts cocompactly on $Y_1 \cap Y_2$ by the local finiteness of $X$.
\end{proof}
The lemma above requires two convex cores to intersect, which can be always assumed after enlarging the convex cores.
\begin{lemma}\label{enlarging}If $(G, X)$ is a cubulated group and $H<(G, X)$ is a convex subgroup, then for any compact set $B \subset X$ there is a convex core $Y \subset X$ for $H$ such that $B \subset Y$ . In addition, for any two convex cores $Y_1, Y_2 \subset X$ for the subgroup $H$, there exists a third convex core $Y_3 \subset X$ containing both $Y_1$ and $Y_2$.
\end{lemma}
\begin{proof}Let $Y'$ be any convex core for $H$. Since $X$ is finite-dimensional, by \cite[Lem.~13.15]{HaglundWise2008Special} the
iterated cubical neighborhoods $\mcal{N}^k(Y')=\mcal{N}(\mcal{N}^{k-1}(Y'))$ are convex cores for $H$ for all $k \geq 1$, and we can choose $k$ such that $Y := \mcal{N}^k(Y')$ contains $B$. The second statement follows by considering a compact set $B_i \subset Y_i$ such that $H \cdot B_i \supset Y_i$ for each $i= 1, 2$, and then finding a convex core $Y_3$ containing $B_1 \cup B_2$.
\end{proof}
We need one more result, which is key in the proof of Proposition \ref{acylindricity}. We use the notation $[p, q]$ for the geodesic segment joining the points $p$ and $q$.
\begin{prop}\label{loxodromics}If $(G,\mcal{P})$ is relatively hyperbolic and cubulated by $X$, then there exists some
$\delta \geq 0$ such that if $h \in G$ is loxodromic and preserves two axes $\gamma_1, \gamma_2 \subset X$, then $d(\gamma_1, \gamma_2) \leq \delta$.
\end{prop}
\begin{proof}Let $x \in X$ be a base-point. The map $G \xrightarrow{\mu} X$, $g \mapsto gx$ is a quasi-isometry for $X$ considered with the $\CAT{0}$ metric, so there
exists some $\delta'\geq 0$ such that for any geodesic triangle $\Delta \subset X$ with vertices $a, b, c$ there is some peripheral left coset $Q_{\Delta}= g_{\Delta}P_{\Delta}$ with $g_{\Delta} \in G$ and $P_{\Delta} \in P$, such that for any point $p \in \Delta$, either:
\begin{itemize}
    \item[($i$)] $p$ lies in the $\delta'$-neighborhood of the union of the sides of $\Delta$ not containing $p$, or
\item[($ii$)] $p \in N_{\delta'}(Q_{\Delta} \cdot x)$
\end{itemize}
(see e.g. \cite[Thm.~4.1 \& Prop.~4.2]{SageevWise2015Cores} or \cite[ Sec.~8.1.3]{DrutuSapir2005Treegraded}). In addition, for such a $\delta'$ there exists some $\lambda \geq 0$ so that if $gP$ and $g'P'$ are distinct peripheral left cosets, then
\begin{equation}\label{diameter}
\mathrm{diam}(N_{\delta'}(gP \cdot x) \cap N_{\delta'}(g'P'\cdot x))<   \lambda,
\end{equation}
where $N_R(M)$ denotes the $R$-neighborhood of $M$, see \cite[10, Cor.~2.3]{Einstein2019RMSQT}.

Let $\delta := 4\delta'$, and suppose by contradiction that $h \in G$ is loxodromic and preserves two axes $\gamma_1, \gamma_2 \subset X$ with $r=d(\gamma_1, \gamma_2) > \delta$. We claim that there is some peripheral left coset $gP$ such that $\gamma_1 \subset N_{\delta'}(gP\cdot x)$.

By \cite[II.2.13]{BridsonHaefliger1999}, $\gamma_1$ and $\gamma_2$ are asymptotic and bound a flat strip isometric to $\R  \times [0, r]$. Let
$a \in \gamma_1$, and let $b$ be its closest point projection into $\gamma_2$. Choose an isometry $\alpha : \R \rightarrow \gamma_1$ sending $0$ to $a$, and for $\eta > r$ consider the geodesic triangle $\Delta_{\eta}$ with vertices $a, b$, and $\alpha(\eta)$. After using some Euclidean trigonometry we can prove that the segment $[\alpha(\sqrt{2}\delta'),\alpha(\eta/2)]$ lies outside the $\delta'$-neighborhood of $[a, b]\cup  [\alpha(\eta), b]$, so by condition ($ii$) above there exists some peripheral left coset
$Q_\eta = g_\eta P_\eta$ with $[\alpha(\sqrt{2}\delta'),\alpha(\eta/2)] \subset N_{\delta'}(Q_\eta \cdot  x)$. Also, by \eqref{diameter} we have $Q_{\eta}= Q_{\eta'}$ for any
$\eta \geq \eta':=\max(2\lambda+2\sqrt{2}\delta',r)$, implying $\alpha([\sqrt{2}\delta',+\infty))\subset N_{\delta'}(Q_{\eta'}\cdot x)$. In fact, by a completely analogous argument we can prove $\gamma_1\subset N_{\delta'}(Q_{\eta'}\cdot x)$, and so the claim follows with $gP := Q_{\eta'}$.

Now, let $L=d(hx, \gamma_1)$ which equals $d(h^nx, \gamma_1)$ for any $n \in \Z$. By our previous claim we have $d(h^nx, gP \cdot x) \leq L + \delta'$ for all $n$, so by means of the quasi-isometry $\mu$ we can find a constant $C \geq 0$ such that
\begin{equation*}
d_S(h^n, gP) \leq C
\end{equation*}
for any $n \in \Z$, for $d_S$ the word-metric with respect to some finite generating set $S$ of $G$. This is our desired contradiction, since in that case the infinite cyclic group generated by $h$ would be bounded in $G$ for the word-metric $d_{(S\cup \bigcup{\mcal{P}})}$, contradicting that $h$ is loxodromic \cite[Lem.~8.3 (1)]{HruskaWise2009Packing}.\qedhere
\tikzset{every picture/.style={line width=0.75pt}} %set default line width to 0.75pt   
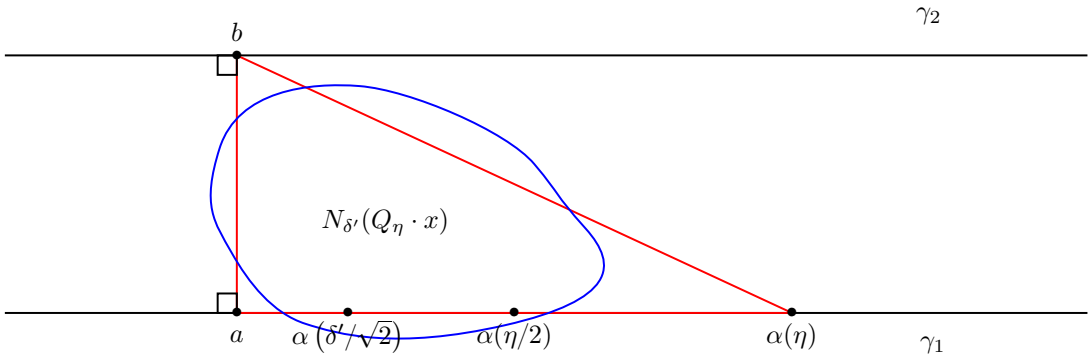
\begin{figure}[hbt]
\begin{tikzpicture}[x=0.75pt,y=0.75pt,yscale=-1,xscale=1]
%uncomment if require: \path (0,395); %set diagram left start at 0, and has height of 395

%Straight Lines [id:da05473030371815013] 
\draw    (104,230) -- (650,230) ;
%Straight Lines [id:da3838937963892197] 
\draw    (104,100) -- (650,100) ;
%Straight Lines [id:da02697098330885983] 
\draw [color={red}  ,draw opacity=1.5 ]   (221,100) -- (221,230) ;
%Straight Lines [id:da01231929864418091] 
\draw [color={red}  ,draw opacity=1.5 ]   (221,100) -- (501,230) ;
%Straight Lines [id:da9142160456055379] 
\draw [color={red}  ,draw opacity=1.5 ]   (221,230) -- (501,230) ;
%Shape: Polygon Curved [id:ds23820521695988073] 
\draw  [color={blue}  ,draw opacity=1.5 ] (283,115.5) .. controls (308.33,117.5) and (357,138.17) .. (371,154.83) .. controls (385,171.5) and (380.33,168.83) .. (397,187.5) .. controls (413.67,206.17) and (411.67,223.5) .. (356.33,236.17) .. controls (301,248.83) and (272.33,240.83) .. (255.67,235.5) .. controls (239,230.17) and (227,214.83) .. (218.33,199.5) .. controls (209.67,184.17) and (203,176.17) .. (212.33,146.83) .. controls (221.67,117.5) and (257.67,113.5) .. (283,115.5) -- cycle ;
%Shape: Rectangle [id:dp2648467637178056] 
\draw   (211.33,100) -- (221,100) -- (221,109.83) -- (211.33,109.83) -- cycle ;
%Shape: Rectangle [id:dp19724254822816878] 
\draw   (211.33,220.17) -- (221,220.17) -- (221,230) -- (211.33,230) -- cycle ;

% Text Node
\draw (221,88) node    {$b$};
% Text Node
\draw (570,80) node    {$\gamma _{2}$};
% Text Node
\draw (361,242) node    {$\alpha ( \eta /2)$};
% Text Node
\draw (277,242) node    {$\alpha \left( \sqrt{2}\delta '\right)$};
% Text Node
\draw (501,242) node    {$\alpha ( \eta )$};
% Text Node
\draw (221,242) node    {$a$};
% Text Node
\draw (296,184) node    {$N_{\delta'}( Q_{\eta } \cdot x)$};
% Text Node
\draw (572,246) node    {$\gamma _{1}$};
% Text Node
\draw (221,100) node    {\small{$\bullet$}};
\draw (221,230) node    {\small{$\bullet$}};
\draw (501,230) node    {\small{$\bullet$}};
\draw (277,230) node    {\small{$\bullet$}};
\draw (361,230) node    {\small{$\bullet$}};
\end{tikzpicture}
	\caption{Proof of Proposition \ref{loxodromics}.}
\end{figure}
\end{proof}
\subsection{Virtually special groups}\label{subsecvs} We will not work directly with the definition of special cube complex \cite{HaglundWise2008Special}. Instead, we present some of their main properties and some criteria for virtual specialness in the case of cubulated hyperbolic groups. We will extend some of these results in Section \ref{proofoftheorem} when we prove Theorem \ref{CMVHimpliesspecial}.
\begin{defi} We say that a cubulated group $(G, X)$ is \emph{special} if $G$ acts freely on $X$ and the (compact) quotient $X/G$ is a special cube complex, and that $(G, X)$ is \emph{virtually special} if there is a finite index subgroup $G'< G$ such that $(G, X)$ is special.
\end{defi}
By abuse of notation, sometimes we will simply say that $G$ is (virtually) special without mentioning the cubulation $X$, and it will be implicit that the quotient $X/G$ is compact.
\begin{defi}A subset $S \subset G$ of a group $G$ is \emph{separable} in $G$ (or simply \emph{separable} if the
ambient group $G$ is understood) if for any $a \in G\backslash S$ there exists a finite quotient of $G$ in which the image of $a$ does not lie in the image of $S$. In particular, a subgroup $H<G$ is separable if it is the intersection of finite index subgroups of $G$. The group $G$ is \emph{residually finite} if $\corchete{1}$ is separable.
\end{defi}
One of the main properties of virtually special groups is the following result due to Haglund and Wise \cite[Thm.~1.3, Cor.~7.9 \& Prop.~13.7]{HaglundWise2008Special}:
\begin{thm}[Haglund-Wise]\label{haglundwiseseparability} If $G$ is hyperbolic and virtually special then every quasiconvex
subgroup of $G$ is separable.\\
In general, if $(G, X)$ is a virtually special group then every convex subgroup of $(G, X)$ is separable.
\end{thm}
As a partial converse, they also proved the following characterization of virtual specialness \cite[
Thm.~9.19]{HaglundWise2008Special}:
\begin{thm}[Haglund-Wise]\label{haglundwisedoublecoset} The cubulated group $(G, X)$ is virtually special if and only if:
\begin{itemize}
    \item[($i$)] $G_W$ is separable for every wall $W \subset X$, and
    \item[($ii$)] the double coset $G_{W_1}G_{W_2}$
is separable for any pair $W_1, W_2$ of intersecting walls of $X$.
\end{itemize}
\end{thm}
The second characterization of virtual specialness is due to Wise \cite[Thm~13.3]{Wise2011Hierar}, and is in terms of the \emph{quasiconvex virtual hierarchy} $\mcal{QVH}$ defined below (see also \cite[Thm.~10.2]{AgolGrovesManning2016MSQT}).
\begin{defi}Let $\mcal{QVH}$ be the smallest class of hyperbolic groups closed under the operations:
\begin{enumerate}
    \item $\corchete{1} \in \mcal{QVH}$.
\item If $G= A\ast_B C$ with $A, C \in \mcal{QVH}$ and such that $B$ is quasiconvex in $G$, then $G \in \mcal{QVH}$.
\item If $G = A\ast_B$ with $A \in \mcal{QVH}$ and such that $B$ is quasiconvex in $G$, then $G \in \mcal{QVH}$.
\item If $H< G$ with $|G : H|<\infty$ and $H \in \mcal{QVH}$, then $G \in \mcal{QVH}$.
\end{enumerate}
\end{defi}
\begin{thm}[Wise]\label{wiseQVH} A hyperbolic group is virtually special if and only if it is in $\mcal{QVH}$.
\end{thm}
As mentioned in the introduction, we will require the following compatibility condition for cubulated relatively hyperbolic groups:
\begin{defi}\label{compatiblevsperi} A cubulated and relatively hyperbolic group $(G, X)$ with peripheral structure $\mcal{P}$ is said to be \emph{hyperbolic relative to compatible virtually
special subgroups} if for any peripheral subgroup $P < G$ there exists some convex core $Z \subset X$ for $P$ such that the cubulated group $(P, Z)$ is virtually special.
\end{defi}
From Theorems \ref{sageevwise} and \ref{haglundwiseseparability}, it follows that if $(G, X)$ is relatively hyperbolic and virtually special, then $(G, X)$ is hyperbolic relative to compatible virtually special subgroups. By Theorem \ref{martin-steenbock}, this compatibility also holds under the assumptions in Proposition \ref{combsmallcancellation}. The same applies to Corollary \ref{virtab} since any subgroup or double coset of a finitely generated virtually abelian group is separable, so Theorem \ref{haglundwisedoublecoset} gives us that any cubulated group that is hyperbolic relative to virtually abelian subgroups satisfies Definition \ref{compatiblevsperi}.

In general, the existence of a virtually special convex core cubulation for a convex subgroup implies that any other convex core gives a virtually special cubulation. This follows from the next proposition, which will be proven in the Appendix as a consequence of Theorem \ref{doublecosetconvexsep}.

\begin{prop}\label{independencecore} Let $(G, X)$ be a cubulated group and let $H<G$ be a convex subgroup with convex core $Y \subset X$. If $(H, Y)$ is virtually special, then $(H, Y')$ is virtually special for any other convex core $Y'\subset X$ for $H$.
\end{prop}
We end this subsection by proving that all the relevant subgroups mentioned in points (2) and (3) of Definition \ref{defiCMVH} are hyperbolic relative to compatible virtually special subgroups.
\begin{lemma}\label{consistentcompatible} If $(G, X)$ is a cubulated group such that $(G,\mcal{P})$ is hyperbolic relative to compatible virtually special subgroups, and $H < G$ is a convex subgroup with convex core $Y \subset X$, then $(H, Y)$ is also hyperbolic relative to compatible virtually special subgroups, when endowed with its induced peripheral structure.
\end{lemma}
\begin{proof} By Theorem \ref{hruskaundistorted}, $H$ is relatively quasiconvex in $G$. Let $P< G$ be a peripheral subgroup such that $H \cap P$ is infinite, and let $U \subset Y$ be a convex core for $H \cap P$. We claim that $(H \cap P, U)$ is virtually special. If $Z \subset X$ is any convex core for $P < G$, by Proposition \ref{independencecore} the cubulation $(P, Z)$ is virtually special, and so by the characterization of special cube complexes given in \cite[Def.~2.4]{HaglundWise2012Comb} it is enough to show that there is a convex core $Z$ with $U \subset Z$. But $U/(H \cap P)$ is compact, so there is a compact subset $B \subset Z$ such that $(H \cap P) \cdot B = U$. Therefore by Theorem \ref{sageevwise} we can choose $Z$ containing $B$, and hence containing $U$. 
\end{proof}

\section{Dehn filling, relative height, and a weak separation theorem}\label{dehnfilling,relativeheight,andaweakseparationtheorem}
In this section we introduce Dehn filling, the malnormal special quotient theorem, and prove some results which will be used in the remaining sections. The main results of this section are Theorems \ref{heightfilling} and \ref{doublecosetfilling}.
\subsection{Wide Dehn fillings} Dehn filling was introduced independently by Groves and Manning \cite{GrovesManning2008DehnFilling} and Osin \cite{Osin2007Perfilling} as a group theoretical generalization of the corresponding concept to cusped hyperbolic 3-manifolds.
\begin{defi} Let $(G,\mcal{P}=\corchete{P_1,\dots , P_n})$ be a relatively hyperbolic group and consider subgroups $N_i \trianglelefteq P_i$. The \emph{group theoretic Dehn filling} of $G$ (or simply the \emph{filling}) is the quotient map
\begin{equation*}
    \phi : G \rightarrow G(N_1, \dots , N_n) =G(\mcal{N}):= G/\genby{\genby{\bigcup{N_i}}}_G,
\end{equation*}
where $\mcal{N}=\corchete{N_1,\dots ,N_n}$ is the collection of \emph{filling kernels} of $\phi$. Let $\ov{\mcal{P}}$ denote the set of images of the subgroups $P_i$ under $\phi$.
\end{defi}
If $H<(G,\mcal{P})$ is a relatively quasiconvex subgroup, we need further conditions on a filling to guarantee good properties for the image of $H$. In \cite{GrovesManning2021Quasiconv}, Groves and Manning introduced $H$-\emph{wide fillings}, as a generalization of $H$-fillings, but with enough flexibility to behave nicely even when $H$ is not necessarily full. Let $\mcal{D}$ be an induced peripheral structure on $H$, such that every $D \in \mcal{D}$ is of the form $D = H \cap P_{i_D}^{c_D}$ for some $P_{i_D} \in \mcal{P}$ and some $c_D \in G$ of shortest length (with respect to a fixed compatible generating set of $G$).
\begin{defi} If $S\subset \bigcup{\mcal{P}}\backslash \corchete{1}$ is a finite set, then a filling $G \rightarrow G(N_1,\dots, N_n)$ is $(H, S)$-\emph{wide} if for any $D \in \mcal{D}$, and for any $d \in D$ and $w \in S \cap P_{i_D}$ , we have $c_Dwc_D^{-1}
\in D$ whenever $dc_Dwc_D^{-1}
 \in N_{i_D}^{c_D}$.
\end{defi}
More generally, if $\mcal{H}=\corchete{(H_1, \mcal{D}_1),\dots,(H_k, \mcal{D}_k)}$ is a collection of relatively quasiconvex subgroups of $G$, then a filling is $(\mcal{H}, S)$-\emph{wide} if it is $(H_j, S)$-wide for each $1 \leq j \leq k$.

Given a collection $\corchete{N_i \trianglelefteq P_i}_i$
of filling kernels of $(G,\mcal{P})$, the \emph{induced filling kernels} of $(H,\mcal{D})$ are the groups of the collection $\mcal{N}_H=\corchete{K_D := D \cap N^{c_D}_{i_D}}_{D\in \mcal{D}}$. These groups define the \emph{induced filling}
\begin{equation*}
    \phi_H : (H,\mcal{D}) \rightarrow (H(\mcal{N}_H), \ov{\mcal{D}}),
\end{equation*}
with $\ov{\mcal{D}}$ being the set of images of the elements of $\mcal{D}$ in $H(\mcal{N}_H) := H/\genby{\genby{\bigcup_{D}K_D}}_H$. Note that
there is a natural map from $H(\mcal{N}_H)$ into $\ov{G}=G(N_1,\dots, N_n)$.

If $G \rightarrow \ov{G}$ is a Dehn filling of $(G,\mcal{P})$ with kernel $K$, and $X$ is a cusped space for $(G,\mcal{P})$ defined in Subsection \ref{subsecrh}, then a combinatorial cusped space for $(\ov{G},\ov{\mcal{P}})$ is obtained from $\ov{X}=X/K$ by removing self-loops. This removing process does not affect the metric on the $0$-skeleton, so we will ignore this ambiguity and simply set $\ov{X}=X/K$ \cite[p4]{GrovesManning2021Quasiconv}.

\begin{defi} A property $\msf{P}$ holds for \emph{all sufficiently long and $H$-wide fillings} if there is a finite set $S\subset 
\bigcup{\mcal{P}}\backslash \corchete{1}$ so that $\msf{P}$ holds for any $(H,S)$-wide filling $G \rightarrow G(N_1,\dots, N_n)$ with $S \cap (\bigcup_{i}{N_i})= \emptyset$. In general, if $\mcal{H}$ is a collection of relatively quasiconvex subgroups of $G$, then $\msf{P}$ holds for \emph{all sufficiently long and $\mcal{H}$-wide fillings} if there is a finite set $S$ so that $\msf{P}$ holds for any $(\mcal{H}, S)$-wide filling with $S \cap (\bigcup_{i}{N_i})= \emptyset$.
\end{defi}
The next result summarizes the main required properties about wide Dehn filling of relatively hyperbolic groups \cite[Sec.~7]{AgolGrovesManning2016MSQT}, \cite[Sec.~3 \& 4]{GrovesManning2021Quasiconv}.
\begin{thm}\label{dehnfilling}Let $(G,\mcal{P})$ be a relatively hyperbolic group and let $X$ be a cusped space for $(G,\mcal{P})$.
Let $H<G$ be a relatively quasiconvex subgroup with quasiconvexity constant $\lambda$ with respect to $X$, and let $A \subset G$ be a finite set. Then there exist positive numbers $\delta'=\delta'(\delta)$ and $\lambda'=\lambda'(\lambda, \delta)$ such that for all sufficiently long and $H$-wide fillings $\phi : G \rightarrow \ov{G} := G(N_1,\dots, N_n)$:
\begin{enumerate}
    \item $(\ov{G},\ov{\mcal{P}})$ is relatively hyperbolic and $\ov{X}$ is $\delta'$-hyperbolic.
\item $\phi(A) \cap \phi(H) = \phi(A \cap H)$.
\item  $\ov{H} := \phi(H)$ is relatively quasiconvex in $(\ov{G},\ov{\mcal{P}})$, and $\lambda'$ is a quasiconvexity constant for $\ov{H}$ w.r.t. $\ov{X}$. In addition, if $H$ is fully relatively quasiconvex, then $\ov{H}$ is fully relatively quasiconvex.
\item  $\ov{H}$ is canonically isomorphic to the induced filling $H(\mcal{N}_H)$.
\end{enumerate}
\end{thm}
\begin{defi}Let $(H,\mcal{D})$ be a relatively quasiconvex subgroup of the relatively hyperbolic group $(G,\mcal{P})$. We say that $H$ is \emph{peripherally separable} in $(G,\mcal{P})$ if $D$ is separable in $P_{i_D}^{c_D}$ for every $D \in \mcal{D}$. We say that $H$ is \emph{strongly peripherally separable} if $D'$ is separable in $P^{c_D}_{i_D}$ for any finite index subgroup $D'< D$ and for any $D \in \mcal{D}$.
\end{defi}
The existence of wide fillings is guaranteed by the following lemma \cite[Lem.~5.2]{GrovesManning2021Quasiconv}.
\begin{lemma}\label{existencewide}Let $(G,\mcal{P})$ be relatively hyperbolic and let $\mcal{H}=\corchete{(H_1, \mcal{D}_1),\dots ,(H_k, \mcal{D}_k)}$ be a finite collection of relatively quasiconvex and peripherally separable subgroups of $G$. Then for any finite set $S \subset
\bigcup{\mcal{P}}\backslash \corchete{1}$ there exist finite index subgroups $\dot{N}_i \trianglelefteq P_i$ such that any filling $G \rightarrow G(N_1,\dots,N_n)$ with $N_i< \dot{N}_i$ is $(\mcal{H}, S)$-wide.
\end{lemma}
In our situation, strong peripheral separability follows from convexity, as is shown in the following lemma:
\begin{lemma}\label{compimpliesstrongly}Let $(G,X)$ be a cubulated group that is hyperbolic relative to compatible virtually special subgroups. Then any convex subgroup of $(G,X)$ is strongly peripherally separable.
\end{lemma}
\begin{proof}Let $H<G$ be a convex subgroup with convex core $Y \subset X$, and let $P< G$ be a peripheral subgroup with $H \cap P$ infinite. We claim that if $D' <  H\cap P$ is a finite index subgroup, then it is separable in $P$. To prove the claim, let $Z \subset X$ be a convex core for $P$, for which we can assume to have non-trivial intersection with $Y$ by Theorem \ref{sageevwise}. In this case, Lemma \ref{intersectionconvex} implies that $D'$
is a convex subgroup of $(P, Z)$. But the cubulated group $(P, Z)$ is virtually special by Proposition \ref{independencecore}, and so the claim follows from Theorem \ref{haglundwiseseparability}.
\end{proof}
\subsection{Relative Height, and $R$-Hulls} The relative height of a relatively quasiconvex subgroup of a relatively hyperbolic group was introduced by Hruska and Wise in \cite[Def.~8.1]{HruskaWise2009Packing} as the appropriate analog of the height for quasiconvex subgroups of hyperbolic groups \cite{GitikMitraRipsSageev1998Width}. In particular, they proved that the relative height is always finite. In this subsection we introduce the corresponding version for a finite collection of relatively quasiconvex subgroups, we show that it is finite, and prove that it is non-increasing under Dehn filling, extending the results in \cite[Sec.~7]{GrovesManning2021Quasiconv} (c.f. \cite[A.3]{Agol2012VirtualHaken}).
\begin{defi} Let $(G,\mcal{P})$ be a relatively hyperbolic group and consider a set $\mcal{H}=\corchete{H_1,\dots , H_k}$ of \emph{distinct} relatively quasiconvex subgroups of $G$. The \emph{relative height} of $\mcal{H}$ in $G$ is the maximum number $n \geq 0$ so that there are distinct left cosets $\corchete{g_1H_{\sigma(1)},\dots , g_nH_{\sigma(n)}}$ with
$\bigcap_{i=1}^{n}{H^{g_i}_{\sigma(i)}}$ containing a loxodromic element of $G$.
\end{defi}
The following is the main result of the subsection.
\begin{thm}\label{heightfilling}
For all sufficiently long and $\mcal{H}$-wide fillings, the groups in $\ov{\mcal{H}}=\corchete{\ov{H}_1\dots, \ov{H}_k}$ are all distinct and the relative height of $\ov{\mcal{H}}$ in $(\ov{G},\ov{\mcal{P}})$ is at most the relative height of $\mcal{H}$ in $(G,\mcal{P})$.
\end{thm}
In the case where $\mcal{H}$ consists of a single subgroup $H$, Theorem \ref{heightfilling} reduces to \cite[Thm.~7.12]{AgolGrovesManning2016MSQT}
for $H$-fillings and to \cite[Thm.~7.15]{GrovesManning2021Quasiconv} for general $H$-wide fillings.

\begin{defi}A collection $\corchete{H_1,\dots,H_k}$ of distinct subgroups of $(G,\mcal{P})$ is \emph{relatively malnormal} if whenever $H_i \cap H_j^{g}$ contains a loxodromic element, then $i = j$ and $g \in H_i$. The collection is \emph{almost malnormal} if $H_i \cap H^g_j$ is finite unless $i = j$ and $g \in H_i$.
\end{defi}
Note that relative malnormality is equivalent to relative height at most 1, and that relative malnormality coincides with almost malnormality when $\mcal{P}$ is empty or consists of finite groups (in which case $G$ is hyperbolic). These observations together with Theorem \ref{heightfilling} imply the following corollary. Recall that a filling $G \rightarrow G(N_1,\dots,N_n)$ is \emph{peripherally finite} if $N_i \trianglelefteq  P_i$
is finite index for each $1 \leq i \leq n$.
\begin{coro}\label{almostmalnormalfilling} Let $\mcal{H}=\corchete{H_1,\dots,H_k}$ be a relatively malnormal collection of relatively quasiconvex subgroups of $(G,\mcal{P})$. Then for all sufficiently long and $\mcal{H}$-wide peripherally finite fillings, the collection $\corchete{\ov{H}_1,\dots , \ov{H}_k}$ is almost malnormal in $\ov{G}$.
\end{coro}
The proof of Theorem \ref{heightfilling} will take us the rest of the subsection, in which we basically refine the techniques from \cite{AgolGrovesManning2016MSQT} and \cite{GrovesManning2021Quasiconv} about $R$-hulls over cusped spaces and relative multiplicity.

Let $X$ be a cusped space for $(G,\mcal{P})$, and let $\delta \geq 0$ be given by Theorem \ref{dehnfilling} (1), such that $X$ and $\ov{X}$ are $\delta$-hyperbolic for all sufficiently long fillings. The \emph{depth} of a vertex of $X$ is its distance
to the Caley graph $\Gamma(G, S)$. Note that the action of $G$ on $X$ is depth-preserving.
\begin{defi}[c.f.~Def.~7.3,\cite{GrovesManning2021Quasiconv}]\label{defRhull} Let $\wtilde{\ast}=1 \in G \subset X$, and consider $(H,\mcal{D})<(G,\mcal{P})$ a relatively quasiconvex subgroup and $R \geq 0$. An \emph{$R$-hull for $H$ acting on $X$} is a connected $H$-invariant full subgraph $\wtilde{Z_H} \subset X$ satisfying:
\begin{enumerate}
    \item $\wtilde{\ast}\in \wtilde{Z_H}$.
\item If $\gamma \subset X$ is a geodesic with endpoints in the limit set $\Lambda(H)$, then $N_R(\gamma) \cap N_R(G) \subset \wtilde{Z_H}$.
\item If $A$ is a horoball in $X$ containing a vertex $a$ at depth greater than 0 in the image $\check{\iota}(X_H)$, then $\wtilde{Z_H} \cap A^{(0)}$ contains every vertex of the maximal vertical ray in $A$ containing $a$.
\item The action of $(H,\mcal{D})$ on $\wtilde{Z_H}$ is \emph{weakly geometrically finite} in the sense of \cite[ A.27]{Agol2012VirtualHaken}.
\end{enumerate}
\end{defi}
By \cite[Lem.~7.6]{GrovesManning2021Quasiconv}, $R$-hulls for the action of $H$ on $X$ exist for any $R \geq 0$, and moreover we can construct them in such a way that they have only finitely many $H$-orbits of vertices at depth 0. Therefore, anytime we consider an $R$-hull, implicitly we will assume it satisfies this property. The relevance of $R$-hulls is that they allow us to extract some algebraic information about their corresponding relatively quasiconvex subgroups, as we will see in Proposition \ref{height=mult} below.

Let $Y = X/G$, and consider an $R$-hull $\wtilde{Z_H}$ for the action of $H$ on $X$ with quotient $Z_H := \wtilde{Z_H}/H$. The natural map $\mathscr{J}_H : Z_H \rightarrow Y$ induces the inclusion $H \rightarrow G$ in the following way. If $\ast_H \in Z_H$ and $\ast \in Y$ are the respective projections of $\wtilde{\ast}$, we obtain canonical surjections $s : \pi_1(Z_H, \ast_H) \rightarrow H$ and $s : \pi_1(Y,\ast) \rightarrow G$ such that the following diagram commutes
\begin{equation}\label{eq2}
\begin{tikzcd}
{\pi_{1}\left(Z_{H}, \ast_{H}\right)} \arrow[r,"\left(\mathscr{J}_{H}\right)_{*}"] \arrow[d,"s"] & {\pi_{1}(Y, \ast)} \arrow[d,"s"] \\
H \arrow[r]           & G          
\end{tikzcd}
\end{equation}
with the bottom map being the inclusion.

Let $\mcal{H} = \corchete{H_1,\dots,H_k}$ be a collection of relatively quasiconvex subgroups, and assume that $H_i \neq H_j$ whenever $i\neq j$. By Theorem \ref{dehnfilling} (3), let $\lambda$ be a quasiconvexity constant for each $H_j$ and $\ov{H}_j$ for all sufficiently long and $\mcal{H}$-wide fillings. Consider $R$-hulls $\wtilde{Z_j}= \wtilde{Z_{H_j}}$ for each $1 \leq j \leq k$, with quotients $Z_j = \wtilde{Z_j}/H_j$ and maps $\mathscr{J}_j=\mathscr{J}_{H_j}: Z_j \rightarrow Y$ as above.
\begin{defi}[cf.~Def.~A.15, \cite{Agol2012VirtualHaken}] Given $m > 0$ and an $m$-tuple $\sigma= (\sigma(1),\dots , \sigma(m)) \in \corchete{1,\dots , k}^m$, we write $|\sigma|= m$ and define
\begin{equation*}
    S_\sigma :=\corchete{(z_1,\dots , z_m) \in Z_{\sigma(1)} \times \dots \times Z_{\sigma(m)}
\colon \mathscr{J}_{\sigma(1)}(z_1)=\cdots =\mathscr{J}_{\sigma(m)}(z_m)
}\backslash \Delta_\sigma
\end{equation*}
where $\Delta_\sigma= \corchete{(z_1,\dots , z_m)\colon \exists  i \neq j \text{ s.t. } \sigma(i)= \sigma(j) \text{ and }z_i= z_j}$ should be understood as the fat
diagonal .
\end{defi}
Points in $S_\sigma$ have a well-defined depth which is the depth of the image in $Y$. Let $C$ be a component of $S_\sigma$ containing a base-point $p=(p_1,\dots , p_m)$ at depth 0, and fix maximal trees $T_j$ in $Z_j$ inducing preferred paths $\beta_v$ between $\ast_j=\ast_{H_j}$ and the vertex $v \in Z_j$. The paths $\beta_i=\beta_{p_i}$ give isomorphisms $\pi_1(Z_{\sigma(i)}, p_i) \rightarrow \pi_1(Z_{\sigma(i)}, \ast_{\sigma(i)})$, and composing them with the push-forwards of the projections $\omega_i: (C, p) \rightarrow (Z_{\sigma(i)}, p_i)$ we obtain homomorphisms $(\omega_i)_{\ast}: \pi_1(C, p) \rightarrow \pi_1(Z_{\sigma(i)},\ast_{\sigma(i)})$.
Define then $\tau_{C,i} : s \circ (\omega_i)_{\ast}: \pi_1(C, p) \rightarrow H_{\sigma(i)}$, where $s: \pi_1(Z_{\sigma(i)}, \ast_{\sigma(i)}) \rightarrow H_{\sigma(i)}$ is as in \eqref{eq2}.
\begin{defi}[cf.~Def.~7.7, \cite{GrovesManning2021Quasiconv}] The \emph{relative multiplicity} of $\corchete{\mathscr{J}_j : Z_j \rightarrow Y}_{1\leq j\leq k}$ is the largest $m$ so that there is some $\sigma$ with $|\sigma|=m$ and $S_\sigma$ containing a component $C$ such that the group $\tau_{C,i}(\pi_1(C, p))$ is infinite non-parabolic for all $1 \leq i \leq m$.
\end{defi}
\begin{rmk}\label{rmkmult}As we mentioned after Definition \ref{defRhull}, we are considering $R$-hulls with only finitely many orbits of vertices at depth 0. This implies that for any $j$, the number $r_j$ of vertices in $Z_j$ at depth 0 is finite. In particular, if $m > r_1+\cdots +r_k$ and $|\sigma|=m$, then any tuple in $Z_{\sigma(1)}\times \cdots \times Z_{\sigma(m)}$ at depth 0 lies in $\Delta_\sigma$, and consequently the relative multiplicity of $\corchete{\mathscr{J}_j : Z_j \rightarrow Y}_{1\leq j\leq k}$ is bounded by $r_1+\cdots +r_k$.
\end{rmk}
The main property of relative multiplicity is given by the following proposition:
\begin{prop}[cf.~Thm.~7.8, \cite{GrovesManning2021Quasiconv}]\label{height=mult} For sufficiently large $R$, depending only on $\delta$ and $\lambda$, if each $\wtilde{Z_j}$ is an $R$-hull for the action of $H_j$ on $X$, then the relative height of $\corchete{H_1,\dots,H_k}$ in $G$ is equal to the relative multiplicity of $\corchete{\mathscr{J}_j : Z_j \rightarrow Y}_{1\leq j\leq k}$.
\end{prop}
We need a preliminary lemma (cf.~\cite[Lem.~A.37]{Agol2012VirtualHaken}).
\begin{lemma}\label{gij} Let $C$ be a fixed component of $S_\sigma$ based at the point $p=(p_1,\dots , p_m)$ at depth 0. Then for any $1 \leq i, j \leq m$ there is some $g_{i,j} \in G$ such that $g_{i,j}\tau_{C,j}([\alpha]){g_{i,j}}^{-1}=\tau_{C,1}([\alpha])$ for any homotopy class $[\alpha] \in \pi_1(C, p)$.
\end{lemma}
\begin{proof} Recall the paths $\beta_i, \beta_j$ in the construction of the maps $\tau_{C,i}, \tau_{C,j}$. In virtue of the commutative diagram \eqref{eq2}, the element $\tau_{C,i}([\alpha]) \in G$ coincides with $s([(\mathscr{I}_{i} \circ \beta_{i}) \cdot(\mathscr{I}_{i} \circ \omega_{i} \circ \alpha) \cdot(\mathscr{I}_{i} \circ \ov{\beta}_{i})])$. But $\alpha$ is a loop in $C$, so $(\mathscr{I}_{i} \circ \omega_{i} \circ \alpha)=(\mathscr{I}_{j} \circ \omega_{j} \circ \alpha)$, and in particular $(\mathscr{I}_{i} \circ \beta_{i}) \cdot(\mathscr{I}_{j} \circ \ov{\beta}_{j})$ is a loop in $Y$ based at $\ast$. This loop defines the element $g_{i,j} \in G$, for which it is easy to check it satisfies the requirement.
 \end{proof}
\begin{proof}[Proof of Proposition \ref{height=mult}] We proceed in the same way as in the proof of \cite[Thm.~7.8]{GrovesManning2021Quasiconv}. First we prove that the relative height is at most the relative multiplicity, so suppose that the relative height is at least $m$, and let $\sigma=(\sigma(1),\dots ,\sigma(m))$ and $g_1H_{\sigma(1)},\dots , g_mH_{\sigma(m)}$ be distinct left cosets such that $\bigcap_{i=1}^{m}{H^{g_i}_{\sigma(i)}}$ contains a loxodromic element, say $h$. In \cite{GrovesManning2021Quasiconv} it was proven that by choosing an appropriate $R$, and by replacing $h$ by a power we may suppose that $h$ has a quasi-geodesic axis $\wtilde{\gamma}$ contained in $\bigcap^m_{i=1}{g_i\wtilde{Z_{\sigma(i)}}}$. In this case the path $\wtilde{\gamma}$ induces a loop $\gamma : \R/\Z \rightarrow Z_{\sigma(1)}\times \cdots \times Z_{\sigma(m)}$ defined by 
$$\gamma(t)=(\pi_{\sigma(1)}(g_1^{-1} \wtilde{\gamma}(t)),\dots , \pi_{\sigma(m)}(g_m^{-1}\wtilde{\gamma}(t))).$$
Since the cosets $g_{i}H_{\sigma(i)}$ are distinct and $G$ acts freely on $X$, $\gamma$ misses the fat diagonal $\Delta_\sigma$, and hence defines a loop in a connected component $C$ of $S_\sigma$ containing a vertex $p$ at depth 0. It is not hard to see that $\tau_{C,i}(\pi_1(C, p))$ contains a conjugate of $h$ for each $i$, so the relative multiplicity is at least $m$.

For the converse, suppose that the relative multiplicity of $\corchete{\mathscr{J}_j : Z_j \rightarrow Y}_{j}$ is $m$, and let $\sigma= (\sigma(1),\dots ,\sigma(m))$ and $C$ be a component of $S_\sigma$ containing a point $p$ at depth 0, such that for each $1 \leq i \leq m$ the group $A_i=\tau_{C,i}(\pi_1(C, p))< H_{\sigma(i)}<G$ contains a loxodromic element. Consider the elements $g_i=g_{1,i} \in G$ given by Lemma \ref{gij} so that $A^{g_i}_i=A_1$ for each $i$, which implies $A_1 \subset \bigcup^m_{i=1}{H^{g_i}_{\sigma(i)}}$.

It is only left to show that the cosets $g_iH_{\sigma(i)}$ are all different. Indeed, let $p=(p_1,\dots , p_m)$ and suppose there exist $i \neq j$ with $g_iH_{\sigma(i)}=g_jH_{\sigma(j)}$, implying $\sigma(i)=\sigma(j)$ and $g_i=g_jh$ for some $h \in H_{\sigma(i)}$. Let $\wtilde{\beta}_i, \wtilde{\beta}_j$ be the unique liftings of $\beta_i$ and $\beta_j$ to $\wtilde{Z_{\sigma(i)}} \subset X$ starting at $\wtilde{\ast}$, respectively, and note that $g^{-1}_jg_i$ is represented by the loop $(\mathscr{I}_{\sigma(j)} \circ \beta_{j}) \cdot(\mathscr{I}_{\sigma(i)} \circ \ov{\beta}_{i})$ in $\pi_1(Y, \ast)$. This implies $\wtilde{p}_j=h\wtilde{p}_i$, where $\wtilde{p}_i$ and $\wtilde{p}_j$ are the corresponding endpoints of $\wtilde{\beta}_i$ and $\wtilde{\beta}_j$, and by projecting back into $Z_{\sigma(i)}$ we get $p_i=p_j$, contradicting $p \notin \Delta_\sigma$.
\end{proof}
The previous proposition together with Remark \ref{rmkmult} implies the following corollary, generalizing \cite[Thm.~1.4]{HruskaWise2009Packing}.
\begin{coro}\label{heightfinite}If $G$ is a relatively hyperbolic group and $\mcal{H}$ is a finite collection of distinct relatively quasiconvex subgroups of $G$, then the relative height of $\mcal{H}$ in $G$ is finite.
\end{coro}
Before starting the proof of Theorem \ref{heightfilling} we state the following lemma, proven in \cite[Lem.~7.12 \& 7.14]{GrovesManning2021Quasiconv}.
\begin{lemma}\label{Rhull} Suppose $(G,\mcal{P})$ is relatively hyperbolic and $\mcal{H}=\corchete{H_1,\dots,H_k}$ is a collection of relatively quasiconvex subgroups of $G$. Then for all $R$ there is some $R'$ satisfying the following: for all sufficiently long and $\mcal{H}$-wide fillings $\phi : G \rightarrow G/K$, if $\wtilde{Z_j}$ is an $R'$-hull for the action of $H_j$ on $X$ for each $1 \leq j \leq k$, then its image $\wtilde{\ov{Z}_j} \subset \ov{X}$ under $\phi$ is an $R$-hull for the action of $\ov{H}_j$ on $\ov{X}$ and is the embedded image of $\wtilde{Z_j}/(H_j \cap K)$.
\end{lemma}
\begin{proof}[Proof of Theorem \ref{heightfilling}] Let $\delta$ be a constant such that the cusped spaces $X$ and $\ov{X}$ are both $\delta$-hyperbolic for all sufficiently long fillings, and let $\lambda$ be a common quasiconvexity constant for the groups $H_j$ , which we may also assume is a quasiconvexity constant for each $\ov{H}_j$ in any sufficiently long and $\mcal{H}$-wide filling $\phi : G \rightarrow G/K$. For these fillings, by Theorem \ref{dehnfilling} (2),(4) we may assume $\ov{H}_i\neq \ov{H}_j$ for all $i\neq j$ (just take a set $A$ containing elements in $H_i\backslash H_j$ for any pair $i\neq j$), and also that each $\ov{H}_j$ is naturally isomorphic to the image of the induced filling of $H_j$ . Let $R$ be sufficiently large so that Proposition \ref{height=mult} applies for $\delta$ and $\lambda$, and let $R'=R'
(R)$ be given by Lemma \ref{Rhull}.

For each $1 \leq j \leq k$ consider the following commutative diagram
\begin{equation*}
\begin{tikzcd}
\wtilde{Z_j} \arrow[r] \arrow[d] \arrow[loop left,"H_j"] & X \arrow[d]  \arrow[loop right,"G"]\\
\wtilde{\ov{Z}_j} \arrow[r]   \arrow[loop left,"H_j/K_j"]        & \ov{X}  \arrow[loop right,"G/K"]  
\end{tikzcd}
\end{equation*}
where $\wtilde{Z_j}$ is an $R'$-hull for the action of $H_j$ on $X$, and $\ov{X}=X/K$, $K_j=K \cap H_j$, and $\wtilde{\ov{Z}_j}=\wtilde{Z_j}/K_j$. From Lemma \ref{Rhull}, $\wtilde{\ov{Z}_j}$ embeds in $\ov{X}$ and is an $R$-hull for the action of $H_j/K_j$ on $\ov{X}$.\\
Taking quotients under the respective groups we obtain the diagram
\begin{equation}\label{eq3}
\begin{tikzcd}
Z_j \arrow[r,"\mathscr{J}_j"] \arrow[d]  & Y \arrow[d] \\
\ov{Z}_j \arrow[r,"\ov{\mathscr{J}}_j"] & \ov{Y}  
\end{tikzcd}
\end{equation}
the vertical maps being homeomorphisms. By our choice of $R$, Proposition \ref{height=mult} implies that the relative heights of $\mcal{H}$ in $G$ and of $\ov{\mcal{H}}$ in $\ov{G}$ coincide with the relative multiplicities of $\corchete{\mathscr{J}_j : Z_j \rightarrow Y}_j$ and $\corchete{\ov{\mathscr{J}}_j : \ov{Z}_j \rightarrow \ov{Y}}_j$, respectively.

The vertical maps of the diagram \eqref{eq3} induce depth-preserving homeomorphisms between
\begin{equation*}
    S_\sigma=\corchete{(z_1,\dots , z_m) \in Z_{\sigma(1)}\times \dots  Z_{\sigma(m)} \colon \mathscr{J}_{\sigma(1)}(z_1)=\cdots= \mathscr{J}_{\sigma(m)}(z_m)}\backslash \Delta_{\sigma}
\end{equation*}
and
\begin{equation*}
    \ov{S}_\sigma=\corchete{(z_1,\dots , z_m) \in \ov{Z}_{\sigma(1)}\times \dots  \ov{Z}_{\sigma(m)} \colon \ov{\mathscr{J}}_{\sigma(1)}(z_1)=\cdots= \ov{\mathscr{J}}_{\sigma(m)}(z_m)}\backslash \ov{\Delta}_{\sigma}
\end{equation*}
for any $\sigma=(\sigma(1),\dots , \sigma(m))$ with $|\sigma|=m > 0$, where $\Delta_\sigma$ and $\ov{\Delta}_\sigma$ are the corresponding fat diagonals.

If $\pi_\sigma : S_\sigma \rightarrow \ov{S}_\sigma$ is this homeomorphism, since $\ov{H}_j \cong H_j/K_j$ for all $j$, then for any component $C$ of $S_\sigma$ with base-point $p$ at depth 0 and for any $1 \leq i \leq m$, the following diagram commutes
\begin{equation*}
\begin{tikzcd}
\pi_1(C,p) \arrow[r,"\tau_{C,i}"] \arrow[d,"(\pi_\sigma)_\ast"]  & H_{\sigma(i)} \arrow[d,"\phi_{H_{\sigma(i)}}"] \\
\pi_1(\ov{C},\ov{p}) \arrow[r,"\tau_{\ov{C},i}"] & \ov{H}_{\sigma(i)}
\end{tikzcd}
\end{equation*}
where $\phi : G \rightarrow \ov{G}$ is the filling map. In particular, if $\tau_{\ov{C},i}(\pi_1(\ov{C}, \ov{p}))<  \ov{G}$ contains a loxodromic element of $(\ov{G},\ov{\mcal{P}})$, then $\tau_{C,i}(\pi_1(C, p))$ contains a loxodromic element of $(G,\mcal{P})$ as well, and therefore the relative height of $\mcal{H}$ in $G$ is at least the relative height of $\ov{\mcal{H}}$ in $\ov{G}$.
\end{proof}
\subsection{Weak separability of double cosets} In this subsection we study the behavior of double cosets of relatively quasiconvex subgroups under Dehn filling. Similarly to the peripheral separability we require for a single relatively quasiconvex subgroup, we need some assumptions on the double cosets.

\begin{defi}Let $H$, $L$ be relatively quasiconvex subgroups of the relatively hyperbolic group $(G,\mcal{P})$. The pair $H, L$ is said to be \emph{doubly peripherally separable} if the double coset $(H^{g_1}\cap P)(L^{g_2}\cap P)$ is separable in $P$ for any $g_1,g_2\in G$ and any peripheral subgroup $P<G$ such that $H^{g_1} \cap P$ and $L^{g_2} \cap P$ are infinite.
\end{defi}

Our next theorem establishes weak separability for double cosets of doubly peripherally separable pairs of relatively quasiconvex subgroups under Dehn fillings. This extends \cite[Prop.~6.2]{GrovesManning2021Quasiconv}, where the peripheral subgroups are assumed to be abelian, and hence peripheral separability and double peripheral separability trivially hold. It also generalizes \cite[Thm.~6.4]{Groves2018HyperbolicImproperly}, where it is assumed that the relatively quasiconvex subgroups are full (though they proved a result for several subgroups, and also allowing some finite subsets removed).

\begin{thm}\label{doublecosetfilling} Let $(G,\mcal{P}=\corchete{P_1,\dots , P_n})$ be relatively hyperbolic and let $H, L <  G$ be relatively quasiconvex and peripherally separable subgroups of $(G,\mcal{P})$ such that the pair $H, L$ is doubly peripherally separable. Also, let $S \subset\bigcup{\mcal{P}}\backslash \corchete{1}$ be a finite set and consider $a \in G\backslash HL$.

Then there exist finite index subgroups $K_i \trianglelefteq P_i$ such that for any subgroups $N_i < K_i$ the filling $G \rightarrow G(N_1,\dots,N_n)=G/K$ is $(\corchete{H, L} , S)$-wide, and moreover $a \notin KHL$.
\end{thm}

The lemmas below are \cite[Prop.~A.6]{ManningMartinezPedroza2009Separation} and \cite[ Lem.~4.1 \& Lem.~4.4]{GrovesManning2021Quasiconv}, respectively, where for \cite[Lem.~4.1]{GrovesManning2021Quasiconv} we cite the statement for the case $L_1=10\delta$.
\begin{prop}\label{stabinf} Let $H$ be a relatively quasiconvex subgroup of a relatively hyperbolic group $(G,\mcal{P})$. For a cusped space $X$ for $(G,\mcal{P})$, there exists a positive constant $R$ such that for any horoball $A$ in $X$, if there is a geodesic in $X$ with endpoints in $H$ and intersecting $A$ in a point at depth $R$, then $H \cap \Stab(A)$ is infinite.
\end{prop}
\begin{lemma}\label{loops}Let $(G,\mcal{P})$ be with cusped space $X$, and let $\delta$ be a hyperbolicity constant for $X$ and for the induced cusped space of any sufficiently long filling. Then for any $L_2 \geq 10\delta$, and for all sufficiently long fillings $\phi : G \rightarrow G/K$ the following holds: if $\gamma$ is a geodesic in $X$ such that $\phi(\gamma)$ is a loop in $\ov{X}$, then:
\begin{itemize}
\item there is a horoball $A$ in $X$ so that $\gamma$ intersects $A$ in a segment $[x, y]$ with $x, y \in G$ and containing a point at depth $L_2$, and
\item there is some $k \in K \cap \Stab(A)$ so that $d_X(x, ky) \leq 20\delta+3$.
\end{itemize}
\end{lemma}
\begin{proof}[Proof of Theorem \ref{doublecosetfilling}] Almost all the work has been done in \cite[Prop.~6.2]{GrovesManning2021Quasiconv}, so we will only check the details where the doubly peripheral separability assumption is required.

By Lemma \ref{existencewide} there exist finite index subgroups $\dot{N}_i \trianglelefteq P_i$ such that any filling $G \rightarrow G(N_1,\dots,N_n)$ with $N_i <\dot{N}_i$ is $(\corchete{H, L} , S)$-wide. The rest of the proof is devoted to finding finite index subgroups $\dot{K}_i \trianglelefteq P_i$ such that $a$ is separated from $HL$ in any filling $G \rightarrow G(N_1,\dots,N_n)$ with $N_i<\dot{K}_i$. In that case, the theorem follows by defining $K_i:= \dot{N}_i \cap \dot{K}_i$.

Consider an induced peripheral structure $\mcal{D}$ on $H$, with each $D \in \mcal{D}$ of the form $D=H \cap P^{c_D}_{j_D}$
for some $1 \leq j_D \leq n$ and some shortest $c_D \in G$. Similarly, let $\mcal{E}$ be an induced peripheral structure on $L$ with each $E \in \mcal{E}$ of the form $E=L \cap P^{d_E}_{k_E}$ for some $1 \leq k_E \leq n$ and some shortest $d_E \in G$. Also, let $X, X_H, X_L$ be cusped spaces for $(G,\mcal{P})$, $(H,\mcal{D})$ and $(L, \mcal{E})$ respectively, with induced maps $\check{\iota}_H : X^{(0)}_H \rightarrow X$ and $\check{\iota}_L : X^{(0)}_L \rightarrow X$. Let $\delta \geq 1$ be a hyperbolicity constant for $X$ and for the induced cusped space of any sufficiently long filling of $G$, and let $\lambda\geq 2\delta+1$ be such that the images $\check{\iota}_H(X^{(0)}_H)$  and $\check{\iota}_L(X^{(0)}_L)$ are $\lambda$-quasiconvex in $X$. We assume that $\delta$ and $\lambda$ are integer numbers.

For $H$ and $L$, consider the constants $R_H$, $R_L$ given by Proposition \ref{stabinf}, and define $M=d_X(1, a)$ and $L_2=\max\corchete{20\delta+ M+ \lambda +4, R_H, R_L}$. Given $1 \leq i \leq n$, let $S_i\subset P_i$ be a finite set containing
\begin{equation*}
\corchete{p \in P_i \colon d_X(1, p) \leq 48\delta + 8M + 16\lambda+5}.
\end{equation*}
Also, consider the collection $\mcal{C}_i$ of pairs of subgroups of $P_i$ of the form $(B_1, B_2)=(D^{\alpha c_D^{-1}}, E^{\beta d_E^{-1}})$ with $D \in \mcal{D}$ and $E \in \mcal{E}$ such that $j_D=k_E=i$, and $\alpha, \beta \in S_i$. Note that there are finitely many such pairs, and that by assumption for each of them the double coset $B_1B_2$ is separable in $P_i$. This implies the existence of finite index subgroups $\dot{K}_{B_1,B_2} \trianglelefteq P_i$ such that
\begin{equation}\label{eq4}
    S_i \cap \dot{K}_{B_1,B_2}B_1B_2\subset B_1B_2.
\end{equation}
Set $\dot{K}_i:=\bigcap_{(B_1,B_2)\in \mcal{C}_i}{\dot{K}_{B_1,B_2}}$, for which we expect the separability conditions to hold.

Let $N_i<\dot{K}_i$ be filling kernels inducing the filling $\phi : G \rightarrow G(N_1,\dots,N_n)=G/K$, for which we claim that $a \notin KHL$. Suppose by contradiction that there is some $g \in K$ of the form $g= hla^{-1}$ for some $h \in H$ and $l \in L$, and assume $d_X(1, g)$ is minimal among the elements of $K \cap HLa^{-1}$.

Consider a geodesic quadrilateral in $X$ with vertices $1, h, hl, g$, and let $\xi_1, \xi_2$, $\eta$ and $\rho$ be the geodesics from $1$ to $h$, from $h$ to $hl$, from $hl$ to $g$ and from $1$ to $g$, respectively. By assumption $g \neq 1$.

The map $X \rightarrow \ov{X}$ induced by $\phi$ sends $\rho$ to a loop, so by Lemma \ref{loops} there is a horoball $A$ in
$X$ intersecting $\rho$ in a geodesic segment $[x, y]$ (with $x$ between 1 and $y$), such that $[x, y]$ contains points at depth $L_2$, and there is some $k \in K \cap \Stab(A)$ with $d_X(x, ky) \leq 20\delta +3$. This implies $d_X(1, kg)< d_X(1, g)$, since $d_X(x, ky) \leq 20\delta+3$ and $d_X(x, y) \geq 2L_2 > 20\delta+4$. Our contradiction will be obtained by showing that $kg \in HLa^{-1}$, contrary to our minimality assumption on $g$.

By \cite[Lem.~3.10]{GrovesManning2008DehnFilling} we can assume that the segment $[x, y]$ is a geodesic through $A$ consisting of a vertical segment down from $x$, a horizontal segment of length at most 3, and then a vertical segment with endpoint $y$. We will make a similar assumption in case $A$ intersects $\xi_1$ or $\xi_2$. Let $x'$, $y'$ be the points on $[x, y]$ directly below $x$ and $y$, respectively, at depth $3\delta+M  +\lambda$. The $\delta$-hyperbolicity condition applied to the quadrilateral $\xi_1 \cup \xi_2 \cup \eta \cup \rho$ implies that $x'$ and $y'$ lie in the $2\delta$-neighborhood of $\xi_1 \cup \xi_2 \cup \eta$, but since $\eta$ is of length $M$ and has extreme points at depth 0, in fact $x'$ and $y'$ are within $2\delta$ of $\xi_1 \cup \xi_2$.

The geodesics $\xi_1$ and $\xi_2$ join points in $H$ and $hL$, respectively, and so $\lambda$-quasiconvexity gives us points $u_0, v_0 \in \check{\iota}_H(X^{(0)}_H)\cup  h\check{\iota}_L(X^{(0)}_L)$ with $d_X(u_0, x'), d_X(v_0, y') \leq 2\delta+\lambda$. The points $u_0$ and $v_0$ lie in $A$ since $x'$ and $y'$ have depth greater than $2\delta+\lambda$. Let $u, v \in A$ be the points at depth 0 directly above $u_0$ and $v_0$ respectively. We have $d_X(u, x) \leq d_X(u, u_0)+  d_X(u_0, x')+d_X(x', x) \leq 2(d_X(x, x')+ d_X(u_0, x')) \leq 10\delta + 2M  +4\lambda$, and similarly $d_X(v, y) \leq 10\delta + 2M + 4\lambda$.

This gives us four cases depending on whether each of $u_0, v_0$ are contained in $\check{\iota}_H(X^{(0)}_H)$ or $h\check{\iota}_L(X^{(0)}_L)$. The cases where $u_0$ and $v_0$ are both contained in  $\check{\iota}_H(X^{(0)}_H)$ or $h\check{\iota}_L(X^{(0)}_L)$ are completely analogous, and the proof given in \cite[Prop.~6.2]{GrovesManning2021Quasiconv} still works here since it only requires peripheral separability. As noted at the end of that proof, the case where $u_0$ is contained in $h\check{\iota}_L(X
^{(0)}_L)$ and $v_0$ is contained in $\check{\iota}_H(X^{(0)}_H)$ 
 essentially becomes the case where $u_0, v_0$ are contained in $\check{\iota}_H(X^{(0)}_H)$. Therefore, we are only left to deal with the case where $u_0$ is contained in $\check{\iota}_H(X^{(0)}_H)$ and $v_0$ is contained in $h\check{\iota}_L(X^{(0)}_L)$, which we now check.

We can write $u_0=(sc_DP_{j_D},h_uc_D, n)$ for some $s \in H, h_u \in sD$ and $n \in \N$, where $D \in \mcal{D}$, and similarly $v_0= (htd_EP_{k_E},hl_vd_E, m)$ for $t \in L, l_v \in tE, m \in \N$ and $E \in \mcal{E}$. This implies $u=h_uc_D$ and $v=hl_vd_E$. We write $c=c_D, d=d_E$, and $j_D=k_E=i$, where $A$ is the horoball based on
the coset $scP_i=htdP_i$.

The geodesic $\xi_1$ intersects $A$ in a segment with extreme point $h_1$ closer to $h$. Since we may assume that $x'$ lies within $2\delta$ of $\xi_1$, there exists a point $h_1'$ in $\xi_1\cap A$ directly below $h_1$ and at depth $\delta+M+\lambda$. By $\lambda$-quasiconvexity of $\check{\iota}_H(X_H)$ we can find some $w'\in \check{\iota}_H(X_H)$ with $d_X(w',h_1')\leq \lambda$, which indeed lies in $A$, so that group element $w\in A$ directly above $w'$ is of the form $w=h_wc \in sDc \subset scP_i$ with $h_w \in H$ and $d_X(w,h_1)\leq d_X(w,w')+d_X(w',h_1')+d_X(h_1',h_1)\leq 2\delta+2M+4\lambda$. In the same way, if $g_2$ is the extreme point of the segment $\xi_2 \cap A$ closer to $h$, then it is directly above the vertex $g_2'\in A$ at depth $\delta+M+\lambda$ and there is some $z=hl_zd \in htEd \subset htdP_i$ with $d_X(g_2, z) \leq 2\delta+2M+4\lambda$. 

Let $[h,h_1']\cup [h,g_2']\cup[h_1',g_2']$ be a geodesic triangle in $X$ with $[h,h_1']\subset \xi_1$ and $[h,g_2']\subset \xi_2$. Since $h_1'$ and $g_2'$ are at depth $\delta+M+\lambda\geq 3\delta+1$ (recall that $\lambda\geq 2\delta+1)$, it follows from \cite[Lemm.~3.26]{GrovesManning2008DehnFilling} that every point of $[h_1',g_2']$ is at depth at least $3\delta+1$. By $\delta$-hyperbolicity, this implies that the vertex in $[h,h_1']$ directly below $h_1$ at depth $\delta+1$ lies within $\delta$ of a point in $[h,g_2']$, which must lie directly below $g_2$ and at depth at most $2\delta+1$. In particular, $d_X(h_1,g_2) \leq 4\delta+2$ and we have $d_X(w,z)\leq 8\delta+4M+8\lambda+2$.

Let $p=(u^{-1}kv)(z^{-1}w)=(u^{-1}ku)(u^{-1}w)(z^{-1}v)
^{w^{-1}z}$. Since $k \in K \cap \Stab(A)=K \cap P_i^{c_D}$
and $K$ is normal, we have $u^{-1}ku \in K \cap P_i =N_i \trianglelefteq P_i$. Also, note that $u^{-1}w \in D^{c^{-1}}, z^{-1}v \in E^{d^{-1}}$ and $w^{-1}z \in P_i$, implying $p \in N_iD^{c^{-1}}E^{(w^{-1}z)d^{-1}}$. In addition, $d_X(u, kv) \leq d_X(u, x) + d_X(x, ky)+d(y, v) \leq 40\delta + 4M + 8\lambda + 3$, so $d_X(1, p) \leq d_X(w, z) + 40\delta + 4M + 8\lambda + 3 \leq 48\delta + 8M + 16\lambda +  5$, and hence $p \in S_i$. The pair $(D^{c^{-1}},E^{(w^{-1}z)d^{-1}})$ is then in $\mcal{C}_i$, so in virtue of \eqref{eq4} there exist $\hat{d} \in D^{c^{-1}}$ and $\hat{e} \in E^{d^{-1}}$ such that $p=\hat{d}(w^{-1}z)\hat{e}(z^{-1}w)$. In consequence
$khl$ can be written as
\begin{align*}
khl &=u\left(u^{-1} k v z^{-1} w\right)\left(w^{-1} z\right) v^{-1} h l \\
&=u p\left(w^{-1} z\right) v^{-1} h l \\
&=u\left(\hat{d}\left(w^{-1} z\right) \hat{e}\left(z^{-1} w\right)\right)\left(w^{-1} z\right) v^{-1} h l \\
&=\left(h_{u} c\right) \hat{d}\left(c^{-1} h_{w}^{-1} h l_{z} d\right) \hat{e}\left(d^{-1} l_{v}^{-1} h^{-1}\right) h l \\
&=\left(h_{u} \hat{d}^{c} h_{w}^{-1} h\right)\left(l_{z} \hat{e}^{d} l_{v}^{-1} l\right),
\end{align*}
the last expression being in $HL$ since $\hat{d}^c \in D$ and $\hat{e}^d \in E$. Therefore, $kg=khla^{-1} \in HLa^{-1}$,
implying the desired contradiction and concluding the proof of this case and of the theorem.
 \end{proof}
The next corollary roughly says that under some mild assumptions, double cosets of doubly peripherally separable pairs of relatively quasiconvex subgroups are ``almost" the intersection of double cosets of fully relatively quasiconvex subgroups.
\begin{coro}\label{almostfullydoublecoset}Let $(G,\mcal{P})$ be relatively hyperbolic and let $H, L <  G$ be relatively quasiconvex subgroups such that $H, L$ is a doubly peripherally separable pair. Also, suppose that $\dot{H}, \dot{L}$ and $\dot{P}$ are separable in $G$ for any finite index subgroups $\dot{H}<H, \dot{L} <  L$ or $\dot{P}<   P$, with $P$ being a peripheral subgroup of $G$.\\
Then for any $a \in G\backslash HL$ there exist fully relatively quasiconvex subgroups $\hat{H}, \hat{L}$ with $H \cap \hat{H} <  H$ and $L \cap \hat{L}< L$ of finite index, such that $a \notin H\hat{H}\hat{L}L$.
\end{coro}
\begin{proof}For $a, H$ and $L$ as in the statement, let $K_1,\dots K_n$ be the filling kernels given by Theorem \ref{doublecosetfilling}, and let $K=K(a, H, L) \trianglelefteq G$ be the kernel of the filling $G \rightarrow G(K_1,\dots , K_n)$. Since by assumption each $K_i$ is finite index in $P_i$ and separable in $G$, the subgroups $K_i \cap H$ are separable in $H$, and there exists a finite index subgroup $\dot{H} \trianglelefteq H$ such that $\dot{H} \cap P_i<   K_i$ for all $i$, implying $\dot{H} \cap P <  K$ for any peripheral subgroup $P$. Similarly we can find $\dot{L} \trianglelefteq L$ of finite index so that $\dot{L} \cap P <  K$ for any peripheral $P$.

We claim the existence of parabolic subgroups $Q_1,\dots , Q_s, R_1,\dots , R_t <  G$, each of them a finite index subgroup of the conjugate of some $K_i$, and such that $\hat{H}:= \genby{\dot{H}, Q_1,\dots , Q_s}$ and $\hat{L} := \genby{\dot{L}, R_1,\dots , R_t}$ are fully relatively quasiconvex subgroups of $G$. By assuming this claim, the corollary follows since $a \notin KHL$ and
\begin{equation*}
    H\hat{H}\hat{L}L=H\genby{\dot{H}, Q_1,\dots , Q_s}\genby{\dot{L},R_1,\dots , R_t}L\subset H\genby{\dot{H},K}\genby{\dot{L},K}L=H\dot{H}K\dot{L}L =KHL.
\end{equation*}

To prove the claim, we will only construct the groups $Q_1,\dots , Q_s$, since the groups $R_1,\dots , R_t$ can be found in exactly the same way. Let us choose $\mcal{D}_0 =\corchete{D_1,\dots , D_s}$ an induced peripheral structure on $\dot{H}_0=\dot{H}$ with $D_i= \dot{H} \cap P^{c_i}_{j_i}$ for some $1 \leq j_i \leq n$ and $c_i \in G$, and inductively construct $Q_1,\dots , Q_s$ with each $Q_i <K^{c_i}_{j_i}$ of finite index, such that $\mcal{D}_i:= \corchete{Q_1,\dots , Q_i, D_{i+1},\dots , D_{s}}$ is an induced peripheral structure on $\dot{H}_i:= \genby{\dot{H}, Q_1,\dots , Q_i}=\genby{\dot{H}_{i-1}, Q_i}$.

Assume we have found $\dot{H}_{i-1}$ and $Q_1,\dots , Q_{i-1}$, and note that every parabolic subgroup of $\dot{H}_{i-1}$ is a subgroup of $K$. If $D_i< K^{c_i}_{j_i}$ is finite index, define $Q_i=D_i$. Otherwise, by \cite[Thm.~1.1]{MartinezPedroza2009Combqc} there exists a finite set $F\subset P^{c_i}_{j_i}\backslash(\dot{H}_{i-1} \cap P^{c_i}_{j_i})$ such that for any $P'< P^{c_i}_{j_i}$ containing $\dot{H}_{i-1}\cap P_{j_i}^{c_i}$ and disjoint from $F$, the group $\genby{\dot{H}_{i-1}, P'}$ is relatively quasiconvex and has $\corchete{Q_1,\dots , Q_{i-1}, P'}$ as induced peripheral structure. It is then enough to find a subgroup $P'<K_{j_i}^{c_i}$ of finite index, containing $\dot{H}_{i-1} \cap P_{j_i}^{c_i}$ and disjoint from $F$, so in that case we can define $Q_i=P'$.
 
To find such $P'$, note that each infinite parabolic subgroup of $\dot{H}_{i-1}$ is $\dot{H}_{i-1}$-conjugate into some group in $\mcal{D}_{i-1}$, implying $\dot{H}_{i-1} \cap P^{c_i}_{j_i}=D_i$. Also, since $\dot{H}$ is separable in $G$, we have that $D_i$ is separable in $P^{c_i}_{j_i}$. But by construction $D_i$ is contained in $K$, so in fact $D_i$ is separable in $K^{c_i}_{j_i}$, and there is a finite index subgroup $P' < K^{c_i}_{j_i}$, disjoint from $F$ and such that $P'\supset D_i$. This solves the claim, and since each subgroup in $\mcal{D}_s$ is finite index in some peripheral subgroup, the group $\hat{H}=\dot{H}_s$ is fully relatively quasiconvex.
\end{proof}
\subsection{The malnormal special quotient theorem} We end this section by presenting a result that will be needed in the proofs of Theorem \ref{finitewalls} and Proposition \ref{weakseparability}. It depends on the relative version of Wise's malnormal special quotient theorem, recently obtained by Einstein \cite[Thm.~2]{Einstein2019RMSQT}.

\begin{thm}[Einstein]\label{msqt} If $(G, \corchete{P_1,\dots , P_n})$ is a relatively hyperbolic group with $G$ cubulated and virtually special, then there exist finite index subgroups $\dot{P}_i \trianglelefteq P_i$ such that if $\ov{G}=G(N_1,\dots,N_n)$ is any filling with each $N_i <  \dot{P}_i$ of finite index, then $\ov{G}$ is hyperbolic and virtually special.
\end{thm}
By combining Theorem \ref{dehnfilling}, Lemma \ref{existencewide}, Theorem \ref{heightfilling} and Einstein's Theorem \ref{msqt}, we obtain:
\begin{prop}\label{superfilling}Let $(G,\mcal{P})$ be a relatively hyperbolic group with each $P \in \mcal{P}$ being residually
finite, and let $\mcal{H}=\corchete{H_0, H_1, H_2\dots , H_k}$ be a collection of relatively quasiconvex subgroups of $G$. Assume that the groups $H_1,\dots , H_k$ are all distinct and contained in $H_0$, that $H_l$ is peripherally separable for $0 \leq l \leq k$, and $H_0$ is strongly peripherally separable. Consider also a finite set $A\subset G$.

Then there exist finite index subgroups $\dot{P}_i \trianglelefteq P_i$ such that for any further finite index subgroups $N_j<\dot{P}_j$, the filling $\phi : G \rightarrow \ov{G} =G(\mcal{N}) =G(N_1,\dots,N_n)$ satisfies:
\begin{enumerate}\item $\ov{G}$ is hyperbolic.
\item $\ov{H}_l=\phi(H_l)$ is quasiconvex in $\ov{G}$ and naturally isomorphic to the induced filling $H_l(\mcal{N}_{H_l})$ for each $0 \leq l \leq k$.
\item  $\phi|_A : A \rightarrow \ov{G}$ is injective and $\phi(A \cap H_l)=\phi(A) \cap \ov{H}_l$ for all $0 \leq l \leq k$.
\item For each $l$, if $H_l$ is virtually special and strongly peripherally separable, then $\ov{H}_l$ is also virtually special.
\item For $1 \leq l \leq k$, $\ov{H}_l$ is isomorphic to the filling induced by $H_0 \rightarrow \ov{H}_0$.
\item The relative height of $\corchete{\ov{H}_1,\dots , \ov{H}_k}$ in $\ov{H}_0$ is at most the relative height of $\corchete{H_1,\dots,H_k}$ in $H_0$.
\end{enumerate}
\end{prop}
\begin{proof}First, consider a peripheral structure $\mcal{D}_0=\corchete{D_{0,1},\dots , D_{0,k_0}}$ on $H_0$ induced by $G$ so that $D_{0,i}=H_0 \cap P^{c_{0,i}}_{\alpha_{0,i}}$ for some $1 \leq \alpha_{0,i} \leq n$ and some shortest $c_{0,i} \in G$. Also, for any $1 \leq l \leq k$ consider a peripheral structure $\mcal{D}_l=\corchete{D_{l,1},\dots , D_{l,k_l}}$ on $H_l$ induced by $(H_0, \mcal{D}_0)$, so that $D_{l,i}=H_l \cap D^{d_{l,i}}_{0,\beta_{l,i}}$ for some $1 \leq \beta_{l,i} \leq k_0$ and some shortest $d_{l,i} \in H_0$. Since $\mcal{D}_l$ is also a peripheral structure induced by $G$, we have $D_{l,i}=H_l \cap P^{c_{l,i}}_{\alpha_{l,i}}$ for $1 \leq \alpha_{l,i} \leq n$ and $c_{l,i} \in G$. The equation
\begin{equation*}
D_{l,i}=H_l \cap P^{c_{l,i}}_{\alpha_{l,i}}=H_l \cap D^{d_{l,i}}_{0,\beta_{l,i}}=H_l \cap P^{d_{l,i}c_{0,\beta_{l,i}}}_{\alpha_{0,\beta_{l,i}}}
\end{equation*}
then implies
\begin{equation}\label{eq5}
\alpha_{l,i}=\alpha_{0,\beta_{l,i}}, \text{ and }(c_{l,i})^{-1}d_{l,i}c_{0,\beta_{l,i}} \in P_{\alpha_{l,i}} .
\end{equation}
The relevance of this equation is that if $\phi : G \rightarrow G(N_1,\dots,N_n)$ induces the filling $\phi_0 : H_0 \rightarrow
H_0(K_{0,1},\dots , K_{0,k_0})$, then for any $1 \leq l \leq k$ the filling $\phi_l: H_l \rightarrow H_l(K_{l,1},\dots , K_{l,k_l})$ induced by $\phi$ is the same as the one induced by $\phi_0$. Indeed, a filling kernel of $\phi_l$ induced by $\phi$ is of the form $K_{l,i}=D_{l,i} \cap N^{c_{l,i}}_{\alpha_{l,i}}$ , while the one induced by $\phi_0$ takes the form $K^{(0)}_{l,i}=D_{l,i} \cap K^{d_{l,i}}_{0,\beta_{l,i}}=D_{l,i} \cap N^{d_{l,i}c_{0,\beta_{l,i}}}_{\alpha_{0,\beta_{l,i}}}$.
The equality $K_{l,i}=K^{(0)}_{l,i}$ then follows from $N^{c_{l,i}}_{\alpha_{l,i}}=N^{d_{l,i}c_{0,\beta_{l,i}}}_{\alpha_{0,\beta_{l,i}}}$, which is consequence of \eqref{eq5} and the fact that $N_{\alpha_{l,i}}$ is normal in $P_{\alpha_{l,i}}$.

With that in mind, by Theorem \ref{dehnfilling} we can find a finite set $S \subset\bigcup{\mcal{P}}\backslash \corchete{1}$ such that all $(\mcal{H}, S)$-wide and peripherally finite fillings $\phi : G \rightarrow \ov{G}=G(N_1,\dots,N_n)$ with $S \cap (\bigcup{N_i}) = \emptyset$ satisfy the conditions (1)-(3) (for the injectivity of $\phi|_A$, include the trivial group into $\mcal{H}$ and consider the finite set $\corchete{ab^{-1}\neq 1: a, b \in A}\subset G$). For such $S$, Lemma \ref{existencewide} gives us finite index subgroups $\dot{N}_j \trianglelefteq P_j$ such that any filling $G \rightarrow G(N_1,\dots,N_n)$ with $N_j <\dot{N}_j$ is $(\mcal{H}, S)$-wide.

Now, let $I$ be the set of $0 \leq l \leq k$ such that $H_l$ is virtually special, and apply Theorem \ref{msqt} to each pair $(H_l, \mcal{D}_l)$ with $l \in I$ to obtain finite index subgroups $\dot{D}_{l,i} \trianglelefteq D_{l,i}$ such that for any further finite index subgroups $\wtilde{D}_{l,i}<\dot{D}_{l,i}$ the filling $H_l(\wtilde{D}_{l,1},\dots , \wtilde{D}_{l,k_l})$ is virtually special.

Given $l \in I$ and $1 \leq i \leq k_l$, there is a chain of inclusions \begin{equation*}
    \dot{D}_{l,i} \cap \dot{N}^{c_{l,i}}_{\alpha_{l,i}} \trianglelefteq D_{l,i} \cap \dot{N}^{c_{l,i}}_{\alpha_{l,i}} < \dot{N}^{c_{l,i}}_{\alpha_{l,i}},
\end{equation*}
the first one being of finite index. Since $H_l$ is strongly peripherally separable, the group $\dot{D}_{l,i}$ is separable in $P^{c_{l,i}}_{\alpha_{l,i}}$ , and hence the inclusion $\dot{D}_{l,i} \cap \dot{N}^{c_{l,i}}_{\alpha_{l,i}}<\dot{N}
^{c_{l,i}}_{\alpha_{l,i}}$ is also separable. Therefore we can find a finite index subgroup $\wtilde{N}_{\gamma_{l,i}} \trianglelefteq P_{\alpha_{l,i}}$ contained in $\dot{N}_{\alpha_{l,i}}$ , such that $\dot{D}_{l,i} \cap \wtilde{N}^{c_{l,i}}_{\gamma_{l,i}}= D_{l,i} \cap \wtilde{N}^{c_{l,i}}_{\gamma_{l,i}}$. We conclude that $K_{l,i} := D_{l,i} \cap \wtilde{N}^{c_{l,i}}_{\gamma_{l,i}}$ is contained in $\dot{D}_{l,i}$ as a finite index subgroup.

For $1 \leq j \leq n$, let $I_j$ be the set of pairs $(l, i)$ with $l \in I$ and $1 \leq i \leq k_l$, such that $j= \alpha_{l,i}$. Let $\wtilde{P}_j :=\bigcap_{(l,i)\in I_j}{ \wtilde{N}_{\gamma_{l,i}}}$ if $I_j \neq \emptyset$, and $\wtilde{P}_j =\dot{N}_j$ otherwise. Since each $\dot{P}_j$ is of finite index in $P_j$ and by assumption the groups $P_j$ are residually finite, we can find finite index subgroups $\dot{P}^{(1)}_j \trianglelefteq P_j$ contained in $\wtilde{P}_j$ and disjoint from $S$. By construction any filling $\phi : G \rightarrow \ov{G}= G(N_1,\dots,N_n)$ with $N_j< \dot{P}^{(1)}_j$ of finite index satisfies the conclusions (1)-(4).

To deal with (5)-(6), note that the groups $H_1,\dots , H_k$ are peripherally separable in $(H_0, \mcal{D}_0)$, and that each group in $\mcal{D}_0$ is residually finite. Therefore, by applying Theorems \ref{dehnfilling}, \ref{existencewide}, and \ref{heightfilling}, and in the same way we constructed the groups $\dot{P}^{(1)}_j$ above, we obtain finite index subgroups $\dot{K}_{0,i} \trianglelefteq D_{0,i}$ such that for any filling $\phi_0 : H_0 \rightarrow H_0(K_{0,1},\dots , K_{0,k_0})$ with $K_{0,i}<\dot{K}_{0,i}$, we have:
\begin{itemize}
    \item[($i$)] for any $1 \leq l \leq k$ the induced filling $\phi_l$ of $H_l$ satisfies $\ker\phi_l= \ker\phi_0 \cap H_l$, and
\item[($ii$)] the height of $\corchete{\phi_0(H_1),\dots , \phi_0(H_k)}$ in $\phi_0(H_0)$ is at most the relative height of $\corchete{H_1,\dots,H_k}$ in $H_0$.
\end{itemize}
By our strong peripheral separability assumption, and by the same separability argument used to find the groups $\wtilde{N}_{\gamma_{l,i}}$, there exist finite index subgroups $\dot{P}_j \trianglelefteq P_j$ contained in $\dot{P}^{(1)}_j$, and such that $D_{0,i} \cap \dot{P}^{c_{0,i}}_{\alpha_{0,i}}<\dot{K}_{0,i}$ for any $1 \leq i \leq k_0$.

By construction, if $\phi : G \rightarrow G(N_1,\dots,N_n)$ is a peripherally finite filling with $N_i <\dot{P}_j$, then it satisfies conclusions (1)-(4). In addition, by ($i$) we have $\ker\phi \cap H_j=(\ker \phi_0 \cap H_0) \cap H_j=\ker \phi_j$ , so it also satisfies (5). Finally, note that by the discussion after equation \eqref{eq5}, the embedding $\phi_0(H_0)\hookrightarrow \phi(G)$ induces isomorphisms $\phi_0(H_j) \xrightarrow{\sim}\phi(H_j)$  for each $j$, and so the height of $\corchete{\phi_0(H_1),\dots , \phi_0(H_k)}$ in $\phi_0(H_0)$ coincides with the height of $\corchete{\phi(H_1),\dots , \phi(H_k)}$ in $\phi(H_0)$. This fact together with ($ii$) proves (6). 
\end{proof}

\section{Proof of Theorem \ref{CMVHimpliesspecial}}\label{proofoftheorem}
In this section we prove Theorem \ref{CMVHimpliesspecial}. We first recall the notion of graph of groups, referring to the work of Bass in \cite{Bass1993}.

\begin{defi}A \emph{graph of groups} is a pair $(\Gamma, \mcal{G})$ consisting of:
\begin{enumerate}\item a connected, non-empty graph $\Gamma$ with vertex set $V=V(\Gamma)$, and an oriented edge set $E=E(\Gamma)$ with an involution $e \mapsto \ov{e}$ that switches the orientation of each edge,
\item an assignment $\mcal{G} : V \sqcup E \rightarrow \textbf{Grp}$ of a group $x \mapsto G_x =\mcal{G}(x)$ for any vertex or edge $x$, such that $G_e=G_{\ov{e}}$ for any edge $e$, and
\item a set of attachment monomorphisms $\psi_e : G_e \rightarrow G_{t(e)}$ where $t(e)$ is the terminal vertex of the edge $e$.
\end{enumerate}
\end{defi}
Given a graph of groups $(\Gamma,\mcal{G})$ we consider the group $$F(\Gamma,\mcal{G})=(\ast_{v\in V}G_v) \ast F(E))/N,$$
where $F(E)$ is the free group generated by the set $E$ and $N$ is the normal subgroup generated by the relations $e^{-1}=\ov{e}$ and $e\psi_e(g)e^{-1}=\psi_{\ov{e}}(g)$ for any edge $e$ and any $g \in G_e$.
\begin{defi}The \emph{fundamental group} of $(\Gamma,\mcal{G})$ based at $v_0 \in V$ is the subgroup $$\pi_1(\Gamma, \mcal{G}, v_0)   <F(\Gamma,\mcal{G})$$
consisting of the elements of the form
$$g=g_0e_1g_1e_2 \cdots e_ng_n,$$
where $e_1, e_2,\dots, e_n$ form a \emph{circuit} in $\Gamma$ based at $v_0$ (i.e. $t(e_i)=t(\ov{e}_{i+1})$ for $1 \leq i \leq n-1$ and $t(e_n)=t(\ov{e}_1)=v_0$), and $g_i \in G_{t(\ov{e}_{i+1})}$ for any $0 \leq i \leq n$, with the convention $t(\ov{e}_{n+1})=v_0$.
\end{defi}
The isomorphism class of $\pi_1(\Gamma, \mcal{G}, v_0)$ is independent of the base-point, and the canonical maps from vertex groups into $F(\Gamma,\mcal{G})$ are injective, so we can consider any vertex group as a subgroup of $\pi_1(\Gamma, \mcal{G}, v_0)$ by means of choosing a maximal tree $T$ of $\Gamma$ containing $v_0$.
\begin{defi} We say that a group $G$ \emph{splits} as a graph of groups $(\Gamma,\mcal{G})$ if $G$ is isomorphic to $\pi_1(\Gamma, \mcal{G}, v_0)$ for some (hence any) vertex $v_0$ of $\Gamma$.
\end{defi}
The proposition below gives existence of hyperbolic and virtually special fillings for groups in $\mcal{CMVH}$, and is the key ingredient in the proof of Theorem \ref{CMVHimpliesspecial}.

\begin{prop}\label{splittingimpliessepa} Let $(G,\mcal{P}=\corchete{P_1,\dots, P_n})$ be a relatively hyperbolic group with each $P_i$ being residually finite. Suppose $G$ splits as a finite graph of groups $(\Gamma,\mcal{G})$ satisfying:
\begin{itemize}
    \item  each edge group is relatively quasiconvex in $G$ and peripherally separable,
\item if $v$ is a vertex of $\Gamma$, then the collection $\mcal{A}_v := \corchete{G_e : e \text{ attached to }v}$ is relatively malnormal in $G_v$, and
\item each vertex group is virtually special and strongly peripherally separable.
\end{itemize}
Then:
\begin{enumerate}
    \item Every relatively quasiconvex and peripherally separable subgroup of $G$ is separable.
\item Every double coset of a doubly peripherally separable pair of relatively quasiconvex and
strongly peripherally separable subgroups of $G$ is separable.
\end{enumerate}
\end{prop}
\begin{prop}\label{factorvspecial} Let $(\Gamma,\mcal{G})$ be as in the statement of Proposition \ref{splittingimpliessepa}. Then there exist finite index subgroups $\dot{P}_j \trianglelefteq P_j$ such that for any further finite index subgroups $N_j <\dot{P}_j$ there are subgroups $M_j < N_j$ (normal in $P_j$) with $\ov{\mbb{G}}=G(M_1,\dots,M_n)$
hyperbolic and virtually special.
\end{prop} 
\begin{rmk} Since we are restricting to finitely generated groups, relative quasiconvexity of edge groups implies relative quasiconvexity of vertex groups \cite[Lem.~4.9]{BigdelyWise2013Quasisplit}.
\end{rmk}
Before proving Proposition \ref{factorvspecial}, we need a lemma. Condition (2) will be used in the proof of Proposition \ref{weakseparability}.
\begin{lemma}\label{hyperbolicsplitting} Let $\ov{\mbb{G}}$ be a group splitting as a finite graph of groups $(\Gamma,\ov{\mcal{G}})$ with hyperbolic vertex groups, and edge groups quasiconvex in their corresponding vertex groups. Suppose that either:
\begin{enumerate}\item  for each vertex $v \in V=V (\Gamma)$ the collection $\ov{\mcal{P}}_v := \corchete{\ov{G}_e : e \text{ is an edge attached to }v}$ is almost malnormal in $\ov{G}_v$, or
\item $\Gamma$ is bipartite with $V(\Gamma)=V_1 \sqcup V_2$ and each edge joining vertices of $V_1$ and $V_2$, and such that for each $v \in V_1$ the collection $\ov{\mcal{P}}_v :=
\corchete{\ov{G}_e : e \text{ is an edge attached to }v}$ is almost malnormal in $\ov{G}_v$.
\end{enumerate}
Then $\ov{\mbb{G}}$ is hyperbolic and the edge groups are quasiconvex in $\ov{\mbb{G}}$.
\end{lemma}
\begin{proof}Since edge groups are quasiconvex in the vertex groups, by almost malnormality the pair $(\ov{G}_v,\ov{\mcal{P}}_v)$ is relatively hyperbolic for each $v \in V$ (resp. $v \in V_1$). Also, for $v \in V_2$ consider the trivial peripheral structure $(\ov{G}_v,\corchete{\ov{G}_v})$. With these conventions, each edge group is maximal parabolic in vertex groups of $V$ (resp. $V_1$), and no two of them are conjugate into a common vertex group $\ov{G}_v$ with $v \in V$ (resp. $v \in V_1$), unless they are finite. Therefore we are in the assumptions of \cite[Cor.~4.6]{BigdelyWise2013Quasisplit} (resp. \cite[Cor.~1.5]{BigdelyWise2013Quasisplit}), and hence $\ov{\mbb{G}}$ is hyperbolic relative to $\bigcup_{v\in V}{\ov{\mcal{P}}_v}-\corchete{\text{repeats}}$ (resp. $\corchete{\ov{G}_v \colon v \in V_2}$), and the edge groups are quasiconvex in some peripheral subgroup of $\ov{\mbb{G}}$. In both cases $\ov{\mbb{G}}$ will be hyperbolic relative to hyperbolic groups, hence hyperbolic as well (see e.g. \cite[Thm.~3.8]{Farb1998RH}), and edge groups will be quasiconvex since the peripheral subgroups are quasiconvex in $\ov{\mbb{G}}$. 
\end{proof}
\begin{proof}[Proof of Proposition \ref{factorvspecial}] Let us fix a vertex $v_0 \in V=V(\Gamma)$, an isomorphism $G \cong \pi_1(\Gamma, \mcal{G}, v_0)$, and a maximal tree $T$ of $\Gamma$ containing $v_0$ that induces embeddings of the vertex/edge groups of $(\Gamma,\mcal{G})$ into $G$. For each vertex, apply Proposition \ref{superfilling} to the collection $\mcal{A}_v$ and the group $G_v$, to find finite index subgroups $\dot{P}_j(v) \trianglelefteq P_j$ such that for any choice of peripherally finite filling kernels $\mcal{N}=\corchete{N_1,\dots, N_n}$ with $N_i<\dot{P}_i(v)$, the filling $\phi : G \rightarrow \ov{G}=G(N_1,\dots,N_n)$ satisfies:
\begin{itemize}
    \item  $\ov{G}$ is hyperbolic.
\item  $\ov{G}_v := \phi(G_v)$ is virtually special, quasiconvex (hence hyperbolic) in $\ov{G}$, and isomorphic to the image of the induced filling $\phi_v : G_v \rightarrow G_v(\mcal{N}_v)$.
\item The collection $\ov{\mcal{A}}_v$ of images under $\phi$ of groups in $\mcal{A}_v$ is almost malnormal in $\ov{G}_v$.
\item  Each $\ov{G}_e := \phi(G_e)$ in $\ov{\mcal{A}}_v$ is naturally isomorphic to the image of the filling $\phi_e : G_e \rightarrow G_e(\mcal{N}_e)$ induced by both $\phi$ and $\phi_v$ (that is, $\ker\phi_e=\ker\phi_v \cap G_e=\ker\phi \cap G_e$).
\end{itemize}
For $1 \leq j \leq n$ define $\dot{P}_j := \bigcap_{v\in V}{\dot{P}_j(v)}$, which is a finite index normal subgroup of $P_j$, and consider finite index subgroups $N_j<\dot{P}_j$ inducing the filling $\phi : G \rightarrow \ov{G}=G(N_1,\dots,N_n)$.

To construct $\ov{\mbb{G}}$, consider a new graph of groups $(\Gamma,\ov{\mcal{G}})$ with the same underlying graph $\Gamma$, with $\ov{\mcal{G}}$ assigning the group $\ov{G}_x$ to each $x \in V \sqcup E$, and with attaching maps being the inclusions $\ov{\psi}_e: \ov{G}_e \hookrightarrow \ov{G}_{t(e)}$ induced by $\phi$. Define $\ov{\mbb{G}}$ as $\pi_1(\Gamma, \ov{\mcal{G}}, v_0)$, and choose embeddings of the vertex groups according to the same maximal tree $T$.

This choice of embeddings induces commuting diagrams
\begin{equation*}
\begin{tikzcd}
G_e \arrow[r,"\phi_e"] \arrow[d,"\varphi_e"]  & \ov{G}_e \arrow[d,"\ov{\varphi}_e"] \\
G_{t(e)} \arrow[r,"\phi_{t(e)}"] & \ov{G}_{t(e)}  
\end{tikzcd}
\end{equation*}
all of them together inducing a homomorphism of graphs of groups (see e.g. Bass \cite{Bass1993})
$$\Phi : G \rightarrow \ov{\mbb{G}}$$
such that $\Phi(x)=\phi_v(x)$ for any vertex $v$ and for any $x \in G_v$.

By our choice of $\phi$, the splitting $(\Gamma,\ov{\mcal{G}})$ of $\ov{\mbb{G}}$ satisfies the assumptions of Lemma \ref{hyperbolicsplitting} (1), and so $\ov{\mbb{G}}$ is hyperbolic and the edge groups $\Phi(G_e)=\ov{G}_e$ are quasiconvex in $\ov{\mbb{G}}$. In addition, by assumption the vertex groups of $(\Gamma,\ov{\mcal{G}})$ are hyperbolic and virtually special, and so Theorem \ref{wiseQVH} implies that $\ov{\mbb{G}}$ is virtually special.

Finally, we need to show that $\ker\Phi=\genby{\genby{
\bigcup_j{M_j}}}_G$ for some $M_j \trianglelefteq P_j$ contained in $N_j$. To do this, first note that $\ker\Phi=\genby{\genby{
\bigcup_{v\in V}{\ker\phi_v}}}_G$, and that for each $v$ we have the equality $\ker\phi_v=\genby{\genby{\bigcup_{D\in \mcal{D}_j}{K_D}}}_{G_v}$, with $\mcal{D}_v$ being a peripheral structure on $G_v$ induced by $(G,\mcal{P})$ and $\corchete{K_D}_{D\in \mcal{D}_v}$ being the filling kernels induced by $\phi$.

Suppose each $D \in \mcal{D}_v$ is of the form $D= G_v \cap P^{c_D}_{i_D}$ with $1 \leq i_D \leq n$ and $c_D \in G$, and hence $K_D=H \cap D \cap N^{c_D}_{i_D}$. For $1 \leq j \leq n$, let $\mcal{D}_j$ denote the set of all $D \in \bigcup_v{\mcal{D}_v}$ such that $i_D=j$, and define $M_j := \genby{\genby{
\bigcup_{D\in \mcal{D}_j}{K_D^{{c_D}^{-1}}}}}_{P_j}$
if $\mcal{D}_j$ is non-empty, and $M_j := \corchete{1}$ otherwise. Note that $M_j<N_j$ for each $j$. We claim that $\genby{\genby{
\bigcup_j{M_j}}}_G=\genby{\genby{\bigcup_{v\in V}{\ker\phi_v}}}_G$ for these choices of $M_j$.

For the inclusion ``$\subset$", it is enough to show $M_j \subset \genby{\genby{\bigcup_{v\in V}{\ker\phi_v}}}_G$ for any $j$, which holds because when $\mcal{D}_j$ is non-empty, we have
\small{\begin{equation*}M_{j}=\genby{\genby{\bigcup_{D \in \mcal{D}_{j}} {K_{D}^{c_{D}^{-1}}}}}_{P_{j}} \subset\genby{\genby{\bigcup_{D \in \mcal{D}_{j}}{K_{D}^{c_{D}^{-1}}}}}_{G}=\genby{\genby{\bigcup_{D\in \mcal{D}_j}{K_{D}}}}_{G} \subset\genby{\genby{\bigcup_{v\in V}{\bigcup_{D \in \mcal{D}_{v}}{K_{D}}}}}_{G} \subset\genby{\genby{\bigcup_{v \in V} {\ker\phi_{v}}}}_{G}.
\end{equation*}}\normalsize
On the other hand, for any $v \in V$ we obtain 
\small{\begin{equation*}\ker\phi_{v}=\genby{\genby{\bigcup_{D \in \mcal{D}_{v}}{ K_{D}}}}_{G_{v}}\subset\genby{\genby{\bigcup_{D \in \mcal{D}_{v}}{K_{D}}}}_{G} \subset\genby{\genby{\bigcup_{D \in \mcal{D}_{v}}{K_{D}^{c_{D}^{-1}}}}}_{G} \subset\genby{\genby{\bigcup_{j}{\bigcup_{D \in \mcal{D}_{j}}{K_{D}^{c_{D}^{-1}}}}}}_{G}=\genby{\genby{\bigcup_{j} {M_{j}}}}_{G},
\end{equation*}}\normalsize
which proves ``$\supset$". 
\end{proof}

\begin{proof}[Proof of Proposition \ref{splittingimpliessepa}] To prove conclusion (1), let $H$ be a relatively quasiconvex and peripherally separable subgroup of $G$ and consider $a \in G\backslash H$, which we want to separate from $H$ in a finite quotient of $G$. Instead of using Dehn filling directly on $H$, we will use \cite[Cor.~6.3]{SageevWise2015Cores} and the peripheral separability of $H$ to find a fully relatively quasiconvex subgroup $\hat{H}<G$ containing $H$ and such that $a\notin \hat{H}$, so now it is enough to separate $a$ from $\hat{H}$. By Proposition \ref{dehnfilling} (1)-(3) there exists a finite set $S\subset\bigcup{\mcal{P}}\backslash \corchete{1}$ such that for any $(\hat{H},S)$-wide filling $\phi : G \rightarrow \ov{G}=G(N_1,\dots,N_n)$ with $S \cap \bigcup_j{N_j} = \emptyset$, the pair $(\ov{G},\ov{\mcal{P}})$ is relatively hyperbolic with $\ov{\hat{H}}=\phi(\hat{H})$ fully relatively quasiconvex in $(\ov{G},\ov{\mcal{P}})$ and $\phi(a) \notin \ov{\hat{H}}$. Since the groups $P_j$ are residually finite, there exist finite index subgroups $Q_j \trianglelefteq P_j$ with each $Q_j$ being disjoint from $S$. Also, since $\hat{H}$ is fully relatively quasiconvex, it is peripherally separable and by Lemma \ref{existencewide} there exist finite index subgroups $\dot{N}_j \trianglelefteq P_j$ such that any filling $G \rightarrow G(N_1,\dots,N_n)$ with $N_j<\dot{N}_j$ for each $j$ is $(\hat{H}, S)$-wide.

Let $\dot{P}_j \trianglelefteq P_j$ be given by Proposition \ref{factorvspecial}, and set $N_j := \dot{N}_j \cap \dot{P}_j \cap Q_j \trianglelefteq P_j$. By construction, $\phi : G \rightarrow \ov{G}=G(N_1,\dots,N_n)$ is an $(\hat{H}, S)$-wide filling factoring through the $(\hat{H}, S)$-wide filling $\Phi : (G,\mcal{P}) \rightarrow (\ov{\mbb{G}}, \ov{\mbb{P}})$, with $(\ov{\mbb{G}}, \ov{\mbb{P}})$ relatively hyperbolic, $\Phi(\hat{H})$ fully relatively quasiconvex in $(\ov{\mbb{G}}, \ov{\mbb{P}})$ and $\Phi(a) \notin \Phi(\hat{H})$ (here $\ov{\mbb{P}}$ denotes the set of images of groups in $\mcal{P}$ under $\Phi$). Also, $\ov{\mbb{G}}$ is hyperbolic and virtually special, and therefore, to solve the problem it is enough to show that $\Phi(\hat{H})$ is separable in $\ov{\mbb{G}}$.

By Theorem \ref{haglundwiseseparability} it suffices to show that $\Phi(\hat{H})$ is quasiconvex in $\ov{\mbb{G}}$, i.e. that $\Phi(\hat{H})$ is relatively quasiconvex in $(\ov{\mbb{G}}, \emptyset)$. But $(\ov{\mbb{G}}, \ov{\mbb{P}})$ is an extended peripheral structure for $(\ov{\mbb{G}}, \emptyset)$, and so by \cite[Thm.~1.3]{Yang2014Periph} it is enough to show that $\Phi(\hat{H}) \cap \Phi(P)$ is quasiconvex in $\ov{\mbb{G}}$ for any peripheral subgroup $P \subset G$. This last statement follows from the quasiconvexity of $\Phi(P)$ in $\ov{\mbb{G}}$ \cite[Lem.~4.15]{DrutuSapir2005Treegraded}, and the full relative quasiconvexity of $\Phi(\hat{H})$ in $\ov{\mbb{G}}$, since quasiconvexity is stable under finite index inclusions.

Now, let $H, L<G$ be relatively quasiconvex and strongly peripherally separable subgroups such that $H, L$ is doubly peripherally separable, and consider $a \in G\backslash HL$. Conclusion (1) implies that the peripheral subgroups of $G$ are separable, as well as any of their finite index subgroups. In addition, by our strong peripheral separability assumption, any finite index subgroup of $H$ or $L$ will be peripherally separable, hence also separable. Therefore, $H$ and $L$ satisfy the assumptions of Corollary \ref{almostfullydoublecoset}, a so there are fully relatively quasiconvex subgroups $\hat{H}$ and $\hat{L}$ with $H \cap\hat{H}<H$ and $L \cap \hat{L}<L$ of finite index, and such that $a \notin H\hat{H}\hat{L}L$.

Thus, to prove conclusion (2) we only need to verify that $H\hat{H}\hat{L}L$ is separable, and since this set is a finite union of translates of $\hat{H}\hat{L}$, it is enough to show that double cosets of fully relatively quasiconvex subgroups are separable. This can be done in the same way as in the proof of conclusion (1), by using Proposition \ref{factorvspecial} together with Theorem \ref{doublecosetfilling} and the separability of double cosets of quasiconvex subgroups of hyperbolic virtually special groups \cite[Thm.~1.1]{Minasyan2006GFERF}. Details are left to the reader.
\end{proof}
\begin{rmk}The last part in the proof of conclusion (2) can also be deduced directly from \cite[Cor.~4.4]{McClellan2019}.
\end{rmk}
\begin{proof}[Proof of Theorem \ref{CMVHimpliesspecial}] Let $(G,X)$ be a cubulated group in $\mcal{CMVH}$. We will prove that $(G,X)$ is virtually special by induction on the minimal number of operations $(1)-(3)$ used in a description of $(G,X)$ (see Definition \ref{defiCMVH}), where clearly $(\corchete{1}, X)$ is virtually special if $X$ is a finite $\CAT{0}$ cube complex.

First, suppose that $H<G$ is of finite index with $(H, X) \in \mcal{CMVH}$. Our inductive assumption implies that $(H, X)$ is virtually special, so clearly $(G,X)$ is also virtually special.

Now, let $G$ be splitting as a finite graph of groups $(\Gamma,\mcal{G})$ such that:
\begin{itemize}
\item if $u \in V(\Gamma) \sqcup E(\Gamma)$ then $G_u$ has convex core $X_u \subset X$, 
\item if $v \in V(\Gamma)$ then $(G_v, X_v) \in \mcal{CMVH}$ (and hence $(G_v, X_v)$ is virtually special by our inductive assumption and Lemma \ref{consistentcompatible}), and
\item $\mcal{A}_v$ is relatively malnormal in $G_v$ for any $v \in V(\Gamma)$.
 \end{itemize}
To show that $(G,X)$ is virtually special we will use Theorem \ref{haglundwisedoublecoset}, but first we need to show that $G$ satisfies the hypotheses of Proposition \ref{splittingimpliessepa}. This follows because by our compatibility assumption the peripheral subgroups are residually finite, and because by Lemma \ref{compimpliesstrongly}, all convex subgroups of $(G,X)$ are strongly peripherally separable.

Therefore, $(G,X)$ satisfies Proposition \ref{splittingimpliessepa}, so any relatively quasiconvex and peripherally separable subgroup of $G$ is separable, as well as any double coset of a doubly peripherally separable pair of relatively quasiconvex and strongly peripherally separable subgroups. In particular, the wall stabilizers of $(G,X)$ are separable, so by Theorem \ref{haglundwisedoublecoset} it is enough to show that the pair $G_{W_1},G_{W_2}$ is doubly peripherally separable for any pair of walls if $W_1, W_2 \subset X$. To prove this, let $g_1,g_2\in G$ and let $P<G$ be a peripheral subgroup so that $G_{{g_1}^{-1}W_1}\cap P=(G_{W_1})^{g_1} \cap P$ and $G_{{g_2}^{-1}W_2}\cap P=(G_{W_2})^{g_2} \cap P$ are both infinite. By Theorem \ref{sageevwise} choose a convex core $Z \subset X$ for $P$ such that ${g_1}^{-1}W_1 \cap Z$ and ${g_2}^{-1}W_2 \cap Z$ are both non-empty, so that $(P,Z)$ is virtually special by Proposition \ref{independencecore}. The subgroups $G_{{g_1}^{-1}W_1}\cap P$ and $G_{{g_2}^{-1}W_2}\cap P$ are then convex in $(P,Z)$ by Remark \ref{wallsubcomplex}, and the conclusion follows from Theorem \ref{doublecosetconvexsep}.
\end{proof}

\section{Construction of the complex with finite walls}\label{constructionofthecomplex}
Now we start the proof of Theorem \ref{relhypimpliesCMVH}, by induction on the dimension of $X$ (the 0-dimensional
case is evident). Henceforth, let $(G,X)$ be a cubulated relatively hyperbolic group with compatible virtually special peripherals, and assume that:

\emph{$(H, \dot{Y}) \in  \mcal{CMVH}$ for every cubulated and relatively hyperbolic group $(H, Y)$ with compatible virtually special peripheral subgroups and such that $\dim{Y}<\dim{\dot{X}}=\dim{X}$}.

From Corollary \ref{vspecialsubdivision}, $(G,\dot{X})$ is also hyperbolic relative to compatible virtually special subgroups.

Fix $R \geq \delta+2\sqrt{\dim X}$, and let $W_1,\dots, W_m$ be a complete set of representatives of $G$-orbits of walls in $\dot{X}$. The goal of this section is to prove the following.
\begin{thm}\label{finitewalls} There exists a torsion-free normal subgroup $K\trianglelefteq G$ such that the quotient cube
complex $\mcal{X} := \dot{X}/K$ satisfies:
\begin{enumerate}
\item $G$ acts cocompactly on $\mcal{X}$.
\item All walls of $\mcal{X}$ are finite.
\item  If $W$ is a wall of $\dot{X}$ then the $R$-neighborhood $N_R(W)$ quotiented by $K \cap G_W$ embeds in $\mcal{X}$. In particular, all walls of $\mcal{X}$ are embedded, and distinct walls in $\dot{X}$ which are less than $R$ apart map to distinct walls in $\mcal{X}$.
\end{enumerate}
\end{thm}
We will need the following weak separability result for virtually special quasiconvex subgroups of hyperbolic groups, which is the main result in the Appendix of \cite{Agol2012VirtualHaken}.
\begin{thm}[Agol-Groves-Manning]\label{agolgrovesmanning} Let $G$ be a hyperbolic group and $H<G$ be a quasiconvex virtually special subgroup. Then for any $g \in G\backslash H$ there is a hyperbolic group $\mcal{G}$ and a surjective homomorphism $\psi : G \rightarrow \mcal{G}$ such that $\psi(g) \notin \psi(H)$ and $\psi(H)$ is finite.
\end{thm}
\begin{proof}[Proof of Theorem \ref{finitewalls}] For every $1 \leq i \leq m$, consider a set $\mcal{A}_i$ of representatives of double cosets $G_{W_i}gG_{W_i}$ with $g \notin G_{W_i}$ and $d(W_i, gW_i) \leq 2R$. Also, choose a set $\mcal{T}\subset G$ of representatives of conjugacy classes of non-trivial torsion elements of $G$. These sets are finite by Lemma \ref{finiteclasses}. Set $\mcal{H}=\corchete{G_{W_1},\dots, G_{W_m}}$.

As a first step, since each subgroup $G_{W_i}$ is convex in $(G,\dot{X})$, Lemma \ref{compimpliesstrongly} and Proposition \ref{superfilling} imply that we can find a peripherally finite Dehn filling $\phi : G \rightarrow \ov{G}=G(\dot{P}_1,\dots, \dot{P}_n)$ so that $\ov{G}$ is hyperbolic, each $\ov{G}_i:= \phi(G_{W_i})$ is quasiconvex, virtually special and disjoint from $\phi(\mcal{A}_i)$, and $1 \notin \phi(\mcal{T})$.

Our second step is to find a quasiconvex and virtually special subgroup $\ov{H}< \ov{G}$ such that $\ov{H} \cap \ov{G}_i< \ov{G}_i$ is finite index for $1 \leq i \leq m$. We will do this by inducting on $k$, and finding for each $1 \leq j \leq m$ a quasiconvex and virtually special subgroup $\ov{H}_j<\ov{G}$ such that $\ov{H}_j \cap \ov{G}_i <\ov{G}_i$ is finite index for $1 \leq i \leq j$. For the base case we choose $\ov{H}_1=\ov{G}_1$.

Suppose we have found $\ov{H}_j$, and consider the intersection $L=\ov{H}_j \cap \ov{G}_{j+1}$. By Gitik's ping-pong theorem \cite[Thm.~1]{Gitik1999PingPong}, there exists a finite set $F \subset (\ov{H}_j \cup \ov{G}_{j+1})\backslash L$ such that if $\wtilde{H}_j<\ov{H}_j$ and $\wtilde{G}_{j+1}<G_{j+1}$ are finite index subgroups with $\wtilde{H}_j \cap \wtilde{G}_{j+1}=L$, and $\wtilde{H}_j$, $\wtilde{G}_{j+1}$ disjoint from $F$, then $\genby{\wtilde{H}_j \cup \wtilde{G}_{j+1}}$ is quasiconvex in $\ov{G}$ and isomorphic to $\wtilde{H}_j \ast_L \wtilde{G}_{j+1}$. The existence of finite index subgroups $\wtilde{H}_j$ and $\wtilde{G}_{j+1}$ disjoint from $F$ is guaranteed by subgroup separability, since by assumption $\ov{H}_j$ , $\ov{G}_{j+1}$ are virtually special and $L$ is quasiconvex in both groups, so we can apply Theorem \ref{haglundwiseseparability}.

Define $\ov{H}_{j+1} := \genby{\wtilde{H}_j \cup \wtilde{G}_{j+1}}$. By construction $\ov{H}_{j+1} \cap \ov{G}_{W_i}$ is finite index in $\ov{G}_{W_i}$ for $1 \leq i \leq j+1$. The virtual specialness of $\ov{H}_{j+1}$ follows from Theorem \ref{wiseQVH} and the isomorphism $\ov{H}_{j+1} \cong \wtilde{H}_j \ast_L \wtilde{G}_{j+1}$, where $\wtilde{H}_j$ and $\wtilde{G}_{j+1}$ are virtually special (finite index subgroups of virtually special groups), and $L$ is quasiconvex in $\ov{H}_{j+1}$ since both $\wtilde{H}_{j}$ and $\wtilde{G}_{j+1}$ are quasiconvex subgroups of $\ov{G}$. The induction is then complete, and set $\ov{H}=\ov{H}_m$.

Before the end of the proof, and after possibly replacing $\ov{H}$ by a finite index subgroup, we can assume that $\ov{H}$ does not intersect $\phi(\mcal{T})$. This is because $\ov{H}$ is virtually special, hence residually finite, and $\phi(\mcal{T})$ is finite. This modification does not affect the expected properties for $\ov{H}$.

For $1 \leq i \leq m$, let $\ov{N}_i=\ov{H}\cap \ov{G}_i$, and define $N_i:= \phi^{-1}(\ov{N}_i) \cap G_{W_i}<G_{W_i}$. Note that $N_i$ is finite index in $G_{W_i}$ for each $i$, so $N_i$ acts cocompactly on $W_i$, and by Lemma \ref{finiteclasses} (2) there is a finite set $\wtilde{\mcal{A}}_i$ of representatives of double cosets $N_igN_i$ with $g \notin N_i$ and $d(gW_i, W_i) \leq 2R$. We claim that $\phi(\wtilde{\mcal{A}}_i) \cap \ov{N}_i=\emptyset$ for each $i$. Indeed, let $a \in \wtilde{\mcal{A}}_i$, and suppose first that $a \notin G_{W_i}$. In that case, $a=g_1bg_2$ for some $b \in \mcal{A}_i$ and $g_1, g_2 \in G_{W_i}$, implying $\phi(a)=\phi(g_1)\phi(b)\phi(g_2) \notin \ov{G}_i\supset \ov{N}_i$ since by construction $\phi(b) \notin \ov{G}_i$. In the case $a \in G_{W_i}$, the conclusion follows from the definition of $N_i$. Proven our claim, the groups $\ov{N}_i$ are quasiconvex in $\ov{H}$, hence separable by Theorem \ref{haglundwiseseparability},
and we can find finite index subgroups $\hat{H}_i<\ov{H}$ such that $\hat{H}_i \cap \ov{G}_i =\ov{N}_i$ and $\hat{H}_i \cap \phi(\wtilde{\mcal{A}}_i)=\emptyset$ for all $i$. Note that the groups $\hat{H}_i$ are quasiconvex in $\ov{G}$ and virtually special.

Now we construct $K$. For each $1 \leq i \leq m$ and each $g \in \phi(\wtilde{\mcal{A}}_i)$ we use Theorem \ref{agolgrovesmanning} to find a quotient homomorphism $\varphi_g$ of $\ov{G}$ with $\varphi_g(\hat{H}_i)$ finite and $\varphi_g(g) \notin \varphi_g(\hat{H}_i)$. Similarly, for each $g \in \phi(\mcal{T})$ we construct a quotient homomorphism $\tau_g$ of $\ov{G}$ with $\tau_g(g) \notin \tau_g(\ov{H})$ and $\tau_g(\ov{H})$ finite.

Define $K:=\phi^{-1}\left(\left(\bigcap_{g \in \cup_{i}{ \wtilde{\mcal{A}}_{i}}} {\ker\varphi_{g}}\right) \cap\left(\bigcap_{g \in \mathcal{T}}{\ker\tau_{g}}\right)\right) \trianglelefteq G$ and $\mcal{X} := \dot{X}/K$. Note that $K$ is disjoint from $\mcal{T}$, hence torsion-free, so $\mcal{X}$ is a cube complex. Clearly $G$ acts cocompactly on $\mcal{X}$, so (1) holds, and by construction $K \cap G_W$ is finite index in $G_W$ for each wall $W$, implying (2). Finally, if $x, y \in N_R(W_i)$ and $k \in K$ are such that $x=ky$, then $d(kW_i, W_i) \leq 2R$. So if $k \notin K_i$, there is some $g \in \wtilde{\mcal{A}}_i$ with $k=h_1gh_2$ for some $h_1, h_2 \in N_i$. But this is impossible since it would imply
\begin{equation*}
    \varphi_{\phi(g)}(\phi(g))=(\varphi_{\phi(g)} \circ \phi)(h_{1}^{-1} k h_{2}^{-1})=(\varphi_{\phi(g)} \circ \phi)(h_{1}^{-1} h_{2}^{-1}) \in \varphi_{\phi(g)}(\ov{N}_{i}) \subset \varphi_{\phi(g)}(\hat{H}_{i}).
\end{equation*}
Therefore $k \in N_i \subset G_{W_i}$, and the map $N_R(W_i)/(K\cap G_{W_i}) \rightarrow \mcal{X}$ is an embedding for all $1 \leq i \leq m$. Property (3) then follows from the normality of $K$ and by considering translates of the $W_i$.
\end{proof}

The point of working with $\dot{X}$ instead of $X$, is that for every wall $W$ of $\dot{X}$, the subgroup $G_W$ does not exchange the sides of $W$.
This allows us to find a $G$-equivariant \emph{co-orientation} on the walls of $\dot{X}$, that is, a labelling $W^+$ and $W^-$ for the half-spaces of each wall $W$ of $\dot{X}$, such that $(gW)^{\pm}=g(W^{\pm})$ for any $g\in G$ and for any wall $W\subset \dot{X}$.

Let $q : \dot{X} \rightarrow \mcal{X}$ denote the quotient map from Theorem \ref{finitewalls}. This map will send a vertex $x$ (resp. an edge $e$ and wall $W$) of $\dot{X}$ to a vertex $\ov{x}$ (resp. an edge $\ov{e}$ and wall $\ov{W}$).

\section{Coloring walls in $\mcal{X}$}\label{coloringwalls}
Now we proceed to color the walls of $\mcal{X}$ , in the same way as in \cite[Sec.~5]{Agol2012VirtualHaken}.
\begin{defi} Let $\Gamma(\mcal{X})$ be the simplicial graph with vertices the walls of $\mcal{X}$ and with an edge joining the walls $\ov{W}_1$ and $\ov{W}_2$ if and only $d(\ov{W}_1, \ov{W}_2) \leq R$, with $d$ being the induced distance on $\mcal{X}$ and $R$ as in Theorem \ref{finitewalls}.
\end{defi}
There is a natural action of $G$ on $\Gamma(\mcal{X})$, and since $\mcal{X}$ is locally finite and with finite walls, there are only finitely many $G$-orbits of vertices in $\Gamma(\mcal{X})$. This implies that there exists some $k$ such that the degree of any vertex of $\Gamma(\mcal{X})$ is bounded above by $k$, and also that $G$ acts cocompactly on $\Gamma(\mcal{X})$.

\begin{defi} A \emph{coloring} of $\Gamma(\mcal{X})$ is a map $c : V (\Gamma(\mcal{X})) \rightarrow \corchete{1,\dots, k+1}$ such that if $\ov{W}_1, \ov{W}_2 \in V(\Gamma(\mcal{X}))$ are adjacent walls, then $c(\ov{W}_1)\neq c(\ov{W}_2)$. Let $C_{k+1}(\Gamma(\mcal{X}))$ denote the set of colorings, which is non-empty since vertices of $\Gamma(\mcal{X})$ have degree $\leq k$.
\end{defi}
The action of $G$ on $\Gamma(\mcal{X})$ induces an action on $C_{k+1}(\Gamma(\mcal{X}))$ via pullback $g : c \mapsto gc := c \circ g^{-1}$ for each $g \in G$.

We define several equivalence classes related to $C_{k+1}(\Gamma(\mcal{X}))$, following the notation of \cite[
Def.~6.2]{Shepherd2019AgolsCubulations}.
\begin{enumerate}
\item If $W$ is a wall of $\dot{X}$, define the equivalence class $[c]_W$ of $c \in C_{k+1}(\Gamma(\mcal{X}))$ as
$$[c]_W :=
\corchete{c'\in C_{k+1}(\Gamma(\mcal{X})) \colon c=c'
\text{ on the ball of radius }c(\ov{W}) \text{ in }\Gamma(\mcal{X}) \text{ centered at }\ov{W}}.$$
\item  If $e$ is an edge of $\dot{X}$ dual to the wall $W$, then $[c]_e := [c]_W$ for any $c \in C_{k+1}(\Gamma(\mcal{X}))$.
\item If $x$ is a vertex of $\dot{X}$ we define $[c]_x:=\bigcap{\corchete{[c]_e\colon e \text{ incident to }x}}$.
\item  We also define equivalence classes on $V(\dot{X})\times C_{k+1}(\Gamma(\mcal{X}))$ and $E(\dot{X})\times  C_{k+1}(\Gamma(\mcal{X}))$ according to $[e, c] := \corchete{e}\times [c]_e$ and $[x, c] := \corchete{x} \times [c]_x$.
\end{enumerate}
There are natural actions of $G$ on these sets of equivalence classes, given by $g[e, c] := [ge, gc]$ and $g[x, c] := [gx, gc]$. For each edge $e$, the class $[-]_e$ depends only on the colors of vertices in some $(k+1)$-ball of $\Gamma(\mcal{X})$, so there are only finitely many equivalence classes $[-]_e$, and similarly for $[-]_x$. Since there are finitely many $G$-orbits of edges and vertices in $\dot{X}$, there are only finitely many $G$-orbits of equivalence classes on $E(\dot{X})  \times C_{k+1}(\Gamma(\mcal{X}))$ and $V(\dot{X})\times C_{k+1}(\Gamma(\mcal{X}))$.

\section{Cubical polyhedra and the gluing construction}\label{cubicalpolyhedra}
We introduce the main construction used to prove Theorem \ref{relhypimpliesCMVH}. As in \cite[Sec.~7 \& 8]{Shepherd2019AgolsCubulations}, we will mainly work with the universal covers instead of the ``cubical polyhedra" used in \cite{Agol2012VirtualHaken}, so our notation will be similar to the one used in \cite{Shepherd2019AgolsCubulations}.

Inductively, we will construct non-empty sets $\mcal{V}_{k+1},\dots, \mcal{V}_0$, where each $\mcal{V}_j$ is a finite collection of triplets $(Z,H,(c_x))$, where
\begin{itemize}
    \item  $Z \subset \dot{X}$ is a non-empty intersection of half-spaces (thus closed and convex),
\item for each vertex $x \in Z$ we have a coloring $c_x \in C_{k+1}(\Gamma(\mcal{X}))$,
\item  $H<G$ acts freely and cocompactly on $Z$, and $c_{hx}=hc_x$ for each $h \in H$ and vertex $x \in Z$.
\end{itemize}
The subset $Z \subset \dot{X}$ is not a subcomplex of $\dot{X}$, but $H$ also acts cocompactly on its cubical neighborhood $\mcal{N}(Z)$. Since $Z$ is intersection of half-spaces, the vertices inside $Z$ span a convex subcomplex. The cubical neighbourhood of this subcomplex is therefore convex, and also equals $\mcal{N}(Z)$, so we have proven:
\begin{lemma}\label{nbhdconvex} $\mcal{N}(Z)$ is a convex subcomplex of $(G,\dot{X})$, hence a convex core for $H$.
\end{lemma}

We permit $\mcal{V}_j$ to contain duplicates of some triplets, and sometimes we will write $Z\in \mcal{V}_j$ instead of $(Z,H,(c_x))\in \mcal{V}_j$, and also $(Z, H,(c_x); \alpha) \in \mcal{V}_j$ to explicit that there are exactly $\alpha \in \N$
duplicates of $(Z, H,(c_x))$ in $\mcal{V}_j$.

\begin{defi}\label{defVj}We require each triplet $(Z,H,(c_x))\in \mcal{V}_j$ to satisfy four conditions:
\begin{enumerate}
    \item If $e \in E(\dot{X})$ joins vertices $x, y \in Z\in \mcal{V}_j$, then $[e, c_x]=[e, c_y]$.
\item If $e \in E(\dot{X})$ joins the vertices $x \in Z\in \mcal{V}_j$ and $y \in \dot{X}$, then $y \in Z$ if and only if $c_x(\ov{W(e)}) > j$.
\item If $(Z,H,(c_x))\in \mcal{V}_j$, then $(H, \mcal{N}(Z))\in \mcal{CMVH}$, when $H$ is endowed with the induced peripheral structure (see Lemma \ref{consistentcompatible}).
\item For $e \in E(\dot{X})$ with endpoints $x_{+} \in W(e)^+$ and $x_- \in W(e)^-$, and $c \in C_{k+1}(\Gamma(\mcal{X}))$, define
$$\mcal{V}_j^{\pm}(e, c) :=\corchete{(H \cdot x, Z)\colon x\in Z\in \mcal{V}_j, \text{ and }\exists g \in G \text{ s.t. }gx=x_{\pm}, [e, gc_x]=[e, c]},$$
where duplicates of $Z\in \mcal{V}_j$ are counted separately. The collection $\mcal{V}_j$ must satisfy the \emph{Gluing Equations}
$$|\mcal{V}_j^+(e,c)|=|\mcal{V}_j^-(e,c)|$$
for any $e \in E(\dot{X})$ and $c \in C_{k+1}(\Gamma(\mcal{X}))$.
\end{enumerate}
\end{defi}
\begin{rmk}\label{rmkcoloring}By Property (1), for an edge $e$ intersecting $Z$ we can consider a coloring $c_e \in C_{k+1}(\Gamma(\mcal{X}))$ (in fact an equivalence class), such that if $e$ is incident to $x \in Z$ then $[e, c_e]=[e, c_x]$.
\end{rmk}
Let us see how the existence of $\mcal{V}_0$ implies Theorem \ref{relhypimpliesCMVH}. Consider an arbitrary triplet $(Z, H,(c_x))$
in $\mcal{V}_0$. By conditions (1)-(2) of $\mcal{V}_0$, any vertex of $\dot{X}$ is contained in $Z$, and since $Z$ is intersection of half-spaces, we must have $\mcal{N}(Z)=Z=\dot{X}$. But then $H<G$ acts cocompactly on $\dot{X}$, implying that $H$ is of finite index in $G$. Condition (3) implies that $(H, \dot{X})\in  \mcal{CMVH}$, so $(G,\dot{X})$ is also in $\mcal{CMVH}$.

The rest of the paper concerns the inductive construction of the sequence $\mcal{V}_{k+1}, \mcal{V}_k,\dots, \mcal{V}_0$. In the hyperbolic case, the existence of $\mcal{V}_{k+1}$ was given by Agol by means of an ingenious argument regarding invariant measures on $\Gamma(\mcal{X})$ \cite[Sec.~7]{Agol2012VirtualHaken} (see also \cite[Lem.~7.1]{Shepherd2019AgolsCubulations}). Definition \ref{defVj} (3) differs from that in \cite{Shepherd2019AgolsCubulations} since it uses $\mcal{CMVH}$ rather than $\mcal{QVH}$, but  the proof of \cite[Lem.~7.1]{Shepherd2019AgolsCubulations} still applies to $\mcal{CMVH}$ since the groups $H$ are finite. Therefore we obtain:

\begin{prop}\label{agolcoloring} There exists $\mcal{V}_{k+1}$ satisfying all the conditions (1)-(4) of Definition \ref{defVj}.
\end{prop}
In the next sections we will need to modify our collections $\mcal{V}_j$, and for that we will use the following lemma: 
\begin{lemma}\label{virtmodif}Let $\mcal{V}_j$ consist of the weighted triplets $(Z, H,(c_x)); \alpha_Z)$ and for each $Z$ consider a finite index normal subgroup $H_0\trianglelefteq H$ of index $i_Z$. Then after replacing each $(Z, H,(c_x); \alpha_Z)$ by $(Z, H_0,(c_x);(\prod_{Z'\neq Z}{i_{Z'}})\alpha_Z)$, the collection $\mcal{V}_j$ still satisfies properties (1)-(4).
\end{lemma}
\begin{proof}Properties (1) and (2) are immediate. To verify (4), let $e$ be an edge of $\dot{X}$, $c \in C_{k+1}(\Gamma(\mcal{X}))$ be a coloring, and let $\wtilde{\mcal{V}}_j^{\pm}(e, c)$ be the set of pairs $(H_0 \cdot x, Z)$ such that $(H \cdot x, Z)$ is in $\mcal{V}_j^{\pm}(e, c)$. The contribution of a triplet $(Z, H_0,(c_x))$ to $\wtilde{\mcal{V}}_j^{\pm}(e, c)$ is $i_Z$ times the contribution of a triplet $(Z, H,(c_x))$ to $\mcal{V}_j^{\pm}(e, c)$. Let $\wtilde{C}^{\pm}_Z$ and $C^{\pm}_Z$ denote these contributions, respectively. Then if we choose $\wtilde{\alpha}_Z:=\prod_{Z'\neq Z}{i_{Z'}}\alpha_Z$, we have
$$\left|\wtilde{\mathcal{V}}_{j}^{+}(e, c)\right|=\sum_{Z} \wtilde{\alpha}_{Z} \wtilde{C}_{Z}^{+}=\left(\prod_{Z} i_{Z}\right) \sum_{Z} \alpha_{Z} C_{Z}^{+}=\left(\prod_{Z} i_{Z}\right) \sum_{Z} \alpha_{Z} C_{Z}^{-}=\sum_{Z} \wtilde{\alpha}_{Z} \wtilde{C}_{Z}^{-}=\left|\wtilde{V}_{j}^{-}(e, c)\right|,$$
so the gluing equations are also satisfied by the modified $\mcal{V}_j$. Finally, property (3) follows from the lemma below that will also be used in Section \ref{constructingVj-1}.
\end{proof}
\begin{lemma}\label{finiteindexCMVH}
If $(H,Y)\in \mcal{CMVH}$ and $H_0 \trianglelefteq H$ is a finite index normal subgroup, then $(H_0,Y)\in \mcal{CMVH}$.
\end{lemma}
\begin{proof}
The lemma follows by induction on the minimal number of operations $(1)-(3)$ used in a description of $(H,Y)$ as a group in $\mcal{CMVH}$, after noting that finite index subgroups of convex subgroups are convex, and that if $(H,Y)$ splits as a graph of groups satisfying the properties of condition (3) in Definition \ref{defiCMVH}, then the induced splitting of $(H_0,Y)$ also satisfies (3) when $H_0$ is considered with the induced peripheral structure, see e.g.~\cite[Prop.~3.18]{AgolGrovesManning2016MSQT}.
\end{proof}
\begin{defi}Any change of $\mcal{V}_j$ by first considering finite index subgroups $H_0<H$ for each $(Z, H,(c_x))$ and then duplicating the triplets $(Z, H_0,(c_x))$ as in the previous lemma will be called a \emph{virtual modification} of $\mcal{V}_j$. 
\end{defi}

\section{Boundary walls and a graph of groups}\label{boundarywallsandgraphofgroups}
For this section fix $(Z,H,(c_x))\in \mcal{V}_j$. We will introduce the main definitions that will be used in Section \ref{constructingVj-1} to define $\mcal{V}_{j-1}$ from $\mcal{V}_j$.
\begin{defi} A \emph{boundary wall} of $Z$ is a wall $W$ dual to an edge $e$ crossing out of $Z$. By property (2) and Remark \ref{rmkcoloring}, $W$ is a boundary wall if and only if $W=W(e)$ for an edge $e$ intersecting $Z$ and $c_e(\ov{W(e)}) \leq j$.
\end{defi}
The next lemma is implicit in \cite[p1062]{Agol2012VirtualHaken} and is stated as the \emph{Zipping Lemma} in \cite[Lem.~8.4]{Shepherd2019AgolsCubulations}.
\begin{lemma}\label{zipping} If $W=W(e_1)=W(e_2)$ is a boundary wall of $Z$ with $e_1, e_2$ edges crossing out of $Z$, then $[c_{e_1}]_W=[c_{e_2}]_W$.
\end{lemma}
\begin{rmk}By the lemma above, the color $c_e(\ov{W})$ is independent of the choice of $e$, and should be thought as the \emph{color} of the boundary wall $W$.
\end{rmk}
\begin{defi}Let $W$ be a boundary wall of $Z$ with color $j$ in the sense of the previous remark. We say that $W$ is a \emph{$j$-boundary wall} of $Z$, and that $P(W) := W \cap Z$ is the \emph{portal of $W$ leading to $Z$}. If an edge $e$ dual to $W$ crosses out of $Z$, then we say that $e$ is \emph{dual} to $P(W)$. Let $\partial_jZ$ be the union of all portals leading to $Z$.
\end{defi}
The next lemma is \cite[Lem.~8.6]{Shepherd2019AgolsCubulations}.
\begin{lemma}\label{notincident}A vertex in $Z\in \mcal{V}_j$ cannot be incident to distinct edges dual to $j$-boundary walls.
\end{lemma}
\begin{defi}For a wall $\ov{W}$ in $\mcal{X}$ and $c \in C_{k+1}(\Gamma(\mcal{X}))$, let $B(\ov{W}, c) := \ov{W} \cap c^{-1}([1,j])$ be the intersection of $\ov{W}$ with other walls in $\mcal{X}$ colored $\leq j$ by $c$ ($j$ is fixed in this section). Define $\ov{W}$ \emph{split along} $c$ by $\ov{W}- c := \ov{W}- B(\ov{W} , c)$, and for a vertex $\ov{x}$ in $\ov{W}$, let $(\ov{W}-c)(\ov{x})$ denote the component of $\ov{W}-c$ containing $\ov{x}$.
\end{defi}
\begin{lemma}[cf.~Lem.~8.9, \cite{Shepherd2019AgolsCubulations}]\label{covering} Let $W$ be a $j$-boundary wall with portal $P=Z \cap W$ and let $e$ be an edge dual to $P$ with midpoint $x_0$. Let $\mathring{P}$ denote the interior of $P$ as a subspace of $W$. Then the following holds:
\begin{enumerate}
    \item The quotient map $q : \dot{X} \rightarrow \mcal{X}$ restricts to a universal covering map
$$q|_{\mathring{P}} : \mathring{P} \rightarrow (\ov{W}-c_e)(\ov{x}_0).$$
\item $(\ov{W}-c_e)(\ov{x}_0)=(\ov{W}-c)(\ov{x})$ for any other vertex $x \in P$ of $W$ and any $c \in [c_e]_W$.
\item The group of deck transformations of $q|_{\mathring{P}}$ 
is $K_P:= \corchete{g \in K \colon gx_0 \in P}=\mathrm{Stab}_{K}(P)$ (where $K$ is from Theorem \ref{finitewalls}), hence $K_P$ acts cocompactly on $P$.
\end{enumerate}
\end{lemma}
If $P$ is a portal leading to $(Z,H,(c_x))$, let $H_P$ denote its set-wise stabilizer in $H$. Suppose $P$ lies in the wall $W$ and let $x \in P$ and $h \in H$ be such that $hx \in P$. If $x'$ is the vertex closest to $x$ in $P$, then $x'$ is the midpoint of the edge $e$ dual to $P$ with $he$ also dual to $P\subset W$, implying $hx'\in P$,  $hW=W$, and $hP=h(Z \cap W)=Z \cap W=P$. We conclude that $h \in H_P$, so the map $P/{H_P} \rightarrow Z/H$ is an embedding, and also that $H_P$ acts properly and cocompactly on $P$ because $Z/H$ is compact.

This last observation and properness of the action of $G$ on $\dot{X}$, together with the previous two lemmas, imply that $H_P \cap K$ is finite index in $H_P$ for any portal $P$ leading to $(Z,H,(c_x))$. Thus we can use Lemma \ref{virtmodif} to modify our set $\mcal{V}_j$.
\begin{coro}\label{HsubsetK} We can virtually modify $\mcal{V}_j$ (in the sense of Lemma \ref{virtmodif}) so that $H_P<K$ for any portal $P$ leading to $Z$.
\end{coro}
\begin{proof}Let $P_1,\dots, P_k$ be a set of representatives of $H$-orbits for portals of $Z$, and note that by Lemma \ref{virtmodif} it is enough to replace $H$ by a finite index normal subgroup $H_0 \trianglelefteq H$ such that $H_0 \cap H_{P_i} \cap K=H_0 \cap H_{P_i}$ for each $i$. But each $H_{P_i} \cap K$ is finite index in $H_{P_i}$, so we just need each subgroup $H_{P_i} \cap K$ to be separable in $H$, which is true by Theorems \ref{haglundwiseseparability} and \ref{CMVHimpliesspecial} since $(H, \mcal{N}(Z)) \in \mcal{CMVH}$ and all the subgroups $H_{P_i} \cap K$ are convex in $(H, \mcal{N}(Z))$ (they preserve the convex subcomplexes $\mcal{N}(P_i) \cap \mcal{N}(Z)$ respectively).
\end{proof}
\begin{defi}We say that two portals $P$ and $P'$ leading to $(Z,H,(c_x)), (Z', H',(c'_x))\in \mcal{V}_j$ respectively are \emph{compatible} if there are edges $e$ and $e'$ dual to $P$ and $P'$ respectively such that $[e, c_e]\in G \cdot [e', c'_{e'}]$.
\end{defi}
Let $P$ and $P'$ be compatible portals as above, say lying in walls $W$ and $W'$. Take $g \in G$ and edges $e$ and $e'$ dual to $P$ and $P'$ such that $[e, c_e]= g[e', c'_{e'}]$, and let $x_0$ and $x_0'$ be the midpoints of $e$ and $e'$. So $e=ge'$, $W=gW'$ and $x_0=gx'_0$. At the level of $\mcal{X}$ the action of $g$ translates to $$g(\ov{W}'-c_{e'}')(\ov{x}_{0}')=(\ov{W}-g c_{e'}')(\ov{x}_{0})=(\ov{W}-c_{e})(\ov{x}_{0})$$
where we used Lemma \ref{covering} (2) and the fact that $[c_e]_W=[gc'_{e'}]_{W}$.

Since $q$ restricts to coverings for $\mathring{P}$ and $\mathring{P}'$, we deduce that $g$ restricts to a cube isomorphism $P' \rightarrow P$ which is equivariant with respect to the group isomorphism $K_{P'} \rightarrow K_P ; k \mapsto gkg^{-1}$. This induces an isomorphism $P'/K_{P'}\xrightarrow{\sim} P/K_P$.

\begin{lemma}[cf.~Lem.~9.2, \cite{Shepherd2019AgolsCubulations}]\label{compatibleportals} Portals $P$ and $P'$ leading to $(Z,H,(c_x)),(Z',H',(c'_x))\in \mcal{V}_j$ are compatible if and only if there exists $g \in G$ such that $$\corchete{[e, c_e]\colon e \text{ is dual to }P}= g\corchete{[e',c'_{e'}]\colon e' \text{ is dual to }P'}.$$
In particular, compatibility of portals is an equivalence relation.
\end{lemma}
Following notation of \cite{Shepherd2019AgolsCubulations}, in the case of $P, P'$ and $g$ as above we will say that $P$ is a \emph{$g$-teleport}
of $P'$.

\subsection{The graph of groups $(\mbb{A},\mcal{A})$} As we saw previously, compatible portals $P$ and $P'$ leading to $(Z,H,(c_x)),(Z',H',(c'_x))\in \mcal{V}_j$ are isomorphic and induce an isomorphism $P'/K_{P'}\rightarrow P/K_P$. However, we would like to glue $Z'/H'$ and $Z/H$ along $P'/H'_{P'}$ and $P/H_P$, which are only isomorphic up to a finite-sheeted cover by Corollary \ref{HsubsetK}. If we want isomorphisms $g : P'/H'_{P'} \rightarrow P/H_P$ we need to virtually modify $\mcal{V}_j$ again, and for that we will construct a graph of groups.

\begin{defi}Let $(\mbb{A},\mcal{A})$ be the finite bipartite (and possibly disconnected) graph of groups defined as follows.
\begin{itemize}
\item  \emph{Type I} vertices of $\mbb{A}$ are triplets $(Z,H,(c_x))\in \mcal{V}_j$ with corresponding vertex group $H$. Here repeated triplets are counted separately.

\item \emph{Type II} vertices of $\mbb{A}$ are portals $\corchete{P_i}$ forming a complete set of representatives for the compatibility classes of portals, with corresponding vertex groups $K_{P_i}$.

\item Edges attached to the Type I vertex $(Z,H,(c_x))$ will be portals $P$ leading to $(Z,H,(c_x)) \in \mcal{V}_j$, such that we choose just one $P$ from each $H$-orbit of portals (repeated triplets will have the same conjugacy representatives). The edge $P$ will be attached to the Type II vertex in its compatibility class of portals, and its edge group will be $H_P$.
\end{itemize}
\end{defi}
For a portal $P$ leading to $(Z,H,(c_x))$, the injection of the edge $H_P$ into its type I vertex group is just the inclusion $H_P \hookrightarrow H$, while the map into a type II vertex group is the composition $H_P \hookrightarrow K_P \xrightarrow{g(-)g^{-1}}K_{P_i}$, where $g \in G$ is so that $P_i$ is a $g$-teleport of $P$ (for the case $g : P_i \rightarrow P_i$ we take $g=1$, and same portals corresponding to repeated triplets will have the same $g$).

The next proposition may be thought of as a relative version of the acylindricity of the graph of groups $\mbb{A}$ proven in the absolute case (cf.~\cite[p1063]{Agol2012VirtualHaken} and \cite[Lem.~8.7]{Shepherd2019AgolsCubulations}). The proof is practically the same, the only difference is that we require Proposition \ref{loxodromics}.

\begin{prop}\label{acylindricity}If $(Z,H,(c_x))$ is a type I vertex group of $(\mbb{A},\mcal{A})$, then the collection
$$\corchete{H_P \colon P \text{ is an edge attached to }Z}$$
is relatively malnormal in $H$. That is, if $P_1$ and $P_2$ are edges attached to $Z$ and $h \in H$ is so that $H_{P_1} \cap H^h_{P_2}$ contains a loxodromic, then $P_1=P_2$ and $h\in H_{P_1}$.
\end{prop}
\begin{proof} Assume $\lambda \in H_{P_1} \cap H_{P_2}$
is a loxodromic element, in which case we claim that $P_1=P_2$.
If $W_1$ and $W_2$ are the walls containing $P_1$ and $P_2$ respectively, then it is enough to prove that $W_1=W_2$, since that implies $P_1=Z \cap W_1=Z \cap W_2=P_2$.

The element $\lambda$ acts freely on $P_1$ and $P_2$, so it acts hyperbolically on them, and there exist axes $\gamma_i\subset P_i\subset W_i$ in which $\lambda$ acts by non-trivial translation. These axes are asymptotic in $Z$, thus $\gamma_1$ and $\gamma_2$ bound a flat strip in $Z$ of width $r \geq 0$. Since $\lambda$ is loxodromic, by Proposition \ref{loxodromics} we have $r \leq \delta$. Let us assume $W_1 \neq W_2$, so that $\ov{W}_1 \neq \ov{W}_2$ and $d(\ov{W}_1, \ov{W}_2) \leq \delta \leq R$, and get a contradiction by showing that $\ov{W}_1$ and $\ov{W}_2$ are colored equal by some coloring.

If $p$ is any point in $P_1$, then it is contained in a cube $C$ of $X$ and we can find a vertex $x \in Z$ incident to an edge dual to $P_1$, with $d(p, x) \leq \frac{1}{2}\sqrt{\dim C} \leq \frac{1}{2}\sqrt{\dim X}$. The same is true for $P_2$, so there are vertices $x_1, x_2 \in Z$ with each $x_i$ being incident to an edge dual to $P_i$ so that the geodesic segment $\alpha$ joining $x_1$ and $x_2$ has length at most $\delta+\sqrt{\dim X}$. By considering the sequence of cubes that $\alpha$ travels through, we can find an edge path $\beta$ in $Z$ from $x_1$ to $x_2$ with $\beta \subset N_{\sqrt{\dim X}}(\alpha)$. Let $e_1,\dots, e_s$ be the edges of $\beta$, and $x_1=y_1, y_2,\dots, y_{s+1}=x_2$ be their vertices, with $e_i$ joining $y_i$ and $y_{i+1}$. Since $R \geq \delta + 2\sqrt{\dim X}$ (see the beginning of Section \ref{constructionofthecomplex}) we have $d(\ov{W(e_i)}, \ov{W}_1) \leq d(W(e_i), W_1) \leq R$ for all $1 \leq i \leq s$, and so $\ov{W(e_i)}$ and $\ov{W}_1$ are adjacent vertices in $\Gamma(\mcal{X})$. For each $1 \leq i \leq s$, from property (1) of $\mcal{V}_j$ we deduce $[c_{y_i}]_{e_i}=[c_{y_{i+1}}]_{e_i}$, so $c_{y_i}(\ov{W}_1)=c_{y_{i+1}}(\ov{W}_1)$ and hence $c_{x_1}(\ov{W}_1)=c_{x_{2}}(\ov{W}_1)=j$, because $W_1$ and $W_2$ are $j$-boundary walls. Thus $c_{x_2}$ is a coloring with $c_{x_2}(\ov{W}_1)=c_{x_2}(\ov{W}_2)$, contradicting $W_1\neq W_2$.

To finish the proposition, let $P_1$ and $P_2$ be edges attached to $(Z,H,(c_x))$ and $h \in H$ be such that $H_{P_2} \cap H^h_{P_2}=H_{P_1} \cap H^h_{P_2}$ contains a loxodromic. By our previous claim we obtain $P_1=hP_2$, and since different edges correspond to distinct representatives of $H$-orbits of portals we must have $P_1=P_2$, implying $h \in H_{P_1}$. \end{proof}

\begin{defi} Let $\mbb{G}=\mbb{G}_c$ be the fundamental group of a connected component $(\mbb{A}_c, \mcal{A}_c)$ of $(\mbb{A},\mcal{A})$ with respect to some vertex $w_0$ of $\mbb{A}_c$, and fix a maximal subtree $T$ of $\mbb{A}_c$ containing $w_0$ that fixes inclusions of the edge/vertex groups of $(\mbb{A}_c, \mcal{A}_c)$ into $\mbb{G}$.
\end{defi}
\begin{prop}\label{weakseparability} Let $G_{e_0}$ be an edge group of $(\mbb{A}_c, \mcal{A}_c)$ attached to the type II vertex group $G_{v_0}$, and let $a \in G_{v_0}\backslash G_{e_0}\subset \mbb{G}$. Then there exists a finite index subgroup $\dot{\mbb{G}}_{e_0,a}<\mbb{G}$ containing $G_{e_0}$ with $a \notin \dot{\mbb{G}}_{e_0,a}$.
\end{prop}
\begin{proof}The proof is almost the same as the one given for Proposition \ref{factorvspecial}, so we just give a sketch of it.

Recall that by Lemma \ref{compimpliesstrongly} each vertex/edge group is convex and strongly peripherally separable in $G$, so by several applications of Proposition \ref{superfilling} we can find finite index subgroups $\dot{P}_j <P_j$ such that the filling $\phi : G \rightarrow \ov{G}=G(\mcal{N}=\corchete{\dot{P}_1,\dots, \dot{P}_n})$ satisfies:
\begin{itemize}
 \item $\ov{G}_v := \phi(G_v)$ is hyperbolic and virtually special, and isomorphic to image of the induced filling $\phi_v : G_v \rightarrow G_v(\mcal{N}_v)$ for any vertex $v$ of $\mbb{A}_c$ (images of type II vertex groups will be virtually special because every type II vertex is a finite index extension of an edge group, which is virtually special by Theorem \ref{CMVHimpliesspecial}, Corollary \ref{vspecialsubdivision} and our inductive assumption, see the beginning of Section \ref{constructionofthecomplex}).

 \item The image $\ov{G}_e := \phi(G_e)$ of the edge group $G_e$ of $(\mbb{A}_c, \mcal{A}_c)$ with terminal vertex of type I is naturally isomorphic to the image of the filling $\phi_e : G_e \rightarrow G_e(\mcal{N}_e)$ induced by both $\phi$ and $\phi_{t(e)}$ (that is, $\ker\phi_e=\ker\phi_{t(e)} \cap G_e=\ker\phi \cap G_e$).

\item  The collection of images under $\phi$ of groups in $\corchete{\ov{G}_e : e \text{ attached to }v}$ is almost malnormal
in $\ov{G}_v$ for any type I vertex $v$ of $\mbb{A}_c$.

\item $\phi(a) \notin \ov{G}_{e_0}$.
\end{itemize}
We then consider the graph of groups $(\mbb{A}_c,\ov{\mcal{A}}_c)$ with the same underlying graph $\mbb{A}_c$, and $\ov{\mcal{A}}_c$ assigning the group $\ov{G}_x$ to each vertex/edge $x$ of $\mbb{A}_c$, and with attaching maps induced by $\phi$ and the attaching maps of $(\mbb{A}_c,\mcal{A}_c)$ (every attaching map of $(\mbb{A}_c,\mcal{A}_c)$ is composition of inclusions and conjugations in $G$). Let $\ov{\mbb{G}}= \pi_1(\mbb{A}_c, \ov{\mcal{A}}_c, w_0)$ and choose embeddings of edge/vertex groups according to the same maximal subtree $T$ of $\mbb{A}_c$.

The homomorphism $\phi$ restricted to each edge/vertex group induces a homomorphism $\Phi : \mbb{G} \rightarrow \ov{\mbb{G}}$ such that $\Phi(x)=\phi_v(x)$ for any vertex $v$ of $\mbb{A}_c$ and for any $x \in G_v$, and since the splitting $(\mbb{A}_c,\ov{\mcal{A}}_c)$ of $\ov{\mbb{G}}$ satisfies the assumptions of Lemma \ref{hyperbolicsplitting} (2), $\ov{\mbb{G}}$ is hyperbolic and $\ov{G}_{e_0}$ is quasiconvex in $\ov{\mbb{G}}$. Then Theorem \ref{wiseQVH} implies that $\ov{\mbb{G}}$ is virtually special, and since $\Phi(a)=\phi(a) \notin \ov{G}_{e_0}=\Phi(G_{e_0})$, Theorem \ref{haglundwiseseparability} gives us the separability of $\Phi(G_{e_0})$  in $\ov{\mbb{G}}$, and hence the existence of $\dot{\mbb{G}}_{a,e_0}$.
\end{proof}
\begin{coro}\label{uniformN} There is a finite index subgroup $\mbb{N}_c \trianglelefteq \mbb{G}_c$ such that if $G_e<\mbb{G}_c$ is an edge group attached to (hence contained in) the type II vertex group $G_v<\mbb{G}_c$, then \begin{equation}\label{eq7}G_e \cap \mbb{N}_c=G_v \cap \mbb{N}_c.
\end{equation}
\end{coro}
\begin{proof}Recall that an edge group $G_e$ is finite index in the type II vertex group $G_v$ it is attached. So, for each edge $e$ let $S_e\subset G_v\backslash G_e$ be any finite set of representatives of non-trivial left cosets of $G_e$ in $G_v$. By Proposition \ref{weakseparability}, for each $a \in S_e$ there is a finite index subgroup $\dot{\mbb{G}}_{a,e_0}<\mbb{G}_c$ separating $G_e$ from $a$, and so the intersection of the finitely many conjugates of $\dot{\mbb{G}}_{a,e_0}$ in $\mbb{G}_c$ is a finite index normal subgroup, that we denote $\mbb{N}_{e,a}$. The group $\mbb{N}:=\bigcap_{e \in E(\mbb{A}_{c})}{\bigcap_{a \in S_{v}}{\mbb{N}_{e, a}}}$ satisfies the required equalities.
\end{proof}

\section{Constructing $\mcal{V}_{j-1}$ from $\mcal{V}_j$}\label{constructingVj-1}
Finally we can start the construction of $\mcal{V}_{j-1}$, which implies Theorem \ref{relhypimpliesCMVH}.

Let $(\mbb{A}_c,\mcal{A}_c)$ be a component of $(\mbb{A},\mcal{A})$ with fundamental group $\mbb{G}_c$, and let $\mbb{N}_c \trianglelefteq \mbb{G}_c$ be given by Corollary \ref{uniformN}. For each type I vertex group $H$ of $(\mbb{A}_c,\mcal{A}_c)$, define $\hat{H} := \mbb{N}_c \cap H \trianglelefteq H$ and see these groups as subgroups of $G$. After doing this for each connected component $(\mbb{A}_c,\mcal{A}_c)$ of $(\mbb{A},\mcal{A})$ we obtain a finite index subgroup of $H$ for each $(Z,H,(c_x))\in \mcal{V}_j$ and so Lemma \ref{virtmodif} gives us a virtual modification $\hat{\mcal{V}}_j$ of $\mcal{V}_j$ with triplets $(Z, \hat{H}, (c_x))$.

We will use $\hat{\mcal{V}}_j$ to construct a (possibly disconnected) graph of spaces $(\mbb{S},\mcal{S})$. As before, let $\corchete{P_i}$ be the set of type II vertices of $\mbb{A}$, and for each $(Z,\hat{H},(c_x)) \in \hat{\mcal{V}}_j$ choose a set $\mbb{B}_Z$ of representatives of portals leading to $Z$, with exactly one portal $P$ for each $\hat{H}$-orbit of translates of portals (with same representatives for repeated triplets). Set $\mbb{B} :=
\bigsqcup_{Z\in\hat{\mcal{V}}_j}{\mbb{B}_Z}$, and choose the representatives of portals in such a way that each edge of $\mbb{A}$ lies in $\mbb{B}$.

As in \cite[Def.~9.6]{Shepherd2019AgolsCubulations}, the \emph{size} of an edge $P$ of $\mbb{A}$ is defined by
$$\mathrm{sz}(P) := | \corchete{H_P\cdot e\colon e \text{ is an edge dual to }P}|,$$
where $P$ leads to $(Z,H,(c_x))\in \mcal{V}_j$. Similarly, for $P \in \mbb{B}$ leading to $(Z,\hat{H},(c_x)) \in \hat{\mcal{V}}_j$, define the
\emph{size} of $P$ as
$$\hat{\mathrm{sz}}(P) := | \corchete{\hat{H}_P\cdot e\colon e \text{ is an edge dual to }P}|,$$
where $\hat{H}_P := \hat{H} \cap H_P$ is the stabilizer of $P$ in $\hat{H}$ . Note that $\hat{\mathrm{sz}}(P)=|H_P : \hat{H}_P|\mathrm{sz}(P)$ for any edge $P$ of $\mbb{A}$ attached to $(Z,H,(c_x))$. Also, by equation \eqref{eq7}, $\hat{\mathrm{sz}}(P)=\hat{\mathrm{sz}}(P')$ whenever $P, P' \in \mbb{B}$ are compatible portals.

If $P$ is a portal leading to $Z$ and contained in the wall $W$, with $Z$ in either $\mcal{V}_j$ or $\hat{\mcal{V}}_j$, and $P_i$ is a $g$-teleport of $P$ contained in the wall $W_i$ that is a type II vertex group of $\mbb{A}$, we say that $P$ is a $P^+_i$-\emph{portal} if $gZ \cap W_i^+ \neq \emptyset$, and a $P^-_i$-portal if $gZ \cap W_i^-\neq \emptyset$. For $Z\in \mcal{V}_j$ define $\mbb{A}_Z(P_i,\pm)$ as the set of $P^{\pm}_i$-portals $P$ leading to $Z$ which are an edge of $A$, and let $\mbb{A}(P_i,\pm):=\bigsqcup_{Z\in \mcal{V}_j}{\mbb{A}_Z(P_i,\pm)}$. The sets $\mbb{B}_Z(P_i,\pm)$ for $Z \in \hat{\mcal{V}}_j$ and $\mbb{B}(P_i,\pm)$ are defined similarly.

Before defining $(\mbb{S},\mcal{S})$ we need some combinatorial results, the first one contained in \cite[Lem.~9.7]{Shepherd2019AgolsCubulations}.

\begin{lemma}\label{comblemma} If $P_i$ is a type II vertex of $\mbb{A}$ then $$\sum_{P \in \mbb{A}\left(P_{i},+\right)} \mathrm{sz}(P)=\sum_{P \in \mbb{A}\left(P_{i},-\right)} \mathrm{sz}(P).$$
\end{lemma}
\begin{coro}\label{combcoro} For each type II vertex $P_i$ of $\mbb{A}$, the number of $P^+_i$-portals $P$ in $\mbb{B}$ equals the number of $P^-_i$-portals $P$ in $\mbb{B}$.
\end{coro}
\begin{proof}Let $s := \hat{\mathrm{sz}}(P)$ be the size of any portal $P$ in $\mbb{B}$ compatible with $P_i$. From the proof of Lemma \ref{virtmodif} it follows that there exists some positive integer $d$ such that for each triplet $(Z,H,(c_x))\in \mcal{V}_j$ there are $d/|H : \hat{H} |$ triplets of $(Z, \hat{H}, (c_x))$ in $\hat{\mcal{V}}_j$. Also, if $P$ is any portal leading to $(Z,H,(c_x))$, then the set of $H$-orbits of $P$ is the disjoint union of $|H : \hat{H} |/|H_P : \hat{H}_P|$ $\hat{H}$-orbits of translates of portals. With this in mind we have
\begin{align*} \sum_{P \in \mbb{B}(P_{i},+)}{1}=\sum_{Z \in \hat{\mcal{V}}_{j}}{\sum_{P \in \mbb{B}_{Z}(P_{i},+)}{1}} &=\sum_{(Z, H,(c_{x})) \in \mcal{V}_{j}}{\frac{d}{|H: \hat{H}|}\left(\sum_{P \in \mbb{B}_{Z}(P_{i},+)}{1}\right)}\\ &=\sum_{(Z, H,(c_{x})) \in \mcal{V}_{j}} {\frac{d}{|H: \hat{H}|}\left(\sum_{P \in \mbb{A}_{Z}(P_{i},+)}{\frac{|H: \hat{H}|}{|H_{P}: \hat{H}_{P}|}}\right)}\\ &=\frac{d}{s} \sum_{(Z, H,(c_{x})) \in \mcal{V}_{j}}{\left(\sum_{P \in \mbb{A}_{Z}(P_{i},+)}{\frac{\hat{\mathrm{sz}}(P)}{|H_{P}: \hat{H}_{P}|}}\right)}\\ &=\frac{d}{s}\sum_{(Z,H,(c_{x}))\in \mcal{V}_{j}}{\left(\sum_{P \in \mbb{A}_{Z}(P_{i},+)}{\mathrm{sz}(P)}\right)} \\ &=\frac{d}{s} \sum_{P \in \mbb{A}(P_{i},+)}{\mathrm{sz}(P)}. \end{align*}
The same is true for the $P^-_i$-portals, and so the conclusion follows by Lemma \ref{comblemma}.
\end{proof}
\begin{defi}Let $(\mbb{S},\mcal{S})$ be the graph of spaces defined as follows.
\begin{itemize} 
\item The vertices of $\mbb{S}$ will be the triplets $(Z,\hat{H},(c_x)) \in \hat{\mcal{V}}_j$, with corresponding vertex spaces $Z/\hat{H}$.
\item If $P_i$ is a type II vertex of $\mbb{A}$, by Corollary \ref{combcoro} there exists a perfect matching between $\mbb{B}(P_i,+)$ and $\mbb{B}(P_i,-)$. If $p := (P, P')$
is an oriented pair given by this matching with $P$ and $P'$
leading to $(Z, \hat{H},(c_x))$ and $(Z', \hat{H}',(c'_x))$ respectively, then $P$ is a $g_p$-teleport of $P'$ for some fixed $g_p \in G$, and there are embeddings
\begin{equation}\label{eq8}
P'/{\hat{H}'_{P}} \hookrightarrow Z'/ {\hat{H}'}, \quad \text{ and } \quad P' / {\hat{H}'_{P'}} \xrightarrow{g_{p}} P/ {\hat{H}_{P}} \hookrightarrow Z /\hat{H},
\end{equation}
where $P' / {\hat{H}'_{P'}} \xrightarrow{g_{p}} P/ {\hat{H}_{P}}$ is an isomorphism for $g_P$ chosen appropriately, due to Corollary \ref{uniformN}.
\item  The edges of $\mbb{S}$ are oriented pairings $p := (\hat{H}_P , \hat{H}'_{P'})$ attached to $Z$ and $Z'$ as above, with edge spaces $P'/\hat{H}'_{P'}$ and attaching maps given by \eqref{eq8}. For the reverse pairing $\ov{p}=(P',P)$, the attaching maps are constructed in the same way with $g_{\ov{p}}=g_p^{-1}$.
\end{itemize}
\end{defi}

Consider a component $(\mbb{S}_c, \mcal{S}_c)$ of $(\mbb{S},\mcal{S})$, with underlying space $\mcal{T}_c$ obtained by gluing vertex spaces along images of attaching maps. We want to construct an embedding $\wtilde{T}_c \hookrightarrow \dot{X}$ of the universal cover of $\mcal{T}_c$ into $\dot{X}$.

First of all, fix a base-point $x \in Z$ for each $(Z,\hat{H},(c_x)) \in \hat{\mcal{V}}_j$. If $Z/\hat{H}$ and $Z'/\hat{H}'$ are vertex spaces joined by the oriented edge $p=(P, P')$ then, up to homotopy, there exists a unique path $\alpha_p$ in $Z \cup g_pZ'$ from $x$ to $g_px$ (this is because $Z \cup g_pZ'$ is simply-connected). Let $\beta_p$ be the projection of $\alpha_p$ into $\mcal{T}_c$ via $Z \rightarrow Z/\hat{H}$ and $Z' \rightarrow Z'/\hat{H}'$. We can choose our paths so that $\beta_{\ov{p}}$ is the reverse path of $\beta_p$. Also, for each $(Z, \hat{H},(c_x))$ and $h \in \hat{H}$ let $\gamma_h$ be a loop in $Z/\hat{H}$ lifting to a path from $x$ to $hx$.

Fix a base vertex $(Z_0, \hat{H}_0,(c_x)_0)$ of $\mbb{S}_c$, and for $p_1,\dots, p_n$ edges in $\mbb{S}_c$ forming a path through vertex spaces
$$Z_0/\hat{H}_0
\xrightarrow{p_1} Z_1/\hat{H}_1
\xrightarrow{p_2} \dots
\xrightarrow{p_n} Z_n/\hat{H}_n,$$
and $h_i \in \hat{H}_i$ for $i=0, 1,\dots, n$, consider the concatenation
\begin{equation}\label{eq9}\gamma=\gamma_{h_0}\cdot \beta_{p_1}\cdot \gamma_{h_{n-1}}\cdots\beta_{p_n}\cdot \gamma_{h_n},
\end{equation}
which gives a path in $\mcal{T}_c$. Note that every path in $\mcal{T}$ between (projections of) base-points of vertex spaces and starting in $Z_0/\hat{H}_0$ is homotopic to a path in this form. For each such $\gamma$, we define
$$g(\gamma)=h_0g_{p_1}h_1 \cdots g_{p_n}h_n,$$
and we take $\wtilde{\mcal{T}}_c \subset \dot{X}$ to be the union of all the possible $G$-translates $g(\gamma)Z_n$. The covering map
$\mu : \wtilde{\mcal{T}}_c \rightarrow \mcal{T}_c$ restricts to $g(\gamma)Z_n$ by
$$\mu : g(\gamma)Z_n
\xrightarrow{g(\gamma)^{-1}} Z_n \rightarrow Z_n/\hat{H}_n \rightarrow \mcal{T}_c.$$
The next lemma is \cite[Lem.~9.13 \& Lem.~9.15]{Shepherd2019AgolsCubulations}.
\begin{lemma}\label{universalcovering} $\mu : \wtilde{\mcal{T}}_c \rightarrow \mcal{T}_c$ is a universal covering map and $\wtilde{\mcal{T}}_c \subset \dot{X}$ is a non-empty intersection of half-spaces.
\end{lemma}
For a loop $\gamma$ as in \eqref{eq9} and $g(\beta)Z$ in $\wtilde{\mcal{T}}_c$, we have that $g(\gamma)g(\beta)=g(\gamma \cdot \beta)$, hence $g(\gamma)g(\beta)Z=g(\gamma \cdot \beta)Z \subset \wtilde{\mcal{T}}_c$. This holds for all translates $g(\beta)Z$ in $\wtilde{\mcal{T}}_c$, thus $g(\gamma)\wtilde{\mcal{T}}_c=\wtilde{\mcal{T}}_c$ and $\mu\circ g(\gamma)= \mu$.

Define
$$H(\mcal{T}_c) := \corchete{g(\gamma)\colon \gamma \text{ is a loop of form }\eqref{eq9}}<G.$$
Since $g(\gamma)g(\beta)=g(\gamma \cdot \beta)$ for any loop $\gamma$ and path $\beta$, $H(\mcal{T}_c)$ is a subgroup of $G$ preserving $\wtilde{\mcal{T}}_c$. Moreover, $\mu \circ g(\gamma) = \mu$ for every $g(\gamma) \in H(\mcal{T}_c)$, so $H(\mcal{T}_c)$ is a subgroup of the group of Deck transformations of $\mu$. In addition, by construction the orbit of the base-point $x_0$ of $Z_0$ is $H(\mcal{T}_c)\mu^{-1}1(\mu(x_0))$, implying that $H(\mcal{T}_c)\cong \pi_1(\mcal{T}_c)$ is the full group of Deck transformations of $\mu$, and hence it acts freely and cocompactly on $\wtilde{\mcal{T}}_c$.

Finally, if $x \in Z$ is any vertex with $(Z,\hat{H},(c_x)) \in \hat{\mcal{V}}_j$ and $\gamma$ is as in \eqref{eq9}, then we endow $g(\gamma)x$ with the coloring $c^{\mcal{T}_c}_{g(\gamma)x}:= g(\gamma)c_x$. It is evident that $H(\mcal{T}_c)$ preserves these colorings.
\begin{defi}Let $\mcal{V}_{j-1}$ consist of the set of triplets $(\wtilde{\mcal{T}}_c, H(\mcal{T}_c),(c^{\mcal{T}_c}_x))$, with one triplet for each underlying space $\mcal{T}_c$ for a component $(\mbb{S}_c, \mcal{S}_c)$ of $(\mbb{S},\mcal{S})$.
\end{defi}
\begin{prop}\label{vj-1} $\mcal{V}_{j-1}$ satisfies all the desires properties (1)-(4) of Definition \ref{defVj}.
\end{prop}
\begin{proof}For properties (1) and (2), the proof is the same as in \cite[p30-31]{Shepherd2019AgolsCubulations} so it will be omitted.

To show property (3), first note that by Lemma \ref{consistentcompatible} each subgroup $H(\mcal{T}_c)$ is convex in $G$ with convex core $\mcal{N}(\wtilde{\mcal{T}}_c)$ so it is also hyperbolic relative to compatible virtually special subgroups by Lemma \ref{nbhdconvex}. In addition, since any edge space embedding into a vertex space of $(\mbb{S}_c, \mcal{S}_c)$ is $\pi_1$-injective, by Van Kampen's theorem there is an induced splitting $(\mbb{S}_c, \mcal{U}_c)$ of $H(\mcal{T}_c) \cong \pi_1(\mcal{T}_c)$ with vertex groups $\hat{H}$ for $(Z,\hat{H},(c_x)) \in \hat{\mcal{V}}_j$, and with edge groups of the form $\hat{H}_P\cong \hat{H}'_{P'}$ for each edge $p=(P, P')$ with $P, P'$ leading to $(Z, \hat{H},(c_x))$ and $(Z', \hat{H}',(c'_x))$ respectively. By assumption and Lemma \ref{finiteindexCMVH} each cubulated vertex group $(\hat{H}, \mcal{N}(Z))$ of $(\mbb{S}_c, \mcal{U}_c)$ is in $\mcal{CMVH}$, and by construction $\mcal{N}(Z)$ is the $G$-translate of a convex subcomplex of $\mcal{N}(\wtilde{\mcal{T}}_c)$. The same is true for the embedding of an edge group into $H(\mcal{T}_c)$ since it acts cocompactly on a translate of the cubical neighborhood $\mcal{N}(P)$ of a portal $P$, that is a convex subcomplex of $\mcal{N}(\wtilde{\mcal{T}}_c)$. This implies that each vertex/edge group is convex in $(H(\mcal{T}_c), \mcal{N}(\wtilde{\mcal{T}}_c))$. Finally note that the conclusion of Proposition \ref{acylindricity} holds also for vertex groups in $(\mbb{S}_c,\mcal{U}_c)$, and hence the collection of embeddings of edge groups into a vertex group of $(\mbb{S}_c, \mcal{U}_c)$ is relatively malnormal. Therefore $(H(\mcal{T}_c),\mcal{N}(\wtilde{T}_c)) \in \mcal{CMVH}$.

For property (4), note that since each $\mcal{T}_c$ is obtained by gluing quotients $Z/\hat{H}$ for $(Z,\hat{H},(c_x)) \in \hat{\mcal{V}}_j$, with each triplet being used in exactly one component $\mcal{T}_c$ (repeated triplets are counted separately), there is a canonical bijection $\Lambda$ from the set of vertices of $\bigsqcup_{(Z,\hat{H},(c_x))\in\hat{\mcal{V}}_j}{Z/\hat{H}}$ onto the set of vertices of $\bigcup_{c}{\mcal{T}_c}$, where $c$ runs among that components of $(\mbb{S},\mcal{S})$ (here by vertex we mean an image of a vertex of $\dot{X}$ contained in some $Z$). Also, any vertex of some $\wtilde{\mcal{T}}_c$ takes the form $\wtilde{x}= g(\gamma)x$ for some path $\gamma$ as in \eqref{eq9} and some vertex $x \in (Z,\hat{H},(c_x)) \in \hat{\mcal{V}}_j$. Since by definition $c^{\mcal{T}_c}_{\wtilde{x}}= g(\gamma)c_x$, $\Lambda$ restricts to a bijection from $\mcal{V}^{\pm}_j(e,c)$ onto $\mcal{V}^{\pm}_{j-1}(e,c)$ for any equivalence class $[e, c]$ with $e$ an edge and $c$ a coloring. Since $\mcal{V}_j$ satisfies the gluing equations, then $\mcal{V}_{j-1}$ also does.
\end{proof}

\appendix
 
\section{Functoriality of the canonical completion}\label{appendix}
This appendix is devoted to proving the following theorem, which generalizes one of the implications of Theorem \ref{haglundwisedoublecoset}.
\begin{thm}\label{doublecosetconvexsep} Let $(G,X)$ be a cubulated virtually special group. Then for any pair $H, K<G$ of convex subgroups, the double coset $HK$ is separable in $G$.
\end{thm}
By using the fact that $gKg^{-1}<G$ is convex whenever $K<G$ is convex and $g\in G$, we deduce: 
\begin{coro}\label{doublecosetconvexsepg}
Let $(G,X)$, $H$ and $K$ be as in the previous Theorem. Then for any $g\in G$ the double coset $HgK$ is separable in $G$.
\end{coro}
Recall that $\dot{X}$ denotes the cubical barycentric subdivision of the $\CAT{0}$ cube complex $X$.
\begin{coro}\label{vspecialsubdivision}
Let $(G,X)$ be a cubulated group. Then $(G,X)$ is virtually special if and only if $(G,\dot{X})$ is virtually special.
\end{coro}
\begin{proof}
If $W'$ is a wall of $\dot{X}$, then $G_{W'}$ is a finite index subgroup of $G_W$ for some wall $W$ of $X$, hence a convex subgroup of $(G,X)$. Therefore, by Theorems \ref{haglundwiseseparability}, \ref{doublecosetconvexsep}, and \ref{haglundwisedoublecoset}, if $(G,X)$ is virtually special then $(G,\dot{X})$ is virtually special. The converse follows by a similar argument. 
\end{proof}

To prove Theorem \ref{doublecosetconvexsep}, we will make use of the properties of the \emph{canonical completion and retraction} introduced by Haglund and Wise \cite[Sec.~6]{HaglundWise2008Special}, so we first proceed recalling this construction.
\begin{defi} Let $f : A \rightarrow B$ be a local isometry of NPC special cube complexes, and assume that $B$ is \emph{fully clean} in the sense of \cite[Def.~6.1]{HaglundWise2008Special} and has simplicial 1-skeleton. Then there exists a covering map $p : \msf{C}(A,B) \rightarrow B$ (of finite degree if $A$ is finite), an injection $j : A \rightarrow \msf{C}(A,B)$ and a cellular map $r : \msf{C}(A,B) \rightarrow A$ such that $rj=1_A$ and $pj=f$. The covering is defined as follows:
\end{defi}
The 0-skeleton of $\msf{C}(A,B)$ is $A^{(0)}\times  B^{(0)}$ with $j(a)=(a, f(a))$ for $a \in A^{(0)}$, and $r : A^{(0)}\times B^{(0)} \rightarrow A^{(0)}$ and $p : A^{(0)} \times B^{(0)} \rightarrow B^{(0)}$ being the projections to the first and second coordinate, respectively.

Since the 1-skeletons of $A$ and $B$ are simplicial, edges in $A$ and $B$ are determined by their endpoints. The edges of $\msf{C}(A,B)$ are of two types:
\begin{itemize}
 \item \emph{Horizontal}: pairs of the form $\corchete{(a, b),(a, b')}$ with $\corchete{b, b'}$ an edge of $B$ and such that there is no edge $e$ of $A$ incident to $a$ with $f(e)$ and $\corchete{b, b'}$ dual to the same wall.
\item \emph{Diagonal}: pairs of the form $\corchete{(a, b),(a', b')}$ with $\corchete{b, b'}$ an edge of $B$ and $e=\corchete{a, a'}$ an edge of $A$ with $f(e)$ and $\corchete{b, b'}$ dual to the same wall (note that $\corchete{(a',b),(a,b')}$ is also a diagonal edge).
\end{itemize}
It follows from this definition that if $\corchete{a, a'}$ is an edge of $A$, then $j(\corchete{a, a'})=\corchete{(a, f(a)),(a', f(a'))}$ is a diagonal edge of $\msf{C}(A,B)$. Also, for an edge $e=\corchete{(a, b),(a, b')}$ as above, we define $p(e)=\corchete{b, b'}$ and $r(e)=a$ if $a=a'$ and $e$ is horizontal, and $r(e)=\corchete{a, a'}$ if $e$ is diagonal.\\
Clearly we have $rj=1_A$ and $pj=f$ at the level of 1-skeletons, and full cleanliness implies that $p$ is a covering map from $\msf{C}(A,B)^{(1)}$ into $B^{(1)}$. It can be proven the any lifting of the 1-skeleton of a square of $B$ is a closed 4-circuit, so the 2-skeleton of $\msf{C}(A,B)$ is constructed in such a way that the boundaries of squares coincide with liftings of boundaries of squares of $B$.\\
The maps $p$ and $j$ then naturally extend to the 2-skeleton, and since any wall of the 2-skeleton of $\msf{C}(A,B)$ only consists of either horizontal or diagonal edges, we can extend $r$ to this 2-skeleton by mapping a square $Q \subset \msf{C}(A,B)$ either to a square (if the walls dual to $Q$ are both diagonal), to an edge (if one wall dual to $Q$ is diagonal and the other one is horizontal, then collapse the horizontal edge to a point), or to a point (if both walls dual to $Q$ are horizontal).

There is a unique way to complete this 2-skeleton to produce a non-positively curved cube complex $\msf{C}(A,B)$ \cite[Lem.~3.13]{HaglundWise2008Special}, for which we can naturally extend $j, r$ and $p$ to maps satisfying the desired commuting properties (see \cite[Cor.~6.7]{HaglundWise2008Special}). We call $\msf{C}(A,B)$ the \emph{canonical completion} of $f : A \rightarrow B$ with (\emph{canonical}) \emph{inclusion} $j$ and (\emph{canonical}) \emph{retraction} $r$.

\begin{rmk}\label{rmkcompletion}In general, $\msf{C}(A,B)$ may be a disconnected covering of $B$. Also, from the construction above, it follows that if $e$ is an edge of $A$ and $e'$ is an edge of $\msf{C}(A,B)$ dual to the wall $W(j(e))$, then $e'$ is diagonal and $r(e')$ is dual to the wall $W(e) \subset A$. It is not hard to see that $j$ maps distinct walls of $A$ to distinct walls of $\msf{C}(A,B)$.
\end{rmk}
To produce separable double cosets from the canonical completion, we will use the following separability criterion \cite[Lem.~9.3]{HaglundWise2008Special}:
\begin{prop}\label{criteriondouble} Let $G$ be a residually finite group, and let $\rho : G \rightarrow G$ be a retraction homomorphism with $H=\rho(G)$. Then $H<G$ is a separable subgroup, and if $K<G$ is separable with $\rho(K)\subset K$, then the double coset $HK$ is separable in $G$.
\end{prop}Our first lemma is just a generalization of \cite[Prop.~9.7]{HaglundWise2008Special}, following exactly the same proof.
\begin{lemma}\label{sepconvex-wall}If $(G,X)$ is a virtually special group, $H< (G,X)$ is a convex subgroup and $W \subset X$ is a wall, then $HG_W \subset G$ is separable.
\end{lemma}
\begin{proof}By using Lemma \ref{enlarging} we can find a convex core $Y \subset X$ for $H$ such that $Y \cap W$ is non-empty, and let $a$ be an edge of $Y$ dual to $W$ and incident to the vertex $y \in Y$ . Since $G$ is residually finite, by Lemma \ref{finiteclasses} there is a finite index subgroup $\hat{G}< G$ acting freely on $X$ such that $\ov{X} := X/\hat{G}$ is fully clean and special with simplicial 1-skeleton, and such that the projection $X \rightarrow \ov{X}$ maps squares to squares \cite[Rmk.~6.8]{HaglundWise2008Special}. In that case, if $\hat{H}= H \cap \hat{G}$, then $\ov{Y}=Y/\hat{H}$ is compact and the composition $f : \ov{Y} \rightarrow X/\hat{H} \rightarrow \ov{X}$ is a local isometry. We can also assume that $(W \cap Y)/\hat{H}_W \rightarrow \ov{Y}$ and $W/\hat{G}_W \rightarrow X/\hat{G}$ embed as walls, that we denote respectively by $W_{\ov{a}}$ and $W_{f(\ov{a})}$, with $\ov{a}$ being the image of $a$ under the projection $Y \rightarrow \ov{Y}$. Let $p : \msf{C}(\ov{Y},\ov{X}) \rightarrow \ov{X}$ be the canonical completion induced by $f$, with retraction $r$ and inclusion map $j$.

Let $\mcal{N}_{f(\ov{a})} \subset \ov{X}$ denote the cubical neighborhood of $W_{f(\ov{a})}$, which is lifted by $p$ to the cubical neighborhood $\mcal{N}_{j(\ov{a})}$ of the wall $W_{j(\ov{a})}\subset \msf{C}(\ov{Y},\ov{X})$ dual to $j(\ov{a})$. By Remark \ref{rmkcompletion}, $r$ maps any edge of $\msf{C}(\ov{Y},\ov{X})$ dual to $W_{j(\ov{a})}$ to an edge dual to $W_{\ov{a}} \subset \ov{Y}$. In particular, since $r$ maps cubes to cubes (possibly of lower dimension), we have $r(\mcal{N}_{j(\ov{a})}) \subset \mcal{N}_{\ov{a}}$, where $\mcal{N}_{\ov{a}}$ is the cubical neighborhood of $W_{\ov{a}}$. On the other hand, since $X \rightarrow \ov{X}$ maps squares to squares, $j$ and $f$ also map squares to squares, and in fact we have $r(\mcal{N}_{j(\ov{a})})=\mcal{N}_{\ov{a}}$. Therefore, if $\ov{y} \in \ov{Y}$ denotes the projection of $y$, there is a commutative diagram
\begin{equation*}
\begin{tikzcd}
(\mcal{N}_{j(\ov{a})},j(\ov{y})) \arrow[r,"\subset"] \arrow[d,"r"]  & (\mathsf{C}(\ov{Y},\ov{X}),j(\ov{y})) \arrow[d,"r"] \\
(\mcal{N}_{\ov{a}},\ov{y}) \arrow[r,"\subset"] & (\ov{Y},\ov{y})  
\end{tikzcd}
\end{equation*}
At the level of fundamental groups, and after considering the corresponding isomorphisms induced by $\hat{G} \cong \pi_1(\ov{X}, f(\ov{y}))$, we obtain a retraction homomorphism $r_\ast: G' \rightarrow \hat{H}$, where $G'< \hat{G}$ is the subgroup corresponding to $\pi_1(\msf{C}(\ov{Y},\ov{X}), j(\ov{y}))$, and such that $r_\ast(G'_W) \subset \hat{H}_W$. It is not hard to see that the group $\wtilde{G}_W := G'_W \cap r_\ast^{-1}(G'_W)< G'$ satisfies $r_\ast(\wtilde{G}_W) \subset \wtilde{G}_{W}$, so by Proposition \ref{criteriondouble} the double coset $\hat{H}\wtilde{G}_W$ is separable in $G'$, and hence in $G$ since $G'< G$ is of finite index. But $G'_W=G_W \cap G'<G_W$ and $\hat{H}<H$ are of finite index, and hence $\wtilde{G}_W=G'_W \cap r_\ast^{-1}(G'_W)=G'_W \cap r_\ast^{-1}(G'_W \cap\hat{H})< G'_W \cap r_\ast^{-1}(G_W \cap\hat{H})=G'_W$ is also of finite index. We conclude that $HG_W$ is finite union of translates of $\hat{H}\wtilde{G}_W$, so it is also separable in $G$.
\end{proof}
\begin{coro}\label{interconvex-wall}Let $(G,X)$ be a virtually special group, and let $H<G$ be a convex subgroup with convex core $Y \subset X$. Then there exists a finite index subgroup $G'<G$ such that for any wall $W \subset X$ intersecting $Y$:
\begin{enumerate}
    \item If $g \in G'$ satisfies $g\mcal{N}(W) \cap Y \neq  \emptyset$, then in fact $gW \cap Y \neq \emptyset$.
\item If $W'$ is another wall of $X$ intersecting $Y$ and $g \in G'$ is such that $W'=gW$, then in fact $W' \cap Y=h'(W' \cap Y)$ for some $h' \in G' \cap H$.
\end{enumerate}
\end{coro}
\begin{rmk}When we project to the corresponding quotients, conclusion (1) of the corollary above may be thought as a version of \emph{no inter-osculation} for a wall and a locally convex subcomplex of a compact special cube complex.
\end{rmk}
\begin{proof}The proof closely follows the idea of \cite[Lem.~9.14]{HaglundWise2008Special}. Let $W_1,\dots, W_n$ be a complete set of representatives of $H$-orbits of walls intersecting $Y$, and for each $i$ define the sets
$$I(G,Y,i)=\corchete{g\in G \colon gW_i\cap Y \neq \emptyset}, \hspace{2mm} J(G,Y,i)=\corchete{g\in G \colon g\mcal{N}(W_i)\cap Y \neq \emptyset}.$$
These sets are clearly $(H, G_{W_i})$-invariant, and also $I(G, Y, i) \subset J(G, Y, i)$, so there are subsets
$\mcal{I}_i\subset \mcal{J}_i\subset G$ such that $I(G, Y, i)=\bigsqcup_{g\in \mcal{I}_i}{HgG_{W_i}}$ and $J(G, Y, i)=\bigsqcup_{g\in \mcal{J}_i}{HgG_{W_i}}$.

Let us prove first that each of the sets $\mcal{J}_i$ is finite. Fix $1 \leq i \leq n$, and consider finite sets $D_i\subset \mcal{N}(W_i)^{(0)}$ and $E \subset Y^{(0)}$ such that $\mcal{N}(W_i)^{(0)}=G_{W_i}\cdot D$ and $Y^{(0)}=H \cdot E$. Given $g \in G$ such that $g\mcal{N}(W_i) \cap Y$ is non-empty, there is a vertex $v$ of $\mcal{N}(W_i)$ with $gv \in Y$, and so there are group elements $w \in G_{W_i}$ and $h \in H$ satisfying $wv \in D_i$ and $hgv \in E$. In particular, since the action of $G$ on $X$ is proper, the composition $hgw^{-1}$ lies in the finite set $\mcal{F}_i$ of group elements $g' \in G$ such that $g'D_i \cap E \neq \emptyset$, and hence we can choose $\mcal{J}_i \subset \mcal{F}_i$.

Next, note that since by assumption $W_i \cap Y \neq \emptyset$, we have $HG_{W_i} \subset I(G, Y, i)$, so we may assume $1\in \mcal{I}_i$. The finite set $\mcal{J}_i\backslash {1}$ is then disjoint from $HG_{W_i}$ which is separable in $G$ by Lemma \ref{sepconvex-wall}. Therefore, there exists a finite index normal subgroup $\hat{G}_i \trianglelefteq G$ such that $(\bigcup_{g \in \mathcal{J}_{i}\backslash \corchete{1}}{g \hat{G}_{i}}) \cap HG_{W_{i}}=\emptyset$. We claim that (any finite index subgroup of) $\hat{G} :=\bigcap_{i}{\hat{G}_i}$ satisfies conclusion (1).

Indeed, let $W \subset X$ be a wall intersecting $Y$, and let $1 \leq i \leq n$ and $h \in H$ such that $W=hW_i$. Assume by contradiction that there is some $g \in \hat{G}$ such that $g\mcal{N}(W) \cap Y \neq \emptyset$ but $gW \cap Y = \emptyset$.
Since $\hat{G}$ is normal, this implies $h^{-1}gh \in \hat{G} \cap J(G, Y, i)\backslash I(G, Y, i) \subset \hat{G}_i \cap J(G, Y, i) \backslash I(G, Y, i)$,
and hence $h^{-1}gh=vg_iw$, for $v \in H$, $g_i \in \mcal{J}_i \backslash \mcal{I}_i$ and $w \in G_{W_i}$. This is a contradiction, because otherwise we would have $g_i((wh^{-1})g^{-1}(wh^{-1})^{-1})= v^{-1}w^{-1} \in g_i\hat{G}_i \cap HG_{W_i}$, and so conclusion (1) follows.

For conclusion (2), since $(G,X)$ is virtually special we can assume that $X/\hat{G}$ is special, fully clean, and with simplicial 1-skeleton, so that the composition $f : Y/(H\cap\hat{G}) \rightarrow X/(H\cap\hat{G}) \rightarrow X/\hat{G}$ is a local isometry. But $(H, Y)$ is also virtually special, so by using Lemma \ref{finiteclasses} (2) and the separability of wall stabilizers in $H$ we can find a finite index subgroup $H'< H$ such that for any further finite index subgroup $H''  <H'$ and for any wall $W \subset X$ intersecting $Y$, the map $(W\cap Y) /(H''\cap G_W) \rightarrow Y/H''$ is an embedding and the image is a wall of  $Y/H''$.

The group $H'  < G$ is convex, hence separable by Theorem \ref{haglundwiseseparability}, so by a separability argument we may assume that $H' \cap \hat{G}= H \cap \hat{G}$. With this in mind, let $\msf{C}$ be the connected component of $\msf{C}(Y/(H \cap \hat{G}), X/\hat{G})$ including $Y/(H \cap \hat{G})$, and let $G'< \hat{G}$ correspond to its fundamental group, which is finite index in $G$ since $Y/(H \cap \hat{G})$ is compact. Also, we have $H \cap \hat{G}  < G'$, and so $H \cap G'= H \cap\hat{G}$. Since the inclusion of $Y/(H \cap G')$ into $\msf{C}$ maps distinct walls to distinct walls (see Remark \ref{rmkcompletion}), our assumptions about $H'$ imply that the group $G'$ satisfies conclusion (2). 
\end{proof}
The key idea in the proof of Theorem \ref{doublecosetconvexsep} is based on the following proposition, which says that under some mild assumptions, the canonical completion is functorial.
\begin{prop}\label{functoriality}Let $\ov{V} , \ov{X}, \ov{Y} ,\ov{Z}$ be special cube complexes such that the following is a commutative diagram of local isometries.
\begin{equation}\label{eq10}
\begin{tikzcd}
\ov{V} \arrow[r,"f"] \arrow[d,"s"]  & \ov{Y} \arrow[d,"t"] \\
\ov{Z} \arrow[r,"g"] & \ov{X}  
\end{tikzcd}
\end{equation}
In addition, assume that
\begin{itemize}
    \item[($i$)] $\ov{X}$ and $\ov{Y}$ are fully clean and have simplicial 1-skeleton.
    \item[($ii$)] $t$ maps distinct walls to distinct walls.
    \item[($iii$)] If $e$ is an edge of $\ov{X}$ incident to a vertex $t(y)$ of $t(\ov{Y})$ with $e$ dual to a wall intersecting $t(\ov{Y})$, then $e=t(e')$ for some edge $e'$ of $\ov{Y}$ incident to $y$.
    \item[($iv$)] For every vertex $v \in \ov{V}$, if there exist edges $e$ of $\ov{Y}$ and $e'$ of $\ov{Z}$ incident to $f(v)$ and $s(v)$ respectively and such that $t(e)=g(e')$, then there is an edge $e''$ of $\ov{V}$ incident to $v$ and
such that $e=f(e'')$ and $e'= s(e'')$.
\end{itemize}
Then there is a local isometry $\hat{t} : \msf{C}(\ov{V},\ov{Y}) \rightarrow \msf{C}(\ov{Z},\ov{X})$ of the canonical completions commuting with the corresponding inclusions and projections in the sense that the following diagrams commute.
\begin{equation}\label{eq11}
\begin{tikzcd}
\ov{V} \arrow[r,"j"] \arrow[d,"s"]  & \msf{C}(\ov{V},\ov{Y}) \arrow[d,"\hat{t}"] \arrow[r,"r"]& \ov{V} \arrow[d,"s"] \\
\ov{Z} \arrow[r,"j'"] & \msf{C}(\ov{Z},\ov{X}) \arrow[r,"r'"] & \ov{Z} 
\end{tikzcd}\hspace{15mm}
\begin{tikzcd}
\msf{C}(\ov{V},\ov{Y}) \arrow[r,"\hat{t}"] \arrow[d,"p"]  & \msf{C}(\ov{Z},\ov{X}) \arrow[d,"p'"] \\
\ov{Y} \arrow[r,"t"] & \ov{X}  
\end{tikzcd}
\end{equation}
\end{prop}
\begin{rmk}As we will see below, the conditions ($i$) ($ii$) and ($iii$) in the previous proposition can be obtained for a general commutative diagram of local isometries between compact special cube complexes after passing to finite coverings. Condition ($iv$) may be achieved, if for instance, the universal cover of $\ov{V}$ coincides with the intersection of the universal covers of $\ov{Y}$ and $\ov{Z}$, when we see these complexes naturally embedded in the universal cover of $\ov{X}$.
\end{rmk}
\begin{proof}We first construct the map $\hat{t}$ for lower dimensional cubes of $\msf{C}(\ov{V},\ov{Y})$, starting with the 0-skeleton where we define $\hat{t} : \ov{V}^{(0)}\times \ov{Y}^{(0)} \rightarrow \ov{Z}^{(0)}\times \ov{X}^{(0)}$ by $(v, y) \mapsto (s(v), t(y))$. In this way, $\hat{t}$ clearly satisfies \eqref{eq11}. For the 1-skeleton, we will check that the image under $\hat{t}$ of a pair vertices of $\msf{C}(\ov{V},\ov{Y})$ representing a horizontal (resp. diagonal) edge is a pair of vertices representing a horizontal (resp. diagonal) edge of $\msf{C}(\ov{Z},\ov{X})$. Let $e=\corchete{(v, y),(v, y')}$ be a horizontal edge of $\msf{C}(\ov{V},\ov{Y})$, for which we claim that the pair $\corchete{(s(v), t(y)),(s(v), t(y'))}$ represents a horizontal edge. Assume by contradiction that there exists an edge $b$ of $\ov{Z}$ incident to $s(v)$, with $g(b)$ dual to the wall $W(\corchete{t(y), t(y')}) \subset \ov{X}$ (note that $\corchete{t(y), t(y')}$ is an edge since $\ov{X}$ has simplicial 1-skeleton and $t$ is local isometry). This edge $g(b)$ is incident to $g(s(v))=t(f(v))$, so by condition ($iii$), $g(b)$ equals $t(b')$ for an edge $b'$ incident to $f(v)$. Condition ($ii$) then implies that $b'$ is dual to the wall $W(\corchete{y, y'})\subset \ov{Y}$, and condition ($iv$) gives us an edge $b''$ of $\ov{V}$ incident to $v$ with $f(b'')=b'$. But this would imply that $e$ is not horizontal, and this contradiction proves the claim. The case of $e= \corchete{(v, y),(v', y')}$ diagonal is easier since $t$ maps walls to walls, and hence $g(\corchete{s(v), s(v')})$ is dual to $W(\corchete{t(y), t(y')})$. Therefore, the image of a horizontal/diagonal edge of $\msf{C}(\ov{V},\ov{Y})$ is defined as the expected horizontal/diagonal edge of $\msf{C}(\ov{Z},\ov{X})$, and since the image of an edge under $r$ or $r'$ only depends on whether the edge is horizontal or vertical, $\hat{t}$ also satisfies \eqref{eq11} at the level of 1-skeleton.

Now, let $Q$ be a square of $\msf{C}(\ov{V},\ov{Y})$, say with 1-skeleton determined by the vertices $\corchete{(v_i,y_i)}^4_{i=1}$. By definition, this means that $p(Q)$ is also a square with 0-skeleton $\corchete{y_i}^4_{i=1}$, and by condition ($i$) the vertices $\corchete{t(y_i)}^4_{i=1}$ are the 0-skeleton of the square $t(p(Q))$ of $\ov{X}$. Since $\msf{C}(\ov{Z},\ov{X})$ is a covering, these vertices lift under $p'$ to the set $\corchete{(s(v_i), t(y_i))}^4_{i=1}$ that is the 0-skeleton a square $Q'$, that we define as the image of $Q$ under $\hat{t}$. Again, since the image of a square under a retraction only depends on whether its 1-skeleton consists of horizontal/diagonal edges, the diagrams \eqref{eq11} still commute.

Finally, by \cite[Lem.~2.5]{HaglundWise2008Special} there is a unique way to extend $\hat{t}$ to a combinatorial map $\msf{C}(\ov{V},\ov{Y}) \rightarrow \msf{C}(\ov{Z},\ov{X})$, which is clearly a local isometry. Also, since the maps $r, r', p$ and $p'$ restricted to a higher dimensional cube depend only on its 2-skeleton, by uniqueness of $\hat{t}$ it must satisfy \eqref{eq11}.
\end{proof}
\begin{rmk}In the proof of Theorem \ref{doublecosetconvexsep} below, we will be interested in the commutative diagrams of fundamental groups induced by \eqref{eq11}, so we will only require these diagrams to commute at the level of 2-skeletons.
\end{rmk}
\begin{proof}[Proof of Theorem \ref{doublecosetconvexsep}] Let us use Lemma \ref{enlarging} to find convex cores $Y$ and $Z$ for $H$ and $K$ respectively, such that $V= Y \cap Z$ is non-empty. By Lemma \ref{intersectionconvex} this will imply that $V$ is a convex core for $H \cap K$. We will prove first that there exists a finite index subgroup $\hat{G} <  G$ such that if $\hat{H}=H \cap\hat{G}$ and $\hat{K}= K \cap \hat{G}$, then after defining $\ov{V}=V/(\hat{H}\cap\hat{K})$, $\ov{X}=X/\hat{G}$, $\ov{Y}=Y/\hat{H}$ and $\ov{Z}=Z/\hat{K}$, the induced diagram \eqref{eq10} is of local isometries and satisfies the conditions ($i$)-($iv$) of Proposition \ref{functoriality}.

By \cite[Cor.~8.9]{HaglundWise2008Special} we can find $\hat{G}<G$ of finite index such that condition ($i$) holds, and by possibly replacing $\hat{G}$ by a further finite index subgroup satisfying Corollary \ref{interconvex-wall}, we can ensure that $\hat{G}$ also satisfies ($ii$).\\
To prove condition ($iii$), let $e$ be an edge of $\ov{X}$ incident to $t(y)$ for a vertex $y \in \ov{Y}$, and let $\wtilde{e} \subset X$ be a lifting of $e$ incident to the lifting $\wtilde{y} \in Y$ of $y$. If $e$ is dual to the wall $W(t(b)) \subset \ov{X}$ for some edge $b \subset \ov{Y}$, then there exists a lifting $\wtilde{b} \subset Y$ of $b$ and some $g \in \hat{G}$ such that $W(\wtilde{e})=gW(\wtilde{b})$. Since $\wtilde{y}\in Y$, we have $W(\wtilde{b}) \cap Y \neq \emptyset$ and $g\mcal{N} (W(\wtilde{b})) \cap Y \neq \emptyset$, so by conclusion (1) of Corollary \ref{interconvex-wall} we have $W(\wtilde{e}) \cap Y \neq \emptyset$ and $\wtilde{e} \subset Y$, implying condition ($iii$).\\
Finally, let $v \in \ov{V}$ be a vertex lifting to $\wtilde{v} \in V$, and let $\wtilde{e}$ and $\wtilde{e}'$ be edges of $Y$ and $Z$ respectively, incident to $\wtilde{v}$ and such that there exists some $g \in \hat{G}$ with $g\wtilde{e}=\wtilde{e}'$. Since the action of $\hat{G}$ is free, we have $g=1$ and $\wtilde{e}=\wtilde{e}' \in V$. Projecting to the corresponding quotients we deduce ($iv$).

Therefore, we are in the assumptions of Proposition \ref{functoriality}, and there is a local isometry $\hat{t}:\msf{C}(\ov{V},\ov{Y}) \rightarrow \msf{C}(\ov{Z},\ov{X})$ such that the diagrams \eqref{eq11} commute. The proof now goes as in Lemma \ref{sepconvex-wall}. If $H'<\hat{H}$ and $G'<\hat{G}$ are the finite index subgroups representing the fundamental groups of the (appropriate connected components of the) canonical completions $\msf{C}(\ov{V},\ov{Y})$ and $\msf{C}(\ov{Z},\ov{X})$ respectively, then $H'<G'$ and there is a retraction homomorphism $r_\ast : G' \rightarrow G'$ with image $\hat{K}$ and such that $r_\ast(H')=\hat{H} \cap \hat{K}$. Again, we can check that $\wtilde{H}:=H'\cap r_\ast^{-1}(H')$ satisfies $r_\ast(\wtilde{H})\subset \wtilde{H}$, and so Proposition \ref{criteriondouble} implies that $\wtilde{H}\hat{K}$ is separable in $G'$. Since $G' <G$, $\wtilde{H}<H$ and $\hat{K}<K$ are all of finite index, we conclude that $HK$ is separable in $G$, completing the proof.
\end{proof}
We now see how Theorem \ref{doublecosetconvexsep} implies Proposition \ref{independencecore}. In fact, by Lemma \ref{enlarging}, Proposition \ref{independencecore} follows from the next proposition.
\begin{prop}\label{specialimpliesspecial} Let $(G,X)$ be a cubulated group, and let $Y \subset X$ be a $G$-invariant convex subcomplex. If the cubulated group $(G, Y)$ is virtually special, then $(G,X)$ is also virtually special.
\end{prop}
Before proving this result, we first recall the definition of gate map projection \cite[Sec.~2]{BehrstockHagenSisto2017HHS1}.
\begin{defi} Let $X$ be a \CAT{0} cube complex, and consider a convex subcomplex $Y \subset X$. The \emph{gate map} is defined as the unique cubical map $\mfrk{g} : X \rightarrow Y$ characterized by the following property: for any point $x \in X$, the wall $W \subset X$ separates $x$ from $\mfrk{g}(x)$ if and only if it separates
$x$ from $Y$.
\end{defi}
The next lemma is well known by experts and is implicit, for instance, in \cite{BehrstockHagenSisto2017HHS1}, so we provide a proof in the absence of a precise reference.
\begin{lemma}\label{gateconvex}Let $(G,X)$ be a cubulated group and let $Y \subset X$ be a $G$-invariant convex subcomplex. Then for any convex subcomplex $K \subset X$, its image $\mfrk{g}(K) \subset Y$ is also a convex subcomplex. Moreover, if $H<G$ preserves $K$ and acts cocompactly on it, then it also acts cocompactly on $\mfrk{g}(K)$.
\end{lemma}
\begin{proof}For the first assertion, it is enough to prove that if $x, y \in X$ are vertices and $\beta$ is a combinatorial geodesic path joining $\mfrk{g}(x)$ and $\mfrk{g}(y)$, then $\beta=\mfrk{g}(\alpha)$, for some geodesic $\alpha$ joining $x$
and $y$. We will prove this by induction on the sum of combinatorial distances $d= d(x, \mfrk{g}(x))+d(y, \mfrk{g}(y))$, where the case $d=0$ holds since $Y$ is convex. So, assume that the claim follows for $d \geq 0$, and let $x, y \in X$ be vertices with $d(x, \mfrk{g}(x))+d(y, \mfrk{g}(y))=d+1$, for which we can assume
$d(x, \mfrk{g}(x))> 0$. Thus, let $\gamma$ be a combinatorial geodesic joining $x$ and $\mfrk{g}(x)$, and $u$ be the vertex on this geodesic at distance 1 to $x$. Except for the wall dual to the edge $e$ determined by $x$ and $u$, any other wall dual to an edge of $\gamma$ separates $u$ and $\mfrk{g}(x)$, so $\mfrk{g}(u)=\mfrk{g}(x)$, and by our inductive assumption there is a geodesic $\alpha'$ joining $u$ and $y$, such that $\mfrk{g}(\alpha')=\beta$.

If $e$ separates $x$ from $y$, the concatenation of $e$ and $\alpha'$ defines a geodesic $\alpha$ projecting to $\beta$. Otherwise, there is an edge $e'$ of $\alpha'$ dual to $W(e)$, say determined by the vertices $p$ and $q$ with $p$ between $u$ and $q$. In this case, the segment $\alpha''$ of $\alpha'$ between $u$ and $p$ lies in one of the sides of $\mcal{N}(W(e))$ which we know is a convex subcomplex, so every vertex of $\alpha''$ lies in an edge dual to $W(e)$. If we follow the extreme points of these edges lying on the other side of $\mcal{N}(W(e))$, we will obtain a geodesic path joining $x$ and $q$. By concatenating this path with the segment of $\alpha$ between $q$ and $y$, we will obtain a geodesic path $\alpha$ (there is no repetition in the walls dual to $\alpha$), and it is easy to see that $\mfrk{g}(\alpha)=\beta$.

The second assertion follows easily since $Y$ is $G$-invariant, and since $\mfrk{g}(K)/H$ is the image of the compact set $K/H$ under the induced projection $\mfrk{g} : X/G \rightarrow Y/H$.
\end{proof}
\begin{proof}[Proof of Proposition \ref{specialimpliesspecial}] Let $W_1, W_2 \subset X$ be walls with stabilizers $G_1$ and $G_2$, respectively. By Theorem \ref{haglundwisedoublecoset}, to prove the proposition it is enough to show that $G_1$ and $G_1G_2$ are separable in $G$. Consider then the gate map $\mfrk{g} : X \rightarrow Y$ and the projections $\mfrk{g}(\mcal{N}(W_1))$ and $\mfrk{g}(\mcal{N}(W_2))$, which by Lemma \ref{gateconvex} are convex subcomplexes of $Y$. This same lemma also implies that each subgroup $G_i$ acts cocompactly on $\mfrk{g}(\mcal{N}(W_i))$, and so $G_1$ and $G_2$ are convex subgroups of $(G, Y)$, which by assumption is virtually special. The conclusion then follows by Theorems \ref{haglundwiseseparability} and \ref{doublecosetconvexsep}.
\end{proof}

\medskip

\hspace{-4.3mm}\textbf{Acknowledgment}\hspace{2mm} 
I am grateful to my advisor Ian Agol for very valuable discussions throughout all this work. I also thank Sam Shepherd for careful reading and for suggesting the short proof of Lemma \ref{nbhdconvex}, Markus Steenbock for pointing out Corollary \ref{Atiyahconj}, and Chi Cheuk Tsang, Daniel Groves, David Futer, and Ariel Reyes-Pardo for useful comments. I am indebted to the anonymous referee for the detailed report and the suggestions and corrections to the text. This article was supported by the Simons Foundation.

%I thank Jairo Bochi for very interesting and valuable discussions and corrections. I was partially supported by CONICYT PIA ACT172001 during the preparation of this article.

%I am very grateful to J. Bochi for very interesting and valuable discussions throughout all this work. I also thank G. Urz\'ua for valuable discussions about Segre embeddings.

%I am grateful to my advisor J.\+Bochi for very valuable discussions and corrections throughout all this work. I also thank to I.\+D.\+Morris for communicating us the counterexample given in Theorem \ref{counterexample}, and helping in the proofs of Theorem \ref{contsd} and Proposition \ref{ultimojeje}.
%\begin{thebibliography}{200}

%\bibitem{yomt}
%E. Oreg\'on-Reyes, \textit{Negative Curvature, Matrix Products, and Ergodic Theory.} \url{http://www.mat.uc.cl/~jairo.bochi/docs/Oregon-Reyes_master_thesis.pdf}

%\bibitem{yogro}
%E. Oreg\'on-Reyes, Properties of sets of isometries of Gromov hyperbolic spaces. \textit{Groups Geom. Dyn.}, \textbf{12} (2018), no. 3, 889--910.

%\bibitem{sch}
%W. Schlag, Regularity and convergence rate for the Lyapunov exponents of linear cocycles. \textit{Journal of Modern Dynamics}, \textbf{7} (2013), 619--637.

%\bibitem{zh}
%Z. Zhang. Uniform hyperbolicity and its applications to spectral analysis of 1D discrete
%Schr\"odinger operators. \url{http://arxiv.org/abs/1305.4226}, arXiv preprint, 2013.

%Include about good groups and Strong Atiyah conjecture

%\end{thebibliography}
%\bibliographystyle{amsplain}
%\bibliography{references}

\small{Eduardo Oreg\'on-Reyes (\texttt{eoregon@berkeley.edu})}\\
\small{Department of Mathematics}\\
\small{University of California at Berkeley}\\
\small{1087 Evans Hall, Berkeley, CA 94720-3840, U.S.A.}

\end{document}